\documentclass{amsart}
\usepackage[utf8]{inputenc}
\usepackage{amsthm,amsmath,amssymb, amscd}
\usepackage[colorlinks=true,allcolors=blue!80!black]{hyperref}
\usepackage[margin=1.25in]{geometry}
\usepackage[capitalise]{cleveref}
\usepackage[parfill]{parskip} 
\usepackage{graphics}
\usepackage{tikz, tikz-cd}
\usetikzlibrary{matrix, snakes, patterns, shadows.blur, backgrounds, decorations.markings}
\usepackage[all]{xy}
\usepackage{enumerate}
\usepackage{bbm,bm}
\usepackage{stmaryrd}
\usepackage{nicematrix}
\usepackage{dsfont}
\usepackage{caption}
\usepackage{subcaption}
\usepackage{import}
\usepackage[shortlabels]{enumitem}

\theoremstyle{theorem}
\newtheorem{thm}{Theorem}[section]
\newtheorem{prop}[thm]{Proposition}
\newtheorem{lem}[thm]{Lemma}
\newtheorem{cor}[thm]{Corollary}

\theoremstyle{definition}
\newtheorem{defn}[thm]{Definition}

\newtheorem{rem}[thm]{Remark}
\newtheorem{obs}[thm]{Observation}
\newtheorem{ex}[thm]{Example}
\newtheorem{notation}[thm]{Notation}

  \AddToHook{env/prop/begin}{\crefalias{thm}{prop}}
  \crefname{prop}{Proposition}{Propositions}
  \AddToHook{env/lem/begin}{\crefalias{thm}{lem}}
  \crefname{lem}{Lemma}{Lemmas}
  \AddToHook{env/cor/begin}{\crefalias{thm}{cor}}
  \crefname{cor}{Corollary}{Corollaries}
  \AddToHook{env/defn/begin}{\crefalias{thm}{defn}}
  \crefname{defn}{Definition}{Definitions}
  \AddToHook{env/rem/begin}{\crefalias{thm}{rem}}
  \crefname{rem}{Remark}{Remarks}
  \AddToHook{env/ex/begin}{\crefalias{thm}{ex}}
  \crefname{ex}{Example}{Example}
  \AddToHook{env/obs/begin}{\crefalias{thm}{obs}}
  \crefname{obs}{Observation}{Observations}
  \AddToHook{env/conj/begin}{\crefalias{thm}{conj}}
  \crefname{conj}{Conjecture}{Conjectures}
  \AddToHook{env/const/begin}{\crefalias{thm}{const}}
  \crefname{const}{Construction}{Constructions}
  \AddToHook{env/war/begin}{\crefalias{thm}{notation}}
  \crefname{notation}{Notation}{Notations}

\DeclareMathOperator\Diff{Diff}

\DeclareMathOperator\Isom{Isom}
\DeclareMathOperator\Aut{Aut}

\DeclareMathOperator\emb{Emb}
\DeclareMathOperator\Map{Map}

\DeclareMathOperator*\hocolim{hocolim}
\DeclareMathOperator\im{im}
\DeclareMathOperator\Conf{Conf}
\DeclareMathOperator\SO{SO}
\DeclareMathOperator\SU{SU}

\DeclareMathOperator\Or{O}

\DeclareMathOperator\GL{GL}

\DeclareMathOperator\Fr{Fr}
\DeclareMathOperator{\umb}{Sub}

\DeclareMathOperator{\Out}{Out}

\DeclareMathOperator\Stab{Stab}
\DeclareMathOperator\hofib{hofib}

\DeclareMathOperator\Sub{Sub}
\DeclareMathOperator\Sym{Sym}
\DeclareMathOperator{\Coll}{Coll}

\newcommand{\interior}[1]{\smash{\mathring{#1}}}
\newcommand{\scl}[1]{\widehat{#1}} 
\DeclareMathOperator\sep{Sep} 

\newcommand{\sepNP}{\sep^{\nparallel}} 
\newcommand{\ca}{\!\mid\!} 
\newcommand{\SepMfd}{\mathrm{SepMfd}} 

\newcommand{\hq}{/\!\!/} 

\newcommand{\fr}{\mathrm{fr}}

\newcommand{\Gr}{\mathrm{Gr}}
\newcommand{\Top}{\mathrm{Top}}
\newcommand{\FI}{\mathrm{FI}}

\newcommand{\id}{\mathrm{id}}

\newcommand{\BDiff}{B\!\Diff}

\newcommand{\BSO}{B\!\SO}

\newcommand{\calC}{\mathcal{C}}
\newcommand{\calD}{\mathcal{D}}

\newcommand{\calS}{\mathcal{S}}
\newcommand{\calI}{\mathcal{I}}

\newcommand{\calH}{\mathcal{H}}
\newcommand{\calF}{\mathcal{F}}
\newcommand{\bbZ}{\mathbb{Z}}
\newcommand{\bbQ}{\mathbb{Q}}
\newcommand{\bbR}{\mathbb{R}}
\newcommand{\bbC}{\mathbb{C}}

\newcommand{\Z}{\mathbb{Z}}
\newcommand{\W}{\mathbb{W}}
\newcommand{\bt}{\bullet}

\newcommand{\Kx}{K\widetilde{\times}I}
\newcommand{\CP}{\mathbb{CP}}

\newcommand{\rmD}{\mathcal{S}}
\newcommand{\rmC}{\mathcal{H}}
\newcommand{\Cthick}{\underline{\mathcal{H}}}
\newcommand{\cyl}{\mathrm{cyl}}

\newcommand{\tripod}
     {{\vcenter{\hbox{\begin{tikzpicture}
        \draw (0,0) -- ++(0:.1);
        \draw (0,0) -- ++(120:.1);
        \draw (0,0) -- ++(240:.1);
                     \end{tikzpicture}}}}}

\newcommand{\fakeenv}{} 
\newenvironment{restate}[2]  
{
  \renewcommand{\fakeenv}{#2} 
  \theoremstyle{plain}
  \newtheorem*{\fakeenv}{#1~\ref{#2}} 
  \begin{\fakeenv}  
}
{
  \end{\fakeenv}
}

\title[The prime decomposition fibre sequence for moduli spaces of reducible 3-manifolds]
{The prime decomposition fibre sequence for moduli spaces of reducible 3-manifolds}

\author{Rachael Boyd}
\address{School of Mathematics and Statistics, University of Glasgow, Glasgow G12 8QQ, UK}
\email{rachael.boyd@glasgow.ac.uk}
\urladdr{https://www.maths.gla.ac.uk/~rboyd/} 

\author{Corey Bregman}
\address{Department of Mathematics, Tufts University, Medford, MA 02155, USA}
\email{corey.bregman@tufts.edu}
\urladdr{https://sites.google.com/view/cbregman}

\author{Jan Steinebrunner}
\address{Gonville \& Caius College, University of Cambridge, Cambridge, UK}
\email{js2675@cam.ac.uk}
\urladdr{https://www.jan-steinebrunner.com} 


\begin{document}

\begin{abstract}
     We study the moduli space~$\BDiff^+(M)$, for $M$ a 
    reducible, oriented  3-manifold with irreducible prime factors $P_1,\ldots,P_n$. 
    A programme of C\'esar de S\'a--Rourke, Hendriks--Laudenbach, and Hendriks--McCullough studies the homotopy type of $\Diff^+(M)$ in terms of the $\Diff^+(P_i)$.
    Inspired by a delooping proposed by Hatcher,
    we construct a map from $\BDiff^+(M)$ to $\BDiff^+(P_1 \sqcup \dots \sqcup P_n)$, called the \emph{splitting map},
    that yields a \emph{prime decomposition fibre sequence}. 
    The fibre $\mathcal{H}_g(P_1, \dots, P_n)$ is a space of $1$-handle attachments which we describe geometrically as a homotopy colimit of certain configuration spaces on the~$P_i$.
    Firstly, this allows us to show that for $n>0$ the fibre is equivalent to a finite, connected cell complex.
    Secondly, this makes the fibre sequence an effective tool for computations, which we illustrate by computing the rational cohomology ring of~$\BDiff^+\!\left((S^1\times S^2)^{\sharp 2}\right)$.
\end{abstract}

\maketitle
\setcounter{tocdepth}{1} 
\tableofcontents

\section{Introduction}
Throughout this paper, $M$ is a compact, connected, oriented 3-manifold with prime decomposition 
$M=P_1\sharp\cdots \sharp P_n\sharp (S^1\times S^2)^{\sharp g}$,
where the $P_i$ are irreducible and not diffeomorphic to $S^3$. 
We will assume throughout that $M$ has no spherical boundary components, \emph{i.e.}~that no $P_i$ is a disc, see \cref{rem:spherical-boundary}.
Recall that $\Diff^+(M)$ is the topological group of orientation-preserving diffeomorphisms of $M$, equipped with the $C^\infty$~topology.
We will focus on the moduli space $\BDiff^+(M)$, which classifies smooth oriented $M$-bundles.

A programme announced by C\'esar de S\'a and Rourke \cite{CesardeSaRourke}, then completed by Hendriks and Laudenbach \cite{HendriksLaudenbach1984}, and Hendriks and McCullough \cite{HendriksMcCullough}, aimed to study diffeomorphisms of $M$ in terms of its irreducible prime factors $P_i$. 
More precisely, they construct a homotopy fibre sequence
\[
\begin{tikzcd}[column sep = large]
   \Omega C(M) \rar &
   \Diff^+_\partial(M) \rar &
   \prod_{i=1}^n \Diff^+_\partial(P_i \setminus \interior{D^3})
\end{tikzcd}
\]
and moreover show that it splits as a product when $g=0$ \cite[Theorem 2]{HendriksLaudenbach1984}.
Here $C(M)$ is a space of decompositions of $M$ into its prime factors and $\Diff_\partial^+(M) \le \Diff^+(M)$ denotes the subgroup of diffeomorphisms that fix the boundary point-wise.
The fibration in this sequence is only defined up to homotopy and is not a group homomorphism. In particular this  fibre sequence only describes $\Diff^+_\partial(M)$ as a space and not as a group.

Hatcher proposed a ``delooping'' of this fibre sequence, which describes the classifying space $\BDiff^+(M)$ and thus also captures group theoretic information about $\Diff^+(M)$.
This idea appeared in an unfinished draft available on the author's website \cite{HatcherReducible}.
In this paper we use tools developed in our previous work \cite{BoydBregmanSteinebrunner-finiteness} to realise this vision by constructing the \emph{prime decomposition fibre sequence}:
\[
\begin{tikzcd}[column sep = large]
   \mathcal{H}_{g}(P_1, \ldots , P_n) \rar &
   \BDiff^+(M) \rar["{\W}"] &
   \BDiff^+(P_1 \sqcup \dots \sqcup P_n).
\end{tikzcd}
\]
The \emph{splitting map} $\W$ admits a geometric description in terms of cutting $M$ along a collection of spheres and capping them with discs.
Its fibre is a moduli space of $1$-handle attachments to $P_1 \sqcup \dots \sqcup P_n$.
This is analogous to the space $C(M)$ in the work of \cite{CesardeSaRourke, HendriksLaudenbach1984, HendriksMcCullough}, but an advantage of our construction is that we provide an effective way of computing $\mathcal{H}_g(\dots)$ by expressing it as a finite homotopy colimit of spaces that admit explicit geometric descriptions in terms of framed configuration spaces of the $P_i$.
We are able to leverage this construction in two ways:
\begin{enumerate}
    \item\label{it:main-thesis-n>0}
    For $n>0$, we show that $\mathcal{H}_g(P_1\sqcup \dots \sqcup P_n)$ is equivalent to a finite, connected cell complex.
    \item\label{it:main-thesis-n=0}
    For $n=0$, we construct a spectral sequence converging to the rational cohomology ring of $\BDiff^+((S^1 \times S^2)^{\sharp g})$ and we completely compute this cohomology ring when $g = 2$.
\end{enumerate}
This is the first such cohomology computation for a $3$-manifold that goes significantly beyond the case of irreducible $3$-manifolds.

\subsection*{Separating systems and the splitting map}
The finiteness of the fibre in (\ref{it:main-thesis-n>0}) complements our previous work \cite{BoydBregmanSteinebrunner-finiteness}, where we showed that when $M$ has non-empty boundary $\BDiff_\partial (M)$ has the homotopy type of a finite CW-complex, proving a conjecture of Kontsevich for oriented $3$-manifolds with boundary.

A key tool introduced in \cite{BoydBregmanSteinebrunner-finiteness} was a space $\sepNP(M)$ of \emph{separating systems} for $M$. 
Elements of $\sepNP(M)$ are collections of spheres $\Sigma \subset M$ 
such that each component of $M \setminus \Sigma$ is a punctured irreducible prime factor of $M$ or a punctured 3-sphere.
This space forms a topological poset under inclusion and we showed that the geometric realisation of the nerve of this poset is contractible when $M\ncong S^1\times S^2$. 
This gave us a new model for $\BDiff^+(M)$:  
 \begin{equation*}\label{eqn model} 
	\BDiff^+(M)\simeq |N_\bullet(\sepNP(M))|\hq \Diff^+(M).
\end{equation*} 
The starting point for this paper is \cref{sec: the colimit formula} where we rewrite this model to express $\BDiff^+(M)$ as a homotopy colimit of terms of the form $\BDiff^+(M, \Sigma)$, 
where $\Diff^+(M, \Sigma)$ is the group of orientation-preserving diffeomorphisms that preserve $\Sigma$ set-wise.
This homotopy colimit is indexed by a category $\Gr(M)$ of \emph{dual graphs} $G_\Sigma$ to separating systems $\Sigma$, where the vertices of $G_\Sigma$ correspond to components of $M \setminus \Sigma$ and the edges to components of $\Sigma$.
\begin{restate}{Theorem}{thm:hocolim-intro}
    For every $M \ncong S^1 \times S^2$ there is a functor $\rmD\colon \Gr(M) \to \Top$ such that 
    \[
        \BDiff^+(M)\simeq \hocolim_{\Gamma \in \Gr(M)} \rmD(\Gamma).
    \]
    If $\Gamma = G_\Sigma$ is the dual graph of a sphere system $\Sigma \subset M$, then $\rmD(\Gamma)$ is homotopy equivalent to the classifying space of the group of those orientation-preserving diffeomorphisms $\varphi\in \Diff^+(M)$ such that for all spheres $S \subseteq \Sigma$ we have $\varphi(S) = S$ and $\varphi_{|S}$ is orientation preserving.
\end{restate}

Giansiracusa has given a similar description for the moduli spaces of $3$-dimensional handlebodies, writing $\BDiff^+((S^1 \times D^2)^{\natural g})$ (for $g \ge 2$) as a homotopy colimit over a finite category of graphs \cite{Giansiracusa11}.
(Note that the handlebody is irreducible and hence admits only the empty separating system, so our theorem is vacuous in this case.
Another point of difference is that the moduli spaces Giansiracusa considers are all aspherical.) 
The perspective taken in \cite{Giansiracusa11} is that the moduli spaces of handlebodies form a modular operad, which is the ``modular envelope'' of the framed little $2$-disc operad.
An analogous interpretation of \cref{thm:hocolim-intro} in terms of modular ($\infty$-)operads will appear in upcoming work of Barkan and the third author, as sketched in \cite[Proposition 4.16]{steinebrunner-CPH-lecture-notes}.

\begin{figure}[ht]
    \centering
    \def\svgwidth{.7\linewidth}
    \import{}{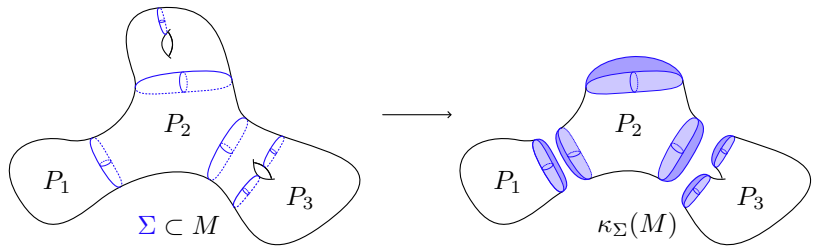}
    \caption{The cut-and-cap map $\kappa_\Sigma$.}
    \label{fig:cutting-and-capping}
\end{figure}

Given a fixed separating system $\Sigma \subset M$, we construct a cut-and-cap map
\[
    \kappa_\Sigma\colon \BDiff^+(M, \Sigma) \longrightarrow \BDiff^+(P_1 \sqcup \dots \sqcup P_n)
\]
that is roughly given by cutting $M$ along $\Sigma$, then capping in the punctures in $M \setminus \Sigma$ with $3$-discs, and finally discarding all connected components that are $3$-spheres, see \cref{fig:cutting-and-capping}.
In \cref{sec:wrongway map}, we assemble these maps into a map out of the homotopy colimit model for $\BDiff^+(M)$.
This yields the \emph{splitting map}
\[
    \W\colon \BDiff^+(M) \longrightarrow \BDiff^+(P_1\sqcup \dots \sqcup P_n)
\]
such that for each separating system $\Sigma \subset M$ the restriction of $\W$ to $\BDiff^+(M, \Sigma)$ gives the cut-and-cap map $\kappa_\Sigma$ described above.
The geometric description of $\kappa_\Sigma$ has the following consequence.
\begin{restate}{Corollary}{cor:W-pi1-surjective}
    The splitting map induces a surjection
    \[
        \pi_1\W\colon \pi_0(\Diff^+(M)) \twoheadrightarrow \pi_0(\Diff^+(P_1 \sqcup \dots \sqcup P_n))
    \]
    from the mapping class group of $M$ onto the mapping class group of $P_1 \sqcup \dots \sqcup P_n$.
\end{restate}

For every compact submanifold $A \subset \partial M$ let $\Diff_A^+(M) \le \Diff^+(M)$ be the subgroup of those diffeomorphisms that fix $A$ pointwise.
We also construct a variant of the splitting map
\[
    \W_A \colon \BDiff_A^+(M) \longrightarrow \BDiff_A^+(P_1 \sqcup \dots \sqcup P_n),
\]
compatible with $\W$, such that the homotopy fibres of $\W_A$ are equivalent to those of $\W$.

\subsection*{Describing the fibre}
Conceptually, the fibre of $\W$ is the space of possible $0$- and $1$-handle attachments one can make to $P_1 \sqcup \dots \sqcup P_n$ to build a manifold diffeomorphic to $M$.
In order to make the map $\W$ an effective tool for computations, we need to make this description more concrete.

Let $\Gr_{g,n}$ denote the category whose objects are pairs $(\Gamma,\sigma)$ of a finite, connected, graph $\Gamma$ with first Betti number $g$ and an $n$-tuple $(\sigma_1,\dots, \sigma_n)$ of distinct marked vertices such that every non-marked vertex is at least trivalent.
Morphisms in this category are generated by graph isomorphisms and edge-contractions, see \cref{defn:Gr_gn}.

Such a tuple $(\Gamma,\sigma)$ is the first half of a blueprint for building a manifold diffeomorphic to $M$ from the disjoint union $P_1 \sqcup \dots \sqcup P_n$.
(The irreducible component $P_i$ will correspond to the marked vertex $\sigma_i$.) 
We begin by adding a disjoint $S^3$ for each unmarked vertex, and then we attach a $1$-handle for each edge in the graph.
The second half of the data, needed to perform these handle-attachments topologically, is a choice of disjoint embeddings of discs into $P_1 \sqcup \dots \sqcup P_n$, 
or equivalently a choice of framed configurations.
We define
\[
    \Conf^{\fr+}_{(\Gamma,\sigma)}(P_1,\ldots,P_n):=\prod_{i=1}^n \Conf_{H_{\sigma_i}}^{\fr+}(P_i) \times \prod_{v \in V_\Gamma^{\rm sph}} \Conf_{H_v}^{\fr+}(S^3),
\]
where $\Conf^{\fr+}_A(N)$ denotes the space of positively framed configurations of a finite set $A$ in $N$, and for $v\in V_\Gamma$, $H_v$ is the set of half-edges at $v$, and $V_\Gamma^{\rm sph} = V_\Gamma \setminus \{\sigma_1,\dots, \sigma_n\}$ is the set of unmarked vertices in $\Gamma$. 
This configuration space comes with an action of the Lie group $ \smash{\prod_{E_\Gamma}\SO(3)\times\prod_{V_{\Gamma}^{\rm sph}}\SO(4)}$ by re-choosing framings along each edge and rotating the $S^3$ components. 
The homotopy quotient of $\Conf_{(\Gamma,\sigma)}^{\fr+}(P_1,\ldots,P_n)$ by this Lie group roughly parametrises the moduli space of all possible 0- and 1-handle attachments in the pattern determined by $(\Gamma,\sigma)$.

To relate this heuristic picture to the homotopy colimit in \cref{thm:hocolim-intro}, we introduce a space $\rmC^{\Gamma}_\sigma(P_1,\ldots,P_n)$ of separating systems augmented by an embedding $\iota\colon P(\Sigma\subset M)\hookrightarrow P_1\sqcup\cdots\sqcup P_n$, where $P(\Sigma\subset M)$ is a manifold called the \emph{maximal prime piece}, see \cref{defn - maximal prime piece}. Abstractly, $P(\Sigma\subset M)$ is diffeomorphic to $P_1\sqcup\cdots \sqcup P_n$ minus a union of open discs, and crucially, it is constructed so that the pair $(P(\Sigma\subset M),\iota)$ vary continuously with $\Sigma$ and functorially with respect to inclusions of sphere systems.
We then rewrite the fibre $\rmC_g(P_1,\ldots,P_n)$ as a homotopy colimit over $\Gr_{g,n}$ of the $\rmC_{\sigma}^\Gamma(P_1,\ldots,P_n)$ and establish the \emph{prime decomposition fibre sequence}:

\begin{restate}{Theorem}{thm:fibre-sequence-intro}
    For every compact, connected, oriented 3-manifold $M$ that is not diffeomorphic to~$S^1\times S^2$, there is a homotopy fibre sequence
    \[
    \hocolim_{(\Gamma, \sigma) \in \Gr_{g,n}} \rmC_\sigma^\Gamma(P_1, \dots, P_n) \longrightarrow \BDiff^+(M) \xrightarrow[\qquad]{\W} \BDiff^+(P_1 \sqcup \dots \sqcup P_n)
    \]
    where for each $(\Gamma, \sigma) \in \Gr_{g,n}$ there is a homotopy equivalence
    \[
        \rmC_\sigma^\Gamma(P_1,\dots,P_n) \simeq \Conf^{\fr+}_{(\Gamma,\sigma)}(P_1,\dots,P_n) \sslash
        \bigg(\prod_{E_\Gamma}\SO(3)\times\prod_{V_{\Gamma}^{\rm sph}}\SO(4)\bigg).
    \]
\end{restate}

The fibre sequence is particularly simple when $g=n=1$, in which case the homotopy colimit reduces to the homotopy colimit over the group 
$\Aut(\vcenter{\hbox{\begin{tikzpicture} \draw[fill] (0,0) circle (.05); \draw (0.13,0) circle (0.13); \end{tikzpicture}}}) = \bbZ/2$.
This yields a description of the moduli space of an irreducible $3$-manifold $P$ with one $1$-handle attached to it.

\begin{ex} (See \cref{ex:n=g=1} for more details.)
    For every irreducible $3$-manifold $P$ that is not $S^3$ or $D^3$, there is a homotopy fibre sequence
    \[
        \frac{\Conf_2(P) \times \SO(3)}{\bbZ/2} \longrightarrow
        \BDiff^+(P \sharp  (S^1 \times S^2)) \longrightarrow
        \BDiff^+(P)
    \]
    where $\bbZ/2$ acts by swapping the configuration points and inverting the element of $\SO(3)$. 
    We can for example use this to determine the classifying space of the identity component $\Diff_0(P \sharp (S^1\times S^2))$ of the diffeomorphism group.
    Suppose, for simplicity, that $P$ is aspherical and not diffeomorphic to $S^1\times D^2$.
    Then $\BDiff_0(P\sharp (S^1\times S^2))$ is homotopy equivalent to a product of $S^3$ and a wedge sum of infinitely many copies of $S^2$, indexed by elements of $\pi_1(P).$ 
    Therefore 
    \[ 
        \pi_1\Diff(P\sharp(S^1\times S^2))=\pi_2\BDiff(P\sharp(S^1\times S^2)) = \oplus_{i=1}^\infty \Z
    \]
    has infinite rank, but we show that, as a module over the integral group ring of the mapping class group $\pi_0\Diff(P\sharp(S^1\times S^2))$, it is generated by one element.
\end{ex}

Kalliongis and McCullough \cite{KalliongisMcCullough} showed that $\pi_1 \Diff(M)$ has infinite rank when $g+n\ge 3$ or when $g=n=1$ and $\pi_1(P)$ is infinite.
The above example precisely determines this infinite rank group in the $g=n=1$ case and shows that it is still finitely generated as a module over the integral group ring of the mapping class group.

In \cite{Hatcher1981,Hatcher1981revised} Hatcher describes $\Diff(M)$ for $M = S^1 \times S^2$ and $P_1 \sharp P_2$ (as before the $P_i$ are irreducible) by using that in these cases any two essential $2$-spheres are isotopic.
The above example deals with the natural next case: in $P \sharp (S^1 \times S^2)$ it is not true that any two essential spheres are isotopic, but any two \emph{non-separating} spheres are isotopic.

If the prime decomposition  
$M=P_1\sharp  \cdots \sharp P_n \sharp (S^1\times S^2)^{\sharp g}$
satisfies $n>0$ and $g+n\geq 2$, 
then the Lie group action in \cref{thm:fibre-sequence-intro} is free.
We use this to show that $\rmC_\sigma^\Gamma(P_1,\dots,P_n)$ is homotopy equivalent to a compact manifold with corners of dimension at most $15(g-1)+12n$. Using this upper bound in the homotopy colimit yields an upper bound for the dimension of the fibre.

\begin{restate}{Corollary}{cor:fibre-finite-and-dimension-bound}
    For $n>0$ and $n+g>1$, the fibre $\calH_g(P_1,\dots,P_n)$
    of the splitting map $\W$ is homotopy equivalent to a finite cell complex of dimension at most $17(g-1)+13n$.
\end{restate}

Combining \cref{cor:fibre-finite-and-dimension-bound} with work of Gabai \cite{Gabai01} yields the following example. 
\begin{ex}
    (See \cref{ex:hyperbolic}.)
    Let $n>0$ and $g \ge 0$ such that $g+n\ge 2$.
    For $P_1,\dots,P_n$ hyperbolic $3$-manifolds, $\BDiff^+(M)$
    admits a finite cover that is equivalent to a finite cell complex of dimension at most $17(g-1) + 13n$.
    However, $\BDiff^+(M)$ itself need not be finite-dimensional.
\end{ex}

\subsection*{Rational cohomology computations}
When studying the moduli space $\BDiff^+(M)$, an invariant of particular interest is its rational cohomology ring.
Each smooth oriented $M$-bundle $\pi\colon E \to B$ is classified by some map $f_\pi\colon B \to \BDiff^+(M)$, and hence each cohomology class $\alpha \in H^k(\BDiff^+(M); \bbQ)$, defines a \emph{characteristic class} $\alpha(\pi) := f_\pi^*\alpha$ of smooth $M$-bundles.

In the case of surfaces $M = \Sigma_g$ the Madsen--Weiss theorem \cite{MadsenWeiss} computes this cohomology ring in the ``stable range'', and shows that it is a polynomial algebra on Miller--Morita--Mumford classes, as conjectured by Mumford.
In even dimensions $d = 2n \ge 6$, work of Galatius--Randal-Williams \cite{GalatiusRandalWilliams2017} generalises the Madsen--Weiss theorem and describes the rational cohomology ring of $\BDiff^+( (S^n \times S^n)^{\sharp g} )$ in a stable range. Similar results have been obtained in high odd dimensions $d = 2n+1 \ge 7$ \cite{BotvinnikPerlmutter2017, LF-triads-25}.

While these results describe each cohomology ring up to a certain degree (depending on the genus), they do not describe the complete cohomology ring for any given genus.
For $g\le 4$, work of Looijenga and Tommasi 
\cite{Looijenga-M3-93,Tommasi-M4-2007}
uses algebro-geometric tools to compute the cohomology ring of the moduli stack $\mathcal{M}_g$ of smooth curves of genus $g$, which is equivalent to $\BDiff^+(\Sigma_g)$.

For $3$-manifolds, one can make concrete computations of the rational cohomology ring of $\BDiff^+(M)$ when $M$ is irreducible.
For example, if $M$ is hyperbolic, then $\Diff^+(M)$ is equivalent to the (finite) group of isometries and hence the rational cohomology of the moduli space is trivial.
In the case of the connected sum of two lens spaces Lelkes in \cite{Lelkes-lens-spaces} used the aforementioned theorem of Hatcher \cite{Hatcher1981} to compute the rational cohomology ring of $\BDiff^+(M)$.

For more general $3$-manifolds, and in particular the family of connected sums $U_g = (S^1 \times S^2)^{\sharp g}$, which are analogous to the manifolds considered in higher-dimensional variants of the Madsen--Weiss theorem,
little is known about the rational cohomology of $\BDiff^+(M)$.
The prime decomposition fibre sequence from \cref{thm:fibre-sequence-intro} presents an opportunity to make such computations in dimension $3$ tractable.
Concretely, there is a Bousfield--Kan spectral sequence (see \cref{cor:BK-spectral-sequence}) converging to the cohomology of the homotopy colimit:
    \begin{equation*}\label{eq:intro-BK-spectral-sequence}
        E_2^{p,q} = \operatornamewithlimits{lim^{\it p}}_{\Gamma \in \Gr_{g,n}^{\rm op}} H^q(\rmC^\Gamma_\sigma(P_1,\dots,P_n); \bbQ) \Rightarrow H^{p+q}( \calH_g(P_1,\dots, P_n) ).
    \end{equation*}
Here the $E_2$-page consists of the derived limits of the cohomology groups of $\rmC^\Gamma_\sigma(\dots)$ as functors on $\Gr_{g,n}^{\rm op}$.
In \cref{sec: computation} we use the description of $\rmC^\Gamma_\sigma(P_1,\dots,P_n)$ from \cref{thm:fibre-sequence-intro} to compute its rational cohomology ring in terms of the rational cohomology of configuration spaces of the $P_i$, see \cref{thm:DGamma Cohomology Ring}.
This description is $\Aut(\Gamma, \sigma)$ and $\prod_i \Diff^+(P_i)$-equivariant (the latter action factors through $\prod_i \pi_0\Diff^+(P_i)$), but we do not fully describe its functoriality in $\Gr_{g,n}^{\rm op}$.

In the case of $U_g$, where the base of the prime decomposition fibre sequence is trivial, we give a complete description of $H^*(\rmC^\Gamma; \bbQ)$ as a functor of $\Gamma \in \Gr_{g,0}$.
This allows us to completely determine the rational cohomology ring of the moduli space for $g=2$.

\begin{restate}{Theorem}{thm:H*BDiffU2}
The rational cohomology ring of $\BDiff^+\!\left((S^1\times S^2)^{\sharp2}\right)$ has presentation
\[
    H^*\left(\BDiff^+\!\left((S^1\times S^2)^{\sharp2}\right); \bbQ\right)\cong \bbQ[\gamma_1,\gamma_2, \varepsilon]/\langle\varepsilon^2, \gamma_1\gamma_2, \gamma_1 \varepsilon\rangle
\]
where $|\gamma_1|=|\gamma_2|=4$ and $|\varepsilon|=8$.
\end{restate}

To illustrate the geometric meaning of these cohomology classes, in \cref{subsec:circle-action} we construct smooth circle actions $\xi_1,\xi_2\colon \SO(2) \to \Diff^+((S^1\times S^2)^{\sharp 2})$ that distinguish $\gamma_1$ and $\gamma_2$ in the sense that $(B\xi_i)^*\gamma_j = e^2$ for $i\neq j$ and $(B\xi_i)^*\gamma_i = 0$, where $e$ is the Euler class in $H^*(\BSO(2); \bbQ) = \bbQ[e]$.
This in particular yields $7$-manifolds $E_1$ and $E_2$ fibred over $\CP^2$ with fibre $U_2$ such that the characteristic classes $\gamma_1$ and $\gamma_2$ evaluate as
\begin{align*}
    \gamma_1(E_2 \to \CP^2) &= \gamma_2(E_1 \to \CP^2) = e^2 &
    \gamma_1(E_1 \to \CP^2) &= \gamma_2(E_2 \to \CP^2) = 0.
\end{align*}

For higher $g$, one can still use the spectral sequence to try to compute the cohomology ring of $H^*(\BDiff(U_g); \bbQ)$, but already determining the $E_2$-page is a non-trivial task:
the $0$th row $E_2^{*,0}$ is the rational group cohomology of $\Out(F_g)$.

\subsection*{Outline}
In \cref{sec:setting up}, we record some preliminary and background results about group actions, 3-manifolds, separating systems, homotopy colimits, and categories of combinatorial graphs.
In \cref{sec: the colimit formula}, we consider dual graphs to separating systems and prove \cref{thm:hocolim-intro}, establishing a homotopy colimit formula for $\BDiff^+(M)$.
In \cref{sec:wrongway map}, we define the ``maximal prime piece'' of a separating system and use it to construct the splitting map $\W$. 
Here we also give a geometric description of $\W$ and introduce variants that fix part of the boundary.
\cref{sec:describing the fibre} describes the fibre in more geometric terms:
we first state \cref{thm:fibre-sequence-intro}, consider several applications, and use it to deduce finiteness of the fibre for $n>0$.
Then, we introduce thickened sphere systems in order to prove \cref{thm:fibre-sequence-intro} and set up a spectral sequence that describes the cohomology of $\rmC^\Gamma_\sigma(P_1,\dots,P_n)$.
In \cref{sec: computation}, we compute the rational cohomology of $\Conf_d(S^3)$ and $\Conf_d(S^3)\sslash\SO(4)$, which we then use to calculate the cohomology of $\rmC^\Gamma_\sigma(P_1,\dots,P_n)$ via the aforementioned spectral sequence, and also determine the effect of graph morphisms on several of the resulting classes.
Finally, in \cref{sec:extended example}, we apply these tools to compute the rational cohomology of $\BDiff^+((S^1\times S^2)^{\sharp2})$ and evaluate the resulting characteristic classes on two bundles obtained from circle actions.

\subsection*{Acknowledgements}

The first author was supported by EPSRC Fellowship No.~EP/V043323/1 and No.~EP/V043323/2. 
We would like to thank Oscar Randal-Williams for several useful conversations over the course of the project and for his comments on a draft of this paper.
\section{Background on separating systems, homotopy colimits, and graphs}\label{sec:setting up}
The purpose of this section is to collect some necessary background results and notation to be used throughout the paper. We will assume some familiarity with foundational results due to Cerf \cite{Cerf} and Palais \cite{Palais60} concerning embedding spaces of submanifolds. For more details, we refer the reader to Section 2 of \cite{BoydBregmanSteinebrunner-finiteness} and Section 2 of \cite{CMRW17}. After reviewing some results about homotopy fibre sequences and group actions in \ref{subsec:locally-retractile}, we recall the notion of a separating system in \ref{subsec:3-mfd}, which we introduced in \cite{BoydBregmanSteinebrunner-finiteness}. Our main technical result from \cite{BoydBregmanSteinebrunner-finiteness} established that the nerve of the poset $\sepNP(M)$ of separating systems has contractible geometric realisation, which is recorded here as \cref{thm: sepNP contractible}. In \cref{subsec:extend-D3}, we recall Smale and Hatcher's theorems concerning the homotopy type of $\Diff^+(D^3)$ and deduce some consequences for general 3-manifolds. \cref{subsec:background-on-hocolims} reviews the homotopy colimit construction and proves a sufficient condition for a homotopy colimit to be finite. Lastly, in \cref{subsec:general-graphs} we introduce the graph categories that we will work with and prove some of their basic properties.

\subsection{Fibre sequences and $G$-locally retractile spaces}
\label{subsec:locally-retractile}
We begin by recalling some of the tools we use for working with topological group actions and for constructing fibrations.

\begin{defn}\label{defn:G-locally retractile}
    A space $X$ equipped with an action of a topological group $G$ is called \emph{$G$-locally retractile} if for every $x\in X$ there exists a neighbourhood $U$ of $x$ and a map $\xi\colon U\rightarrow G$ such that  $\xi(u).x = u$ for all $u\in U$. We call $\xi$ a \emph{$G$-local retraction at $x$.}
\end{defn}

The crucial feature of $G$-locally retractile spaces is that every $G$-equivariant map into them is automatically a fibre bundle.
(This elementary, but extremely useful observation is due to Cerf, see e.g.~\cite[Lemma 2.5]{CMRW17} for a proof.)
As a special case of this we have the following.

\begin{lem}\label{lem:Retractile Covering}
    Let $X$ be a locally path connected $G$-locally retractile space and let $p\colon Y\to X$ be a $G$-equivariant map with discrete fibres.  Then $p$ is a covering map and $Y$ is $G$-locally retractile.
\end{lem}
\begin{proof}
By \cite[Lemma 2.5]{CMRW17}, $p$ is a locally trivial fibre bundle, hence a covering map since the fibres of $p$ are discrete. Fix $y\in Y$ and let $x=p(y)$. We can find path-connected open sets $U\ni y$, $V\ni x$ such that $p|_U\colon U\cong V$ and there exists a $G$-local retraction $\xi\colon V\rightarrow G$ satisfying $\xi(x)=1$. Define $\tau=\xi\circ p|_U\colon U\rightarrow G$ and let $f=(p|_U)^{-1}$. We claim that $\tau$ is a $G$-local retraction at $y.$ 
Given $u\in U$, choose a path $\gamma_t\colon [0,1]\to V$ such that $\gamma_0=x$ and $\gamma_1=p(u)$. Using that $\xi(x)=1$, we obtain two paths in $U$ starting at $y$ defined by $\gamma_t'=\xi(\gamma_t).y$ and $\gamma''_t=f(\xi(\gamma_t).x)$.  Since $p$ is $G$-equivariant, both $\gamma'$ and $\gamma''$ cover $\gamma$, hence must be equal by uniqueness of path-lifting. It follows that \[u=f(\gamma_1)=\gamma_1''=\gamma_1'=\xi(\gamma_1).y=\xi(p(u)).y=\tau(u).y,\]hence $\tau$ is a $G$-local retraction at $y\in Y$, as required.
\end{proof}

\begin{defn}
    A $G$-action on a space $X$ \emph{principal} if the quotient map $X \to X/G$ is a principal $G$ bundle.
    We call a topological subgroup $H \le G$ \emph{admissible} if $H$ acts principally on $G$, \emph{i.e.}~if $G \to G/H$ is a fibre bundle.
\end{defn}

\begin{obs}\label{obs:admissible}
We have the following general constructions for principal actions and admissible subgroups.
\begin{enumerate}[(1)]
    \item\label{it:adm-stabliser} $G$ acts on $X$ in a locally retractile way, then the inclusion of the stabilizer $\Stab_x(G) \le G$ is an admissible subgroup inclusion for all $x \in X$.
    \item If $G$ acts principally on $X$, then it also acts principally on $X \times Y$ for any $G$-space $Y$.
    \item\label{it:adm-quotient}
    By \cite[Lemma 2.9]{BoydBregmanSteinebrunner-finiteness}, if $H \le G$ is admissible and $G$ acts principally on $X$, then so does $H$.
\end{enumerate}
\end{obs}

\begin{ex}\label{ex:Principal-action-examples}
    We give a few examples of principal actions and admissible subgroups.
    Let $M$ be a compact manifold and $N \subset M$ a compact submanifold.
    See \cite[\S 2.2]{BoydBregmanSteinebrunner-finiteness} for references.
\begin{enumerate}
    \item The action of $\Diff(N)$ on $\emb(N, \interior{M})$ is principal.
    Its quotient is the space $\Sub(N, \interior{M})$ of submanifolds of $\interior{M}$ that are abstractly diffeomorphic to $N$. The action of $\Diff(M)$ on $\Sub(N, \interior{M})$ is locally retractile.
    \item The subgroup $\Diff(M, N) \le \Diff(M)$, which is the stabiliser of the action of $\Diff(M)$ on $\Sub(N, \interior{M})$, is admissible. 
    \item The action of $\Diff(N)$ on $\emb(N, \bbR^\infty)$ is principal.
\end{enumerate}
\end{ex}

Throughout the paper we will use the homotopy orbit construction $X\sslash G$, which we define as $(X \times EG)/G$ for $EG$ a weakly contractible principal $G$-space.
We refer the reader to \cite[\S2.3]{BoydBregmanSteinebrunner-finiteness} for a recollection of the basic properties of this construction, but we especially highlight the following fact, which can be found in \cite[Corollary 2.14]{bonatto2023}.
\begin{lem}\label{lucis lemma}
    Consider group actions $G_i \curvearrowright X_i$ and equivariant maps between them inducing maps
    \[
        X_1 \sslash G_1 \longrightarrow X_2 \sslash G_2 \longrightarrow X_3 \sslash G_3.
    \]
    This is a homotopy fibre sequence, if the $X_i$ and the $G_i$ each form a homotopy fibre sequence and $G_2 \to G_3$ is surjective on path components.
\end{lem}

\subsection{Review of 3-manifolds and separating systems}\label{subsec:3-mfd}
We will always assume that our manifolds are smooth, compact, and oriented. Recall that a 3-manifold $N$ is called \emph{irreducible} every embedded $S^2\subset N$ bounds $D^3$, and \emph{reducible} otherwise.  Irreducible 3-manifolds are \emph{prime} in the sense that any decomposition $N=N_1\sharp N_2$ implies that either $N_1$ or $N_2$ is homeomorphic to $S^3$. The only prime but not irreducible orientable 3-manifold is $S^1\times S^2$. It is a classical result \cite{Kneser1929,Milnor62} that orientable 3-manifolds admit  connected sum decompositions into essentially unique irreducible factors and copies of $S^1\times S^2$:

\begin{thm}
    Any compact, oriented 3-manifold $N$ admits a minimal connected sum decomposition\[N=P_1\sharp \cdots \sharp P_n\sharp (S^1\times S^2)^{\sharp g}\sharp (D^3)^{\sharp m}\] 
    satisfying:
    \begin{enumerate}
        \item The $P_i$ are oriented, irreducible and not diffeomorphic to $D^3$ nor $S^3$.
        \item The factors are uniquely determined by $N$, up to reordering and orientation-preserving diffeomorphism. 
    \end{enumerate}
\end{thm}
\begin{rem}
    The number $g$ of $S^1\times S^2$ factors is the maximal number of disjointly embedded 2-spheres in $N$ whose union does not separate. The number $m$ of $D^3$ factors is the same as the number of $S^2$-components of $\partial N$. In particular, when $N$ has no spherical boundary components then $m=0$. We regard the case $N=S^3$ as having no factors: $n=g=m=0$. 
\end{rem}

If $\Sigma\subset \interior{N}$ is an embedded, codimension 1 submanifold, then $\interior{N}\setminus \Sigma$ is homeomorphic to the interior of a compact manifold with boundary which we denote $N\ca \Sigma$. When $\Sigma$ is a disjoint union of 2-sided submanifolds, \emph{e.g.}~when $\Sigma=\sqcup_kS^2,$ then the boundary of $N\ca \Sigma$ contains two copies of $\Sigma$, one for the left and right pushoff, which  we denote by $2\Sigma\subset \partial(N\ca\Sigma).$ There is a canonical map $N\ca \Sigma\rightarrow N$ which is a homeomorphism away from $\Sigma$ and a double cover along $\Sigma$. If $N$ has $S^2$ boundary components, the \emph{spherical closure of $N$}, denoted $\scl{N}$, is the manifold obtained by capping all $S^2$ boundary components with copies of $D^3$.

\begin{rem}\label{rem:spherical-boundary}
    All of our main theorems concern $3$-manifolds (smooth, compact, connected, oriented) \emph{without spherical boundary}.
    If $N$ is a $3$-manifold with $m$ spherical boundary components, then we can write $N = \scl{N} \setminus \amalg_m \interior{D^3}$.
    As in \cite[Lemma 6.3]{BoydBregmanSteinebrunner-finiteness} there is a homotopy fibre sequence
    \[\begin{tikzcd}
        {\frac{\Conf_m(\scl{N})}{\Sym_m}} \rar &
        {\BDiff^+(N)} \rar &
        {\BDiff^+(\scl{N})} .
    \end{tikzcd}\]
    The fibre, which is the unordered configuration space of $m$ points in $\scl{N}$, is fairly well-understood -- for example, it always admits a compactification as a compact $3m$-manifold with corners \cite{Sinha04}, hence is equivalent to a finite CW complex.
    Therefore, this fibre sequence allows us to reduce most qualitative or computational questions about $\BDiff(N)$ to the case without spherical boundary.
\end{rem}

In \cite{BoydBregmanSteinebrunner-finiteness}, we introduced the notion of a separating system, which we now recall.  

\begin{defn}
    An (unparametrised) embedded 2-sphere $S\subset N$ is \emph{reducing} if it does not bound a ball in the spherical closure $\scl{N}$.
    (Equivalently, if each component of $\smash{\scl{N\ca S}}$ has fewer prime factors than $\scl{N}$.)
    A submanifold $\Sigma\cong \sqcup_k S^2\subset N$ is called a \emph{separating system} if every component of $\Sigma$ is reducing, and if $\smash{\scl{N\ca \Sigma}}$ is irreducible. 
\end{defn}

The set of separating systems forms a topological poset $\sep(N)$ with respect to inclusion, topologised as a subspace \[\sep(N)\subset \coprod_{k=0}^\infty\umb(\sqcup_k S^2,N).\]
With this topology, a connected component of $\sep(N)$ corresponds to an isotopy class of separating system in $N$. Since $\sep(N)$ is a disjoint union of $\Diff(N)$-locally retractile spaces, namely $\amalg_{k=0}^\infty \umb(\sqcup_kS^2,N)$, it is itself $\Diff(N)$-locally retractile.  
In particular,  for any fixed $\Sigma\in \sep(N)$, the orbit map $\Diff(N)\rightarrow \sep(N)$ defined by $\varphi \mapsto \varphi(\Sigma)$ is a principal $\Diff(M,\Sigma)$-bundle.  
\begin{defn}
    Two disjointly embedded 2-spheres $S_0,S_1\subset N$ are called \emph{parallel} if they cobound a submanifold of $N$ diffeomorphic to $S^2\times [0,1]$. Let $\sepNP(N)\subset \sep(N)$ denote the subposet of all separating systems $\Sigma$ where no two components of $\Sigma$ are parallel. 
\end{defn}

It follows from results of Laudenbach \cite{Laudenbach} and Perelman's solution to the Poincar\'e conjecture \cite{Perelman:2002-1,Perelman:2003-1,Perelman:2003-2} that two disjoint 2-spheres are parallel if and only if they are (unparametrised) isotopic (see \cite[Lemma 3.14]{BoydBregmanSteinebrunner-finiteness}). Thus, an element of $\sepNP(N)$ may equivalently be defined as a separating system whose components lie in pairwise distinct isotopy classes. 

\begin{prop}\label{prop:sepNP-Properties}The poset $\sepNP(N)$ enjoys the following properties:
\begin{enumerate}
    \item As a union of $\Diff(N)$-orbits of $\sep(N)$, it is $\Diff(N)$-locally retractile.
    \item  By \cite[Lemma 6.5]{BoydBregmanSteinebrunner-finiteness}, it has finite height.
    \item By \cite[Lemma 6.6]{BoydBregmanSteinebrunner-finiteness},  it is a union of finitely many $\Diff(N)$ orbits.
\end{enumerate}
\end{prop}

\begin{rem}\label{rem:Finite-Index-Retractile}
   Since finite index subgroups are admissible  \cref{obs:admissible}(3) implies that Properties (1) and (3) of \cref{prop:sepNP-Properties} also hold after passing to finite-index subgroups, such as $\Diff^+(N)$. 
\end{rem}

Let $\sepNP_k(N)=\{\Sigma_0\subseteq \cdots\subseteq \Sigma_k\mid \Sigma_i\in \sepNP(N)\}$, be the space of $k$-chains of $\sepNP(N)$.
This is a simplicial space, namely exactly the nerve of the topological poset $\sep(M)$. 
We let $\|\sepNP_\bt(N)\|$ denote its fat geometric realisation. 
(The fat and thin realisation of $\sepNP_\bullet(N)$ are homotopy equivalent, see e.g.~the proof of \cite[Theorem 6.2]{BoydBregmanSteinebrunner-finiteness}.)
In \cite{BoydBregmanSteinebrunner-finiteness}, we prove the following.

\begin{thm}[Theorem 3.20, \cite{BoydBregmanSteinebrunner-finiteness}]\label{thm: sepNP contractible}
    If $N\not \cong S^1\times S^2$, then $\|\sepNP_\bt(N)\|$ is contractible.
\end{thm}
Since the $\Diff(N)$-action on $\sepNP(N)$ preserves inclusions, it acts component-wise on every level $\smash{\sepNP_k(N)}$, and hence on the fat geometric realisation $\smash{\|\sepNP_\bt(N)\|}$. By taking homotopy orbits, \cref{thm: sepNP contractible} implies the following following fundamental equivalence:
\begin{equation}
    \BDiff(N)\simeq \|\sepNP_\bt(N)\|\hq\Diff(N).
\end{equation}

\subsection{Extending diffeomorphisms over 3-discs}\label{subsec:extend-D3}
At the heart of our construction of $\W$ in \cref{thm: wrong way map} are the following two foundational results:

\begin{thm}\label{thm:Smale-Hatcher}Let $D^3$ be a closed 3-ball with boundary $\partial D^3=S^2$.
\begin{enumerate}[(i)]
    \item \emph{(Smale, \cite{Smale59})} The inclusion $\SO(3)\hookrightarrow \Diff^+(S^2)$ is a homotopy equivalence.
    \item \emph{(Hatcher, \cite{Hatcher})} The inclusion $\SO(3)\hookrightarrow \Diff^+(D^3)$ is a homotopy equivalence.
\end{enumerate}
\end{thm}
The first part implies that every family of diffeomorphisms of $S^2$ extends over $D^3$, while the second part (which is equivalent to the Smale conjecture) says that this extension is unique up to homotopy, even when varying in families. Equivalently, the second part states that $\Diff_\partial(D^3)$ is contractible.

\begin{defn}\label{defn: equal-after-scl-embeddings}
    Given compact 3-manifolds $N,Q$, let $\emb^{\scl{=}}(N, Q)$ denote the subspace of those embeddings $i\colon N\hookrightarrow Q$ for which there exists a diffeomorphism $\scl{N}\cong Q$ restricting to $i$ on $N$. 
\end{defn}

The notation ``$\scl{=}$" is meant to indicate ``diffeomorphism after taking spherical closures." We apply the above results of Smale and Hatcher to obtain canonical fillings of $2$-sphere boundary components.

\begin{lem}\label{lem:embedding complement of discs}
    Let $N$ and $Q$ be compact $3$-manifolds, and assume that $\scl{N}$ is abstractly diffeomorphic to $Q$.
    Then for each embedding $\iota \colon N \hookrightarrow Q$ such that the complement of $\iota(N)$ is a disjoint union of discs, the map
    \[
        \Diff^+(Q) \longrightarrow \emb^{\scl{=}}(N, Q), \qquad
        \varphi \mapsto (\varphi \circ \iota).
    \]
    is a homotopy equivalence.
    Moreover, the homotopy orbit space 
    \[
        \emb^{\scl{=}}(N,Q) \hq \Diff^+(Q)
    \]
    is contractible.
\end{lem}
\begin{proof}
 By Smale's theorem, every diffeomorphism of $\partial D^3$ extends over $D^3$, so in particular $\iota$ extends to a diffeomorphism $
 \scl{\iota}\colon \widehat{N}\cong Q$ and $\iota\in \emb^{\scl{=}}(N,Q)$. Given $f\in \emb^{\scl{=}}(N,Q)$, there exists a diffeomorphism $\scl{f}\colon \widehat{N}\cong Q$ restricting to $f$ on $N$ and therefore $\varphi =\scl{f}\circ \scl{\iota~}^{-1}$ is a diffeomorphism of $Q$ such that $\varphi\circ \iota=f$. Therefore, the orbit map $\Diff^+(Q)\rightarrow \emb^{\scl{=}}(N,Q)$ is onto and by \cref{ex:Principal-action-examples}(1), there is a fibre sequence
    \begin{equation}\label{eqn:disc-stabilisers}
      \Diff_{N}(Q) \to \Diff^+(Q) \to \emb^{\scl{=}}(N, Q).  
    \end{equation}
On the other hand, $\Diff_{N}(Q)\simeq \prod_{i=1}^k\Diff_\partial(D^3)$ which is contractible by \cite{Hatcher}. 
Thus $\Diff(Q)\rightarrow \emb^{\scl{=}}(N, Q)$ is a homotopy equivalence.
Lastly, the fibre sequence in \cref{eqn:disc-stabilisers} is equivariant with respect to the exact sequence of groups $\{1\}\rightarrow \Diff^+(Q)\rightarrow \Diff^+(Q)$, hence  \cref{lucis lemma} tells us that $\emb^{\scl{=}}(N, Q)\hq \Diff^+(Q)\simeq \BDiff^+_{N}(Q) $, which is contractible. 
\end{proof}

\begin{rem}
   Although we will not need it, one can show that if $\widehat N\cong Q$, then the complement of any embedding of $N$ into $Q$ is diffeomorphic to a union of discs.
\end{rem}

\subsection{Background on homotopy colimits}\label{subsec:background-on-hocolims}
Our main theorem \cref{thm: wrong way map} describes the fibre of the splitting map as a homotopy colimit over a certain finite category of graphs.
In this subsection we recall some of the relevant definitions and constructions, and establish a criterion for showing when such homotopy colimits are homotopy finite.

\begin{defn}[Homotopy colimit indexed by $\calC$]\label{defn - hocolim}
For a reference, see \cite[Section 4]{Dugger2008}.
Let $\calC$ be  a small category. Given a diagram of topological spaces
\[
    X \colon \calC \to \Top
\]
we construct a simplicial space $E_\bullet$ by
\[
    E_n := \coprod_{C_0 \to \dots \to C_n} X(C_0)
\]
where the coproduct runs over all sequences of $n$ composable morphisms in $\calC$ and we denote the morphisms by $f_i:C_{i-1}\to C_i$ for $1\leq i \leq n$. The first face map is given by 
\[
   \coprod_{C_0 \to \dots \to C_n} X(C_0) \,\, \overset{X(f_1)}{\longrightarrow}\,\,\coprod_{C_1 \to \dots \to C_n} X(C_1) 
\]
and the remaining face maps are given by mapping the copy of~$X(C_0)$ in~$E_n$ indexed by $C_0 \to \dots \to C_n$ to the copy of~$X(C_0)$ in~$E_{n-1}$ indexed by the chain $C_0 \to \cdots \hat{C}_j\cdots \to C_n$ with~$C_j$ removed. Degeneracy maps are given by adding an identity morphism~$C_j\to C_j$ to the indexing chain.
The homotopy colimit of $X$ is now defined as the geometric realisation
\[
    \hocolim_{C \in \calC} X(C) := |E_\bullet|.
\]
\end{defn}

Recall that a category $\calI$ is an EI-category if every endomorphism in $\calI$ is invertible, \emph{i.e.}~it is an automorphism.
(Equivalently, $\calI$ admits a conservative functor to a poset, i.e.~a functor $F\colon \calI \to P$ such that a morphism in $\calI$ is an isomorphism if and only if its image in $P$ is an isomorphism.)
For such EI-categories we can compute homotopy colimits in two steps: first we take a homotopy quotient for finite group actions and then a homotopy colimit over finite posets.

\begin{defn}
    Given an EI-category $\calI$ and $n\ge 0$ we let $P_n(\calI)$ denote the groupoid of conservative functors $[n] \to \calI$: 
    its objects are $n$-tuples of composable morphisms $(f_1,\dots, f_n)$ (where none of the $f_i$ are isomorphisms) and its morphisms are natural isomorphisms
    \[\begin{tikzcd}
        i_0 \rar["{f_1}"] \dar["{\alpha_0}", "\cong"'] & i_1 \rar["{f_2}"] \dar["{\alpha_1}", "\cong"'] & \dots \rar["{f_n}"] & i_n \dar["{\alpha_n}", "\cong"'] \\
        j_0 \rar["{f_1}"] & j_1 \rar["{g_2}"] & \dots \rar["{g_n}"] & j_n .
    \end{tikzcd}\]
    We let $\pi_0 P(\calI) := \coprod_{n\ge 0} \pi_0 P_n(\calI)$ denote the poset of isomorphism classes of such tuples for all $n$, ordered by refinement.
\end{defn}

\begin{thm}[\cite{Slominska89}]\label{thm:Slominska}
    The homotopy colimit of an EI-diagram $X\colon \calI \to \Top$ can be calculated in two steps as
    \[
        \hocolim_{i \in \calI} X(i) 
        \simeq \hocolim_{[f_1,\dots,f_n] \in \pi_0P(\calI)} \left( 
        X(i_0) \sslash \Aut(i_0 \xrightarrow{f_1} \dots \xrightarrow{f_n} i_n)
        \right)
    \]
\end{thm}

\begin{ex}\label{ex:hocolim-one-morphism-EI}
    Suppose that $\calI$ is an EI-category with two objects $a, b \in \calI$ such that there is a morphism $f\colon a \to b$.
    Suppose further that every other morphism $g\colon a \to b$ is of the form $\beta \circ f \circ \alpha$ for $\alpha \in \Aut(a)$ and $\beta \in \Aut(b)$.
    Then S\l ominska's formula says that for every diagram $X\colon \calI \to \Top$ there is a homotopy pushout square
    \[\begin{tikzcd}
        {X(a) \sslash \Aut(f)} \rar \dar \ar[dr, very near end, phantom, "\ulcorner"] & {X(a) \sslash \Aut(a)} \dar \\
        {X(b) \sslash \Aut(b)} \rar & {\hocolim_{i \in \calI} X(i)}
    \end{tikzcd}\]
    where the top map is induced by the group homomorphism $\Aut(f) \to \Aut(a)$ and the left map by the group homomorphism $\Aut(f) \to \Aut(b)$ as well as the map $X(a) \to X(b)$.
\end{ex}

\begin{defn}
    We say that and EI-category $\calI$ is \emph{finite} if it has finitely many objects and finitely many morphisms.
    In this case, we define the depth $D\ge 0$ of $\calI$ to be the length of the longest chain of composable non-isomorphisms in $\calI$.
\end{defn}

Note that the depth is bounded by the total number of morphisms in $\calI$, and in particular it is always finite.
Indeed, suppose we had a chain of composable morphisms $(f_1,\dots,f_n)$ that is longer than the number of morphisms in $\calI$.
Then we must have $f_k = f_l$ for $k<l$, which implies that $f_l \circ \dots \circ f_{k+1}$ and $f_{l-1} \circ \dots \circ f_k$ both are endomorphism and thus invertible, which in turn implies that $f_k = f_l$ is an isomorphism.

\begin{rem}
    A finite EI-category $\calI$ is not necessarily finite as an $\infty$-category: for example, every finite group defines a finite EI-category with one object, but this is only finite as an $\infty$-category if the group is trivial.
    In particular, it is not generally true that the homotopy colimit of a diagram $\calI \to \Top$ valued in finite CW complexes must again be equivalent to a finite cell complex.
    Take for instance the diagram to be constant at the point, in which case the homotopy colimit is the classifying space of $\calI$.
\end{rem}

Nevertheless, we have the following finiteness result for homotopy colimits over certain EI-categories, if we require that the diagram is such that the homotopy colimit over each automorphism group is homotopy finite.
Recall that a morphism $f\colon i \to j$ is an epimorphism if for all $g_1,g_2\colon j \to k$ we have that $g_1 \circ f = g_2 \circ f$ implies $g_1 = g_2$ -- this will be true for all morphisms in our graph categories, see \cref{rem:Gr-epi}.

\begin{prop}\label{prop:EI-finite}
    Let $\calI$ be a finite EI-category such that every morphism in $\calI$ is an epimorphism. 
    Let $X\colon \calI \to \Top$ be a functor for which each $X(i) \hq \Aut_\calI(i)$ is homotopy finite, then $\hocolim_{i \in \calI} X(i)$ is homotopy finite.
    
    Moreover, if $D$ is the depth of $\calI$ and $N$ is such that each $X(i) \hq \Aut_\calI(i)$ is equivalent to a finite CW-complex of dimension at most $N$, then $\hocolim_{i \in \calI} X(i)$ is  homotopy equivalent to a finite CW-complex of dimension at most $N+D$.
\end{prop}
\begin{proof}
    First consider the case where $\calI$ is a finite poset of depth $D$. 
    In this case, we construct a semi-simplicial space $E_\bullet^{\neq}$ that is a subspace of the simplicial space $E_\bullet$ described in \cref{defn - hocolim} by only taking the disjoint union over those tuples of morphisms where none of them is the identity.
    The simplicial space $E_\bullet$ is then obtained by freely adjoining degeneracies to $E_\bullet^{\neq}$ and as such we have a homeomorphism $|E_\bullet| \cong \|E_\bullet^{\neq}\|$.
    On the other hand, the assumption that $\calI$ is a finite poset of depth $D$ and that each $X(i)$ is homotopy finite implies that $E_n^{\neq}$ is equivalent to a finite CW-complex of dimension at most $N$ for all $n$ and is empty for $n>D$.
    Thus $\hocolim_{i \in \calI} X(i) \cong \|E_\bullet^{\neq}\|$ is equivalent to a finite CW complex of dimension at most $N+D$.

    For the general case we use S\l ominska's formula as recalled in \cref{thm:Slominska}.
    We begin by arguing that each of the homotopy quotients $X(i_0) \sslash \Aut(f_1,\dots,f_n)$ is still homotopy finite.
    There is a group homomorphism 
    \[
        \theta\colon \Aut(i_0 \xrightarrow{f_1} \dots \xrightarrow{f_n} i_n) \longrightarrow \Aut(i_0)
    \]
    defined by evaluating the automorphism on the first object.
    This homomorphism is injective: if $\alpha = (\id_{i_0}, \alpha_1, \dots, \alpha_n)$ is in the kernel, then because $f_1$ is an epimorphism the equation $\alpha_1 \circ f_1 = f_1 \circ \id_{i_0}$ implies that $\alpha_1 = \id_{i_1}$ and inductively we see that $\alpha$ is trivial.
    Comparing homotopy orbits with respect to the source and target of $\theta$ thus yields a finite covering map
    \[
        X(i_0) \sslash \Aut(i_0 \xrightarrow{f_1} \dots \xrightarrow{f_n} i_n) 
        \longrightarrow X(i_0) \sslash \Aut(i_0)
    \]
    and as we assumed that the base is homotopy finite it follows that the total space is as well, and in fact we can choose the finite model to also be $N$-dimensional.

    For each $[f_1,\dots,f_n] \in \pi_0 P(\calI)$ we have that $n \le D$ is bounded by the depth, and since we assumed that $\calI$ has only finitely many morphisms, it follows that $\pi_0 P(\calI)$ is finite.
    In fact, $\pi_0 P(\calI)$ has the same depth as $\calI$.
    We have thus written $\hocolim_{i\in \calI} F(i)$ as the homotopy colimit over a finite poset (of depth $D$) of a diagram all of whose values are homotopy equivalent to a finite CW complex of dimension at most $N$ -- this is exactly the case we dealt with at the start of the proof.
\end{proof}

\subsection{Graph categories}\label{subsec:general-graphs}
We now introduce the combinatorial categories of graphs that index the homotopy colimits we will consider later.
Inspired by \cite{Serre}, we define graphs in terms of vertices and half-edges.
In this definition, an edges are defined as the unions of two half edges that are paired via the half-edge involution~$i$.
Although it is more convenient for us to define graphs in terms of half-edges, we will also rely on standard terminology from graph theory, such as paths, cycles, trees, etcetera.
\begin{defn}\label{defn: graph}
    A \emph{graph} $\Gamma$ is a quadruple $(V, H, r, i)$ where $V$ is a finite set of vertices, $H$ is a finite set of half-edges, $i\colon H \to H$ is a fixed point free involution, called the \emph{half-edge involution}, and $r\colon H \to V$ is a map, called the \emph{root map}.
    We define the set of edges of $\Gamma$ as 
    \[
        E_\Gamma = \{ \{h, i(h)\} \;|\; h \in H\} .
    \]
    A \emph{labelled graph} is a graph~$\Gamma$ with a map $l\colon V_\Gamma \to L$ for some set of labels~$L$.
    For a vertex $v \in V$, we let $H_v := r^{-1}(v)$ denote the set of half-edges incident at $v$. The number of elements of $H_v$ is the \emph{valence} of $v$.
\end{defn}

Our notion of graph morphism is inspired by \cite[\S2.2]{Chan2021}.

\begin{defn}
    A morphism of graphs $f\colon \Gamma_1 \to \Gamma_2$ is a map 
    $f\colon V_1 \sqcup H_1 \to V_2 \sqcup H_2$ such that
    \begin{enumerate}
        \item $f$ sends vertices to vertices: $f(V_1) \subseteq V_2$,
        \item for all $h \in H_1$, either $f(h) \in H_2$ and $f(i(h)) = i(f(h))$ or $f(h) \in V_2$ and $f(h) = f(r(h))$.
        \item for each half-edge $h \in H_2$ the preimage $f^{-1}(h)$ is a single half-edge,
        \item for every vertex $w \in V_2$ the preimage $f^{-1}(w)$ forms a tree.
    \end{enumerate}
    Let $\Gr$ denote the category of finite graphs.
\end{defn}

A morphism of graphs $f\colon \Gamma_1 \to \Gamma_2$ is geometrically visualised by  collapsing a collection of subtrees of~$\Gamma_1$ separately to points followed by an isomorphism to~$\Gamma_2$.

\begin{rem}\label{rem:Gr-epi}
    A morphism of graphs $\Gamma_1 \to \Gamma_2$ is always given by a surjective map $X_1 \twoheadrightarrow X_2$ and thus every morphism in $\Gr$ is an epimorphism.
\end{rem}

We also introduce a smaller category of graphs of fixed rank, equipped with a partial marking of the vertices.

\begin{defn}\label{defn:Gr_gn}
  An object of $\Gr_{g,n}$ is a pair $(\Gamma, \sigma)$, where $\Gamma$ is a connected rank $g$ graph and  $\sigma \colon \{1,\dots,n\} \hookrightarrow V_\Gamma$ is an (injective) marking of a subset of the vertices $V_\Gamma$ such that every non-marked vertex is at least trivalent. A morphism $f\colon (\Gamma_1,\sigma_1)\rightarrow (\Gamma_2,\sigma_2)$ is a $\Gr$-morphism $f\colon \Gamma_1\rightarrow \Gamma_2$ such that  $\sigma_2(i)=f(\sigma_1(i))$ for all $i\in \{1,\ldots,n\}$.
\end{defn}

Geometrically, for a morphism $f\colon (\Gamma_1,\sigma_1)\to (\Gamma_2,\sigma_2)$, the pre-image of $\sigma_2(i)$ is a tree containing $\sigma_1(i)$ as its only marked vertex.  A graph $(\Gamma,\sigma)\in \Gr_{g,n}$ is \emph{maximal} if it is never the target of non-invertible morphism. Similarly, a graph $(\Delta,\tau)\in \Gr_{g,n}$ is \emph{minimal} if it is never the source of a non-invertible morphism. The following lemma enumerates some properties of $\Gr_{g,n}$ that we will need in the sequel.
\begin{lem}[Properties of $\Gr_{g,n}$]\label{lem:depth-of-Gr_gn}
    The category $\Gr_{g,n}$ enjoys the following properties:
    \begin{enumerate}[(i)]
        \item The number of edges is $3(g-1)+2n$ in any maximal graph and $(g-1)+n$ in any minimal graph. The longest chain of composable non-invertible morphisms in $\Gr_{g,n}$ has length at most $2(g-1)+n$.
        \item $\Gr_{g,n}$ has only finitely many isomorphism classes of objects. 
        \item $\Gr_{g,n}$ is a connected category, that is, any two objects can be joined by a zig-zag of morphisms.
    \end{enumerate}
\end{lem}
\begin{proof}
    If $(\Gamma,\sigma)\in\Gr_{g,n}$ is maximal, then each labelled vertex must be a leaf and each non-leaf vertex must have valence exactly 3. Otherwise, we could increase the number of edges by at least 1. Since only labelled vertices are allowed to be leaves, the number of leaves is exactly $n$. The number of edges  in any maximal graph $\Gamma$ can then be determined by the Euler characteristic and the relation $3v+n=2e$, where $v$ is the number of non-leaf (and hence non-labelled) vertices. Since $\Gamma$ has rank $g$, we obtain \[1-g=\chi(\Gamma)=(v+n)-e=\frac{2}{3}e-\frac{1}{3}n+n-e=\frac{2}{3}n- \frac{1}{3}e\]
    Hence, the maximal number of edges is $3(g-1)+2n$, independent of $\Gamma$. 
    
    By contrast, a minimal graph $(\Delta,\tau)\in\Gr_{g,n}$ cannot contain any non-loop edges where at least one endpoint is unlabelled.  Therefore there are exactly $n$ vertices (all of which are labelled) and any maximal tree in $\Delta$ must  have $n-1$ edges. Since $\Delta$ has rank $g$, the remaining edges consist of $g$ loops. Thus, the minimal number of edges is $(g-1)+n$, and is also independent of $\Delta$. 

    Since the number of edges (and hence vertices) of any $(\Gamma,\sigma)\in \Gr_{g,n}$ is bounded and there are only finitely many possible labellings of the vertices of such a graph by $\{1,\ldots, n\}$, we conclude that there are only finitely many isomorphism classes of objects in $\Gr_{g,n}$. As each non-invertible morphism collapses at least one edge, the longest possible chain of composable non-invertible morphisms starts from a maximal graph and ends in a minimal graph, and each morphism only collapses one edge.  The number of morphisms in any longest chain is therefore the difference in the number of edges between a maximal and a minimal graph: $2(g-1)+n$. This proves (i) and (ii).

    For (iii), let $(\Gamma_0,\sigma_0)\in \Gr_{g,n}$ be the unique labelled graph consisting of a wedge of $g$ circles and $n$ intervals, where the wedge point is unlabelled, and the $n$ leaves are labelled by $\{1,\ldots, n\}$. We claim there is a zig-zag of morphisms from an arbitrary graph $(\Gamma,\sigma)$ to $(\Gamma_0,\sigma_0)$. Given $(\Gamma,\sigma)\in \Gr_{g,n}$,  we first insert edges where exactly one endpoint is unlabelled to obtain a graph $(\Gamma',\sigma')$ such that each labelled vertex is a leaf. In particular, every internal  vertex of $\Gamma'$ is unlabelled and no vertex labelled by $i\in\{1,\ldots, n\}$ separates. Let $\Delta'\subset \Gamma'$ be the (connected) subgraph obtained by deleting all leaves and their incident edges, and let   $T'\subset \Delta'$ be a maximal tree. Since $T'$ only contains edges connecting unlabeled vertices, we can collapse $T'$ to obtain $(\Gamma_0,\sigma_0)$ as a quotient. Thus we have constructed a zig-zag\[(\Gamma,\sigma)\leftarrow(\Gamma',\sigma')\rightarrow(\Gamma_0,\sigma_0)\] from the arbitrary graph $(\Gamma,\sigma)$ to the fixed graph $(\Gamma_0,\sigma_0)$, as desired.  
\end{proof}

\begin{defn}\label{def:Gr_gn-redundant}
    An edge in $(\Gamma,\sigma)$ is \emph{redundant} if it splits the graph into two components, one of which does not contain labelled vertices.
    Let $\Gr_{g,n}^{\rm nr} \subset \Gr_{g,n}$ denote the full subcategory on the graphs with no redundant edges.
\end{defn}

\begin{lem}\label{lem:No-Redudant-Htpy-Final}
    The inclusion $\Gr_{g,n}^{\rm nr} \hookrightarrow \Gr_{g,n}$ admits a right adjoint.
    In particular it is homotopy final, so any homotopy colimit over $\Gr_{g,n}$ is equivalent to the homotopy colimit over the smaller category.
\end{lem}
\begin{proof}
    Let $\rm J\colon \Gr_{g,n}^{\rm nr} \hookrightarrow \Gr_{g,n}$ be the inclusion.
    Given $(\Gamma,\sigma)\in \Gr_{g,n}$, let $T^{\rm red}(\Gamma)\subset\Gamma$ be the set of all redundant edges. By definition,  $T^{\rm red}(\Gamma)$ is a collapsible forest. Moreover,  if $f\colon (\Gamma,\sigma)\rightarrow (\Gamma',\sigma')$ is a morphism, then $f(T^{\rm red}(\Gamma))=T^{\rm red}(\Gamma')$ since if $\Gamma'\setminus e'=C_1'\sqcup C_2'$ where $C_2'$ only contains   unlabelled vertices, then $\Gamma\setminus f^{-1}(e)=f^{-1}(C_1')\sqcup f^{-1}(C_2')$ and $f^{-1}(C_2')$ also only contains unlabelled vertices. 
    Thus the map $\Gamma\mapsto \Gamma/T^{\rm red}(\Gamma)$ collapsing each component of $T^{\rm red}(\Gamma)$ separately to a point defines a functor $\rm R\colon \Gr_{g,n}\rightarrow\Gr_{g,n}^{\rm nr}$ along with a morphism $\nu_{(\Gamma,\sigma)} \colon (\Gamma,\sigma)\rightarrow \rm J R(\Gamma,\sigma)$ for every $(\Gamma,\sigma)\in \Gr_{g,n}$.
    This defines a right adjoint to $\rm J$ with unit $\nu$ and counit the identity.
    In particular, for every $(\Gamma,\sigma)$ the category $(\Gamma,\sigma)/\rm J$ has an initial object, namely $\rm R(\Gamma,\sigma)$, and hence is contractible, which proves that the inclusion is homotopy final.
\end{proof}

\section{The colimit formula for moduli spaces of \texorpdfstring{$3$}{3}-manifolds}\label{sec: the colimit formula}
Let $M$ be a compact, connected, orientable, non-prime $3$-manifold without $S^2$-boundary. The goal of this section is to write $\BDiff^+(M)$ as a homotopy colimit over a (finite) category of dual graphs $\Gr(M)$.

\subsection{Dual graphs to separating systems}\label{sec: dual graphs}
In this section we introduce the dual graph of a separating system.
These dual graphs will be labelled by oriented diffeomorphism classes of irreducible $3$-manifolds, that is, they are objects in the following category.
\begin{defn}\label{defn:Gr-mfd}
    We let $\Gr^{\rm mfd}$ denote the category of $3$-manifold graphs.
    Its objects are graphs $\Gamma$ with a labelling $l$ that assigns to each vertex an oriented diffeomorphism class of irreducible $3$-manifolds such that if a vertex $v$ is labelled by $l(v) = [S^3]$, then $v$ has valence at least $3$.
    A morphism $(\Gamma,l) \to (\Gamma',l')$ in $\Gr^{\rm mfd}$ is a morphism of graphs $f\colon \Gamma \to \Gamma'$ such that for every vertex $v \in V_{\Gamma'}$ its label is 
    \[
        l'(v) = \sharp_{w \in V_\Gamma \cap f^{-1}(v)} l(w) ,
    \]
    \emph{i.e.}~the connected sum of the labels of the vertices that get collapsed to $v$ (recall that $f^{-1}(v)$ is a tree).
\end{defn}

\begin{rem}
    Determining a morphism~$f\colon\Gamma_1 \to \Gamma_2$ in~$\Gr^{\rm mfd}$ is equivalent to choosing a forest $F\subset \Gamma_1$ such that each subtree of $F$ has at most one vertex without an~$S^3$ label, and a label-preserving graph isomorphism from the graph~$\Gamma_1/F\to \Gamma_2$ where~$\Gamma_1/F$ is obtained by collapsing each tree in~$F\subset\Gamma_1$ and labelling the resultant vertex with the unique non-$S^3$ label, or with~$S^3$ otherwise.
    The condition that each tree of $F$ has at most one non-$S^3$-label is necessary to ensure that $\Gamma_2$ is still labelled by irreducible manifolds.
\end{rem}

\begin{rem}\label{remark: oreinted diffeomorphism classes} Let $-N$ denote the manifold $N$ endowed with the reverse orientation. Then if $N$ admits an orientation reversing diffeomorphism it follows that its oriented diffeomorphism class $[N]$ satisfies $[N]=[-N]$. If $N$ does not admit an orientation reversing diffeomorphism ($N$ is said to be \emph{chiral}) then $[N]$ and $[-N]$ are distinct labels.
\end{rem}

    Let~$\Gamma \in \Gr^{\rm mfd}$ be a $3$-manifold graph.
    Pick a manifold $M_v$ in the diffeomorphism class $l(v)$ for each vertex $v$, and consider the $3$-manifold obtained by taking the disjoint union $\coprod_{v \in V} M_v$ and then attaching a $1$-handle for each edge in the graph. 
    The resulting manifold $M_\Gamma$ has prime decomposition:
    \[
        M_\Gamma \cong (\sharp_{v \in V} M_v) \sharp (S^1 \times S^2)^{\sharp b_1}
    \]
    where $b_1$ is the first Betti number of $\Gamma$.
    If $\Gamma \to \Lambda$ is a morphism in $\Gr^{\rm mfd}$ then $M_\Gamma$ and $N_\Gamma$ must have the same prime decomposition and hence they are (oriented) diffeomorphic.
\begin{defn}
    Let~$\Gr(M) \subset \Gr^{\rm mfd}$ be the full subcategory of those $3$-manifold graphs $(\Gamma, l)$ such that $M_\Gamma$ is (oriented) diffeomorphic to $M$.
    We refer to $\Gamma\in\Gr(M)$ as an \emph{$M$-graph}.
\end{defn}
    We therefore have a disjoint decomposition
    \[
        \Gr^{\rm mfd} = \coprod_{[M]} \Gr(M)
    \]
   where the coproduct runs over the oriented diffeomorphism classes of $3$-manifolds.

\begin{defn} \label{defn: dual graph}
    If $\Sigma \subset M$ is a separating system with no parallel spheres ($\Sigma\in \sepNP(M)$), then its \emph{dual graph}~$G_{\Sigma\subset M}$ is the $3$-manifold graph defined as follows:
    \begin{itemize}
        \item $V_\Sigma$ is the set of connected components of the compact manifold~$M\ca\Sigma$. Each vertex $K\in \pi_0(M\ca \Sigma)$ is labelled by the oriented diffeomorphism class $[\scl{K}]$, \emph{i.e.}~the oriented irreducible manifold corresponding to the spherical closure of $K$, up to orientation-preserving diffeomorphism. (Recall there are no spheres in $\partial M$ so taking the spherical closure is equivalent to filling in spheres of $2\Sigma$.) 
        \item $H_\Sigma$ is the set of spheres in~$\Sigma$ with a choice of coorientation. This means there are two half edges for each sphere~$S\subset \Sigma$, and we pair these using the half edge involution~$i$.
        \item Each half edge $h \in H_\Sigma$ has as root $r(h)$ the component of $M \ca \Sigma$ with boundary the sphere associated to~$h$, such that the coorientation is inward-pointing with respect to the connected component.
        An example is shown in \cref{fig:dual graphs}.
    \end{itemize}
\end{defn}

\begin{rem}
When the ambient manifold $M$ containing $\Sigma$ is clear from the context, we will sometimes just write $G_\Sigma$ to denote the dual graph.
\end{rem}

\begin{figure}[h!]
    \centering
    \resizebox{.9\linewidth}{!}{\import{}{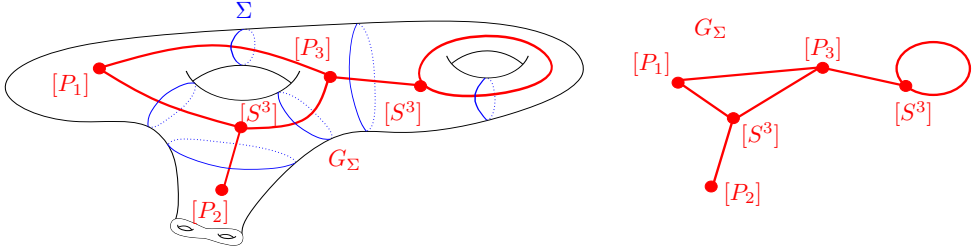}}
    \caption{A separating system $\Sigma \in \sepNP(M)$ and the corresponding dual graph.}
    \label{fig:dual graphs}
\end{figure}

\begin{lem}\label{lem:G-essentially-surjective}
    A $3$-manifold graph $\Gamma \in \Gr^{\rm mfd}$ appears as the dual graph of a separating system in $M$ if and only if it is an $M$-graph.
\end{lem}
\begin{proof}
    The manifold $M_\Gamma$ was constructed by $1$-handle attachments.
    If we let $\Sigma \subset M_\Gamma$ be the collection of belt spheres for each $1$-handle, then the dual graph of $\Sigma$ is canonically isomorphic to $\Gamma$.
    By assumption we have $\varphi\colon M_\Gamma \cong M$, so $\Gamma$ is isomorphic to the dual graph $G_{\varphi(\Sigma) \subset M}$.
\end{proof}

\begin{defn}\label{defn:SepMfd}
     Let $\SepMfd$ denote the following category.
     Objects are tuples $(M,\Sigma)$ of a connected, compact, oriented $3$-manifold $M$ without spherical boundary, not diffeomorphic to $S^1 \times S^2$, and a separating system $\Sigma \subset M$ in $\sepNP(M)$.
     Morphisms $(M,\Sigma) \to (N,\Upsilon)$ are orientation-preserving diffeomorphisms $\varphi\colon M \cong N$ satisfying $\varphi(\Sigma) \supseteq \Upsilon$.
\end{defn}

\begin{defn}\label{defn:dual-graph-functor}
    Every morphism $\varphi\colon (M, \Sigma) \to (N, \Upsilon)$ in $\SepMfd$ yields a graph map $G(\varphi)\colon G_{\Sigma \subset M} \to G_{\Upsilon \subset N}$ as follows.
    (An example is shown in \cref{fig:M-graphs}.)
    A vertex in $G_{\Sigma \subset M}$ is a connected component $U \subset M \ca \Sigma$ and $G(\varphi)$ sends it to the unique connected component $U' \subset N \ca \Upsilon$ such that $\varphi(\interior{U}) \subset \interior{U'}$ as submanifolds of $M$.
    A half-edge in $G_{\Sigma \subset M}$ is a connected component $S \subset \Sigma$ together with a coorientation.
    There are two cases:
    if $\varphi(S) \subset \Upsilon$, then the graph map $G(\varphi)$ sends $S$ to the half edge $\varphi(S)$ (with its induced coorientation) in $G_{\Upsilon \subset N}$,
    otherwise it sends it to the unique vertex $U' \subset N \ca \Upsilon$ such that $\varphi(S) \subset U'$.
\end{defn}

\begin{lem}
    The association $(M,\Sigma)\mapsto G_{\Sigma\subset M}$ yields a well-defined functor
    \[
        G\colon \SepMfd \longrightarrow \Gr^{\rm mfd}.
    \] 
\end{lem}
\begin{proof}
    First we show that $G(\varphi)$ is compatible with the root map and half-edge involution.
    Let $S \subset \Sigma$ be some component with a coorientation pointing into $\interior{U}$ for some component $U \subset M \ca \Sigma$, \emph{i.e.}~$S$ is a half-edge incident at the vertex $U$ in $G_{\Sigma \subset M}$.
    If $\varphi(S) \not\subset \Upsilon$ then both $S$ and $U$ are contained in the same connected component of $N\setminus \Upsilon$ and hence the half edge $S$ and the vertex $U$ are sent to the same vertex of $G_{\Upsilon \subset N}$.
    If, on the other hand, $\varphi(S) \subset \Upsilon$, then $S$ is mapped to the half-edge $\varphi(S)$, whose coorientation points into some component $U' \subset N \ca \Sigma$ and we will have $\varphi(U) \subseteq \interior{U'}$. 

    For $G(\varphi)$ to be a well-defined map in $\Gr^{\rm mfd}$ we have to show that the preimage of each vertex is a tree and that it is compatible with the labelling by diffeomorphism classes.
    If $v \in G_{\Upsilon \subset N}$ is a vertex corresponding to a connected component $U' \subset N \ca \Upsilon$, then $G(f)^{-1}(v)$ contains exactly those vertices corresponding to components $U_i \subset M \ca \Sigma$ with $\varphi(\interior{U_i}) \subseteq \interior{U'}$.
    Therefore, $U'$ can be obtained by gluing the $U_i$ along boundary spheres and we must have $[\widehat{U'}] = (\sharp_i [\widehat{U_i}]) \sharp [S^1 \times S^2]^{\sharp b_1}$ for $b_1$ the first Betti number of the graph $G(f)^{-1}(v)$.
Since we know that $\smash{\widehat{U'}}$ is irreducible, we must have $b_1=0$ and $[\widehat{U_i}]=[S^3]$ for all but at most one $i$.
    Therefore, $G(f)$ only collapses forests and moreover the labelling satisfies $l'(v) = \sharp_{w \in G(f)^{-1}(v)} l(w)$, as required.

    Finally, inspecting the definition shows that $G$ is functorial, \emph{i.e.}~that $G(f_2 \circ f_1) = G(f_2) \circ G(f_1)$.
\end{proof}

\begin{figure}[h!]
    \centering
     \resizebox{\linewidth}{!}{\import{}{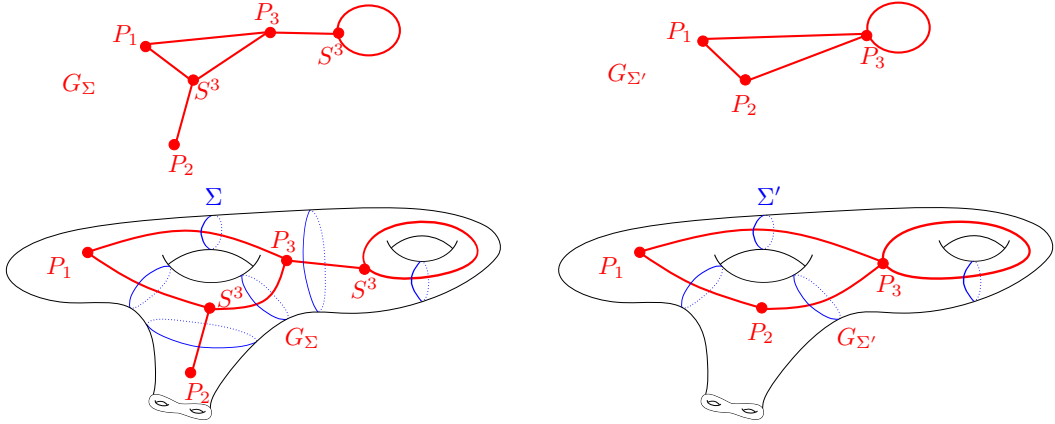}}
    \caption{The morphism $G_{\Sigma}\to G_{\Sigma'}$ induced by $\Sigma'\subset \Sigma \in \sepNP(M)$.}
    \label{fig:M-graphs}
\end{figure}

\begin{lem}\label{lem:graph-functor-is-full}
    The functor $G\colon \SepMfd \to \Gr^{\rm mfd}$ is full:
    for all $(M, \Sigma)$ and $(N, \Upsilon)$ in $\SepMfd$ and every graph morphism $f\colon G_{\Sigma \subset M} \to G_{\Upsilon \subset N}$ there is an orientation-preserving diffeomorphism $\varphi\colon M \cong N$ such that $\varphi(\Sigma) \supseteq \Upsilon$ and such that $G(\varphi) = f$.
    Moreover, $G$ is essentially surjective.
\end{lem}
\begin{proof}
    Let $\Sigma_0 \subseteq \Sigma$ be the union of those spheres for which the corresponding edge of $G_{\Sigma \subset M}$ is not collapsed by $f$.
    The dual graph $G_{\Sigma_0 \subset M}$ is hence obtained from $G_{\Sigma \subset M}$ by collapsing exactly those edges that are collapsed by $f$, and $f$ induces an isomorphism of labelled graphs $f_0\colon G_{\Sigma_0 \subset M} \cong G_{\Upsilon \subset N}$.
    It will suffice to show that there is a diffeomorphism $\varphi\colon M \cong N$ with $\varphi(\Sigma_0) = \Upsilon$ and $G(\varphi)\colon (M,\Sigma_0)\to (N,\Upsilon)) = f_0$.
    (Indeed, such a $\varphi$ also satisfies $\varphi(\Sigma) \supseteq \varphi(\Sigma_0) = \Upsilon$ and $G(\varphi\colon (M,\Sigma)\to (N,\Upsilon)) = f$.)

    The graph isomorphism $f_0\colon G_{\Sigma_0 \subset M} \cong G_{\Upsilon \subset N}$ induces a bijection between the sets of edges $\pi_0(\Sigma_0) \cong \pi_0(\Upsilon)$, and we can find a diffeomorphism $\psi\colon \Sigma_0 \cong \Upsilon$ such that $\pi_0(\psi)$ is exactly that bijection.
    Choose collars of $\Sigma_0 \subset \Sigma_0^\varepsilon$ in $M$ and $\Upsilon \subset \Upsilon^\varepsilon$ in $N$. 
    (That is, we have $\Sigma_0^\varepsilon \cong \Sigma_0 \times [-1,1]$ and $\Upsilon^\varepsilon \cong \Upsilon \times [-1,1]$.)
    Then we can extend $\psi$ to an orientation-preserving diffeomorphism $\psi^\varepsilon\colon \Sigma_0^\varepsilon \cong \Upsilon^\varepsilon$.
    Up to composing with reflections where necessary, we may assume that the map $\psi^\varepsilon$ induces on the set of spheres with coorientation is exactly the bijection that $f_0$ induces on the set of half-edges.

    We need to show that $\psi^\varepsilon$ extends to a diffeomorphism $\varphi\colon M \cong N$. (If it does, $\varphi$ will automatically induce the map $f_0$ on dual graphs because $f_0$ is determined by what it does on half-edges.)
    For each connected component $U \subset M\ca \Sigma_0$ corresponding to some vertex $v_U$ in $G_{\Sigma_0 \subset M}$, let $V \subset N \ca \Upsilon$ be the component  corresponding to the vertex $f_0(v_U)$ in $G_{\Upsilon \subset N}$.
    Because $f_0$ is an isomorphism in $\Gr^{\rm mfd}$, it preserves the labelling by oriented diffeomorphism classes (of spherical closures) and in particular $[\scl{U}]$ must agree with $[\scl{V}]$, \emph{i.e.}~there exists an orientation preserving diffeomorphism $\eta\colon \scl{U} \cong \scl{V}$.

    The preimage of $\Sigma_0^\varepsilon$ under the map $U \to M$ (which on the interior is the inclusion) defines a collection of annuli $A \subset U$ that collar the spherical boundaries of $U$.
    There is one such annulus for each half-edge incident at $v_U$.
    Similarly, the preimage of $\Upsilon^\varepsilon$ under the map $V \to N$ defines a collection of annuli $A'\subset V$ collaring the spherical boundaries.
    Because $f_0$ is a graph map it preserves the incidence of half-edges at vertices and thus $\psi^\varepsilon$ yields a diffeomorphism $A \cong A'$.
    In the spherical closure, we can complete the annuli to discs and $\psi^\varepsilon$ to a diffeomorphisms between these collections of discs.
    Now isotope $\eta \colon \scl{U} \cong \scl{V}$ in such a way that it restricts to this prescribed diffeomorphism on discs.
    (This follows because $\emb^+(\sqcup_k D^3,M)$ is connected for all $k$.)
    In particular, after the isotopy $\eta$ will restrict to a diffeomorphism $\varphi_U\colon U \cong V$ that agrees with $\psi^\varepsilon$ on overlaps.
    Performing this construction for all $U \in \pi_0(M \ca \Sigma_0)$ constructs the desired diffeomorphism $\varphi$.

    Essential surjectivity follows from \cref{lem:G-essentially-surjective}.
\end{proof}

By restricting the functor $G$ to the automorphisms of an object $(M,\Sigma)\in \SepMfd$ we obtain:
\begin{cor}\label{cor:SepMfd-Grmfd-automorphisms}
The functor $G\colon \SepMfd\to\Gr^{\rm mfd}$ induces a surjective group homomorphism
\begin{equation}\label{eqn:Aut-GSigma Gr mfd}
    \gamma_\Sigma\colon\Diff^+(M, \Sigma)\to\Aut_{\Gr^{\rm mfd}}(G_\Sigma).
\end{equation} 
\end{cor}

\subsection{Separating systems indexed by graphs}
In light of \cref{lem:graph-functor-is-full}, one can regard $\Gr(M)$ as encoding, at least combinatorially, all of the ways to build $M$ by attaching 1-handles to the irreducible factors of $M$. In order to get a model for $\BDiff^+(M)$, we will need to define a functor from $\Gr(M)$ to a collection of spaces equipped with a $\Diff^+(M)$-action.

The natural target for $\Gamma\in \Gr(M)$ will be a space parametrising all  separating systems whose dual graph is identified with $\Gamma$. To see how $\Diff^+(M)$ acts, observe that for any separating system $\Sigma$ and $\varphi\in \Diff^+(M)$, $\varphi\colon (M, \Sigma) \to (M,\varphi(\Sigma)=\Upsilon$) is a morphism in $\mathrm{SepMfd}$, so under the functor $G$ of \cref{defn:dual-graph-functor} we obtain an induced isomorphism of dual graphs $\varphi_G\colon G_\Sigma\cong G_\Upsilon$. 
In particular, if $G_\Sigma\cong \Gamma$ then $G_\Upsilon\cong \Gamma$. 

\begin{defn}\label{defn:Gamma-separating-system}
    For $\Gamma \in \Gr(M)$ we define a \emph{$\Gamma$-separating system} to be a separating system $\Sigma \in \sep^\Gamma(M)$ whose dual graph is (identified with) $\Gamma$.
    To be precise, we define the set of $\Gamma$-separating systems as
    \[
        \sep^\Gamma(M) := \{ (\Sigma, \alpha) \;|\; \Sigma \in \sepNP(M) \text{ and } \alpha\colon G_\Sigma \cong \Gamma\} 
    \]
    where $\Sigma$ is a separating system and $\alpha$ an isomorphism of labelled graphs.
    The group $\Diff^+(M)$ acts on $\sep^\Gamma(M)$ by
    $\varphi.(\Sigma, \alpha) = (\varphi(\Sigma), \alpha \circ \varphi_G^{-1})$
    where $\varphi_G \colon G_{\Sigma} \cong G_{\varphi(\Sigma)}$ is the isomorphism that $\varphi$ induces on dual graphs.
    (This action is transitive by \cref{lem:graph-functor-is-full}.)
    We topologise $\sep^\Gamma(M)$ in the unique way such that the orbit map,
    \[
        \Diff^+(M) \twoheadrightarrow \sep^\Gamma(M),
    \]
    defined by acting on some (and hence any) point, is a quotient map. 
\end{defn}

\begin{defn}\label{defn:Sep-functor} 
    For a morphism $f\colon \Gamma \to \Lambda$ in $\Gr(M)$ we define
    \[
        f_* \colon \sep^\Gamma(M) \to \sep^\Lambda(M), \qquad
        (\Sigma, \alpha) \mapsto (\Sigma', \alpha')
    \]
    where $\Sigma' \subseteq \Sigma$ contains all those spheres corresponding to edges in $G_\Sigma \overset{\alpha}{\cong} \Gamma$ that are not collapsed by $f$
    and $\alpha'\colon G_{\Sigma'} \cong \Lambda$ is the unique isomorphism such that the following square commutes
\[\begin{tikzcd}
	{G_\Sigma} & \Gamma \\
	{G_{\Sigma'}} & \Lambda
	\arrow[from=1-1, to=2-1]
	\arrow["\alpha", from=1-1, to=1-2]
	\arrow["f", from=1-2, to=2-2]
	\arrow["{\alpha'}", from=2-1, to=2-2]
\end{tikzcd}\]
    where the left hand map is the image of the morphism $\id\colon(M,\Sigma)\to (M, \Sigma') \in \SepMfd$ under the functor $G$ from \cref{defn:dual-graph-functor}.
    This defines a functor $\sep^{(-)}(M) \colon \Gr(M) \to \Diff^+(M)\text{-}\Top$ from $M$-graphs to topological spaces with a continuous $\Diff^+(M)$-action.
\end{defn}

\begin{rem}\label{rem:straightening}
    From a (higher) category theory perspective, we can think of $\sep^\Gamma(M)$ as the fibre of the restricted functor
    \[
        G_{|\sep(M)}\colon \sep(M) \longrightarrow \Gr(M) \subset \Gr^{\rm mfd}.
    \]
    This functor is (equivalent to) a left fibration and hence sending an object $\Gamma \in \Gr(M)$ to its fibre $G_{|\sep(M)}^{-1}(\Gamma)$ defines a functor from $\Gr(M)$ to the category of spaces.
    (To make this precise, we would have to turn our topological categories into $\infty$-categories.)
    Moreover, ``straightening'' a left fibration in this way always yields a diagram whose colimit is exactly the weak homotopy type of the total space \cite[Corollary 3.3.4.6]{HTT}.
    This analogy hence lines up with \cref{prop:sep-contractible} where we will see that the homotopy colimit of $\sep^{(-)}(M)$ is equivalent to $\|\sep(M)\|$.
\end{rem}

The next lemma highlights some topological properties of $\sep^\Gamma(M)$ that we will need.

\begin{lem} \label{lem: Sep-Gamma retractile}For any $\Gamma\in \Gr(M)$: 
    \begin{enumerate}
        \item The projection map $p_\Gamma\colon \sep^\Gamma(M)\to \sepNP(M)$ sending $(\Sigma,\alpha)\mapsto \Sigma$ is a covering whose fibre at a separating system $\Sigma$ is the set of graph isomorphisms $\Gamma \cong G_\Sigma$.
        (In particular the fibre is empty when $G_\Sigma$ is not isomorphic to $\Gamma$.)
        \item $\sep^\Gamma(M)$ is $\Diff^+(M)$-locally retractile.
    \end{enumerate}
\end{lem}
\begin{proof}
    We know that $\sepNP(M)$ is $\Diff^+(M)$-locally retractile by \cref{prop:sepNP-Properties}(1) and \cref{rem:Finite-Index-Retractile}. Since $p_\Gamma$ is $\Diff^+(M)$-equivariant, it will suffice to verify the fibres of $p_\Gamma$ are discrete by Lemma \ref{lem:Retractile Covering}. By definition, the fibre of $p_\Gamma$ is the set of graph isomorphisms $\Gamma\cong G_\Sigma$ in $\Gr(M)$, which is either empty or finite and discrete.
\end{proof}

\begin{defn}
    For a separating system $\Sigma \subset M$ we define the subgroup $\smash{\Diff_{\pi_0(2\Sigma)}^+(M, \Sigma)}$ to be the subgroup consisting of $\smash{\varphi \in \Diff^+(M,\Sigma)}$ such that for each sphere $S\subset \Sigma$, we have $\varphi(S)=S$ and $\varphi$ preserves the coorientation of $S$.
\end{defn}
    
    This subgroup is the kernel of the homomorphism $\gamma_\Sigma\colon\Diff^+(M,\Sigma)\to \Aut_{\Gr(M)}(G_\Sigma)$ defined in \cref{defn:dual-graph-functor}. 
    Since $\gamma_\Sigma$ is surjective, $\smash{\Diff_{\pi_0(2\Sigma)}^+(M, \Sigma)}$ fits into the short exact sequence
    \begin{equation}\label{eqn:Aut_Gr(M)-sequence}
        1\rightarrow \Diff_{\pi_0(2\Sigma)}^+(M, \Sigma)\rightarrow \Diff^+(M,\Sigma)\twoheadrightarrow \Aut_{\Gr(M)}(G_\Sigma)\rightarrow 1.
    \end{equation}
    In terms of the group action, $\smash{\Diff_{\pi_0(2\Sigma)}^+(M, \Sigma)}$ is exactly the $\Diff^+(M)$-stabiliser of $(\Sigma,\alpha)\in \sep^\Gamma(M)$ for any isomorphism $\alpha\colon G_\Sigma\cong \Gamma$.
    In particular, since $\Aut_{\Gr(M)}(G_\Sigma)$ is finite, $\smash{\Diff_{\pi_0(2\Sigma)}^+(M, \Sigma)}$ is a finite index subgroup of $\smash{\Diff^+(M, \Sigma)}$.

\begin{rem}
    The subgroup
    $\smash{\Diff_{\pi_0(2\Sigma)}^+(M, \Sigma)}$ can be identified as those orientation-preserving diffeomorphisms $\varphi\colon M \to M$ that are isotopic (in $\Diff^+(M , \Sigma)$) to a diffeomorphism that is the identity in a neighbourhood of $\Sigma$.
    However, $\Diff_{\pi_0(2\Sigma)}^+(M, \Sigma)$ is not equivalent to the subgroup of diffeomorphisms which fix a neighbourhood of $\Sigma$ because we still allow the spheres to rotate.
\end{rem}

\begin{lem}\label{lem:Diff(M)->sep(M) principal}
    Let $\Sigma \subset M$ be a separating system with dual graph $\Gamma = G_\Sigma$. 
    \begin{enumerate}
        \item There is a homeomorphism
        \[
            \sep^\Gamma(M) \cong \Diff^+(M)/\Diff_{\pi_0(2\Sigma)}^+(M, \Sigma).
        \]
        \item The quotient map $\Diff^+(M) \to \sep^\Gamma(M)$ is a  principal $\Diff_{\pi_0(2\Sigma)}^+(M, \Sigma)$-bundle, and in particular, $\Diff_{\pi_0(2\Sigma)}^+(M, \Sigma)$ is an admissible subgroup of $\Diff^+(M)$.
    \end{enumerate}
     
\end{lem}

\begin{proof}
    Since the action of $\Diff^+(M)$ on $\sep^\Gamma(M)$ is transitive, we can identify $\sep^\Gamma(M)$  with the coset space of the stabiliser of any  separating system $\Sigma$ whose dual graph is isomorphic to $\Gamma$. This stabiliser is $\Diff_{\pi_0(2\Sigma)}^+(M, \Sigma)$, showing (1). Now by Lemma \ref{lem: Sep-Gamma retractile}(2) $\sep^\Gamma(M)$ is $\Diff^+(M)$-locally retractile, which, together with (1), implies (2).
\end{proof}

Since $\sep^{(-)}(M)$ (\cref{defn:Sep-functor}) defines a functor to topological spaces equipped with a $\Diff^+(M)$-action, we obtain a new functor from $\Gr(M)$ to $\Top$ by taking a homotopy quotient with respect to the $\Diff^+(M)$-action.

\begin{defn}\label{defn:D-Gamma-functor}
    We define a functor $\rmD \colon \Gr(M) \to \Top$ by
    \[
        \rmD(\Gamma) := 
        \sep^\Gamma(M) \hq \Diff^+(M) =
         \big( \emb(M, \bbR^\infty) \times \sep^\Gamma(M) \big)/ \Diff^+(M).
    \]
\end{defn}

A point in $\rmD(\Gamma)$ can be described as a triple $(W, \Sigma, \alpha)$ where $W \subset \bbR^\infty$ is an oriented $3$-manifold, $\Sigma \subset W$ is a separating system, and $\alpha \colon G_\Sigma \cong \Gamma$ an isomorphism of labelled graphs. 
This implies that $W$ is abstractly diffeomorphic to $M$, but there is no preferred such diffeomorphism.

\begin{cor}\label{cor:rmD-classifying-space}Let $\Gamma\in \Gr(M)$, then for every $\Sigma$ such that $G_\Sigma\cong \Gamma$, there is an equivalence
    \[
        \rmD(\Gamma) \simeq \BDiff_{\pi_0(2\Sigma)}^+(M, \Sigma).
    \] 
\end{cor}
\begin{proof}
It follows from Lemma \ref{lem:Diff(M)->sep(M) principal}(1) that 
    \begin{align*}
        \rmD(\Gamma) 
         &=\big( \emb(M, \bbR^\infty) \times \sep^\Gamma(M) \big)/ \Diff^+(M)\\
        &\cong \left(\emb(M, \bbR^\infty) \times \big(\Diff^+(M)/\Diff_{\pi_0(2\Sigma)}^+(M, \Sigma)\big)\right)/ {\Diff^+(M)}\\
         &\cong \emb(M, \bbR^\infty)/\Diff_{\pi_0(2\Sigma)}^+(M, \Sigma),
    \end{align*}
    which is a model for $\BDiff_{\pi_0(2\Sigma)}^+(M)$ by Lemma \ref{lem:Diff(M)->sep(M) principal}(2) and \cref{obs:admissible}(3).
\end{proof}

\subsection{The homotopy colimit}
Recall the definition of homotopy colimit from \cref{subsec:background-on-hocolims}. We will now show that the contractibility of $|\sepNP_\bullet(M)|$ implies that the homotopy colimit of $\sep^\Gamma(M)$  indexed over $\Gr(M)$ is  contractible.
After taking homotopy orbits with respect to the $\Diff^+(M)$ action on $\sep^\Gamma(M)$, we will then obtain a description of  $\BDiff^+(M)$ as a homotopy colimit of $\rmD(\Gamma)$. Let us first recall a technical result that we will need to deduce contractibility.

\begin{lem}[Base-change lemma, {\cite[Theorem 3.25]{Ste21}}]\label{lem - base change}
    Let $f \colon X_\bullet \to Y_\bullet$ be a map of simplicial spaces such that the square
    \[
        \begin{tikzcd}
            X_n \ar[r, "f_n"] \ar[d, "\partial_n"] & Y_n \ar[d, "\partial_n"] \\
            (X_0)^{n+1} \ar[r, "(f_0)^{n+1}"] & Y_0^{n+1}
        \end{tikzcd}
    \]
    is a homotopy pullback square and such that $f_0\colon X_0 \to Y_0$ is surjective on connected components.
    (Here the vertical map sends an $n$-simplex to its $n+1$ vertices.)
    Then the induced map on fat geometric realisations is an equivalence:
    \[
        \|f\| \colon \|X_\bullet\| \xrightarrow{\simeq} \|Y_\bullet\|.
    \]
\end{lem}

\begin{prop}\label{prop:sep-contractible}
    The homotopy colimit
    $
        \hocolim_{\Gamma \in \Gr(M)} \sep^\Gamma(M)
    $
    is contractible.
\end{prop}
\begin{proof}
    From Definition~\ref{defn - hocolim} the homotopy colimit is given by 
    $|E_\bullet|$, where $E_\bullet$ is the simplicial space with 
    \[
        E_n := \coprod_{\Gamma_0 \to \dots \to \Gamma_n} \sep^{\Gamma_0}(M)
    \]
    such that the coproduct runs over all sequences of $n$ composable morphisms in $\Gr(M)$.
    Because the inclusion of the degenerate simplices is a cofibration (in fact, it is a disjoint summand) we can equivalently work with the fat geometric realisation.
    Thus a point~$p\in E_n$ has the form
    \[
        p=(\Gamma_0 \overset{f_1}{\to} \cdots \overset{f_n}{\to} \Gamma_n, (\Sigma_0, \alpha_0)) 
    \]
    where the first entry is the indexing chain, and $(\Sigma_0, \alpha_0)\in \sep^{\Gamma_0}(M)$. 
    We can now define a map from $E_\bullet$ to  $\sepNP_\bullet(M)$,
    which is given levelwise by 
    \begin{eqnarray*}
        \phi_n \colon E_n &\longrightarrow &\sepNP_n(M)\\
        (\Gamma_0 \overset{f_1}{\to} \cdots \overset{f_n}{\to} \Gamma_n, (\Sigma_0, \alpha_0)) &\mapsto& (\Sigma_n\subseteq \cdots \subseteq \Sigma_{1} \subseteq \Sigma_0)
    \end{eqnarray*}
    where $\Sigma_i \subseteq \Sigma_0$ is the subsystem containing exactly those spheres of $\Sigma_0$ that correspond to edges in $\Gamma_0$ that are not collapsed by the composite $(f_i \circ \dots \circ f_1)\colon \Gamma_0 \to \Gamma_i$.
    Equivalently, we can use the functoriality of $\sep^\Gamma(M)$ to define $\Sigma_i$ to be the image of $\Sigma_0$ under the map $\sep^{\Gamma_0}(M) \to \sep^{\Gamma_i}(M)$ induced by $(f_i \circ \dots \circ f_1)$.

    We claim that $\phi_\bullet$ satisfies the conditions of the base-change lemma (Lemma~\ref{lem - base change}) and hence induces an equivalence:
    \[
        \|\phi_\bullet\|\colon \|E_\bullet\| \xrightarrow{\simeq} \|\sepNP_\bullet(M)\|
    \]
    This will conclude the proof as we know from Theorem~\ref{thm: sepNP contractible} that the right hand side is contractible.
    
    It remains to prove that $\phi_\bullet$ satisfies the conditions of the base-change lemma, \emph{i.e.}~we want to show that the square
    \[
        \begin{tikzcd}
            E_n \ar[r, "\phi_n"] \ar[d, "\partial_n"] & \sepNP_n \ar[d, "\partial_n"] \\
            (E_0)^{n+1} \ar[r, "(\phi_0)^{n+1}"] & (\sepNP_0)^{n+1}
        \end{tikzcd}
    \]
    is a homotopy pullback square and $\phi_0\colon E_0 \to \sepNP_0$ is surjective on connected components. 
    To show this, we first argue that the map~$\phi_0$ is a surjective covering map (Claim 1) from which it follows that the square is a homotopy pullback if it is a pullback of spaces. We then show that the vertical maps are subspace inclusions (Claim 2), and 
    hence being a pullback of spaces is equivalent to being a pullback of sets.

    {\bf Claim 1}: $
    \phi_0\colon E_0 \to \sepNP_0$ is a surjective covering map.
    
    \textit{Proof of Claim 1:} 
    For each $\Gamma$,  $\sep^{\Gamma}(M) \to \sepNP(M)$ is a covering map by Lemma \ref{lem: Sep-Gamma retractile}.
    Its image is the subspace of all those separating systems whose dual graph is isomorphic to $\Gamma$.
    By taking the disjoint union over all $\Gamma \in \Gr(M)$ we obtain a surjective covering map
    \[
        E_0 = \coprod_{\Gamma \in \Gr(M)} \sep^\Gamma(M) \twoheadrightarrow \sepNP(M).\qed
    \]
    
    {\bf Claim 2:} The vertical maps $E_n\to (E_0)^{n+1}$ and $\sepNP_n \to (\sepNP_0)^{n+1}$ are both subspace inclusions.
    
    \textit{Proof of Claim 2:} 
    For $\sepNP$ this follows from the definition, as we defined $\sepNP_n(M)$ as a subspace of $\sep_n(M)$, which we defined as a subspace of $\sep_0(M)^{n+1}$.

    In the case of $E_n$ the map in question is:
    \[
        E_n = \coprod_{\Gamma_0 \to \dots \to \Gamma_n} \sep^{\Gamma_0}(M)
        \longrightarrow
        \coprod_{\Gamma_0, \dots, \Gamma_n} \sep^{\Gamma_0}(M) \times \dots  \times \sep^{\Gamma_n}(M) \subset (E_0)^{n+1}.
    \]
    This is indeed a subspace inclusion, as the information of the map $\Gamma_i \to \Gamma_{i+1}$ can be recovered from the inclusion of separating systems $\Sigma_{i+1} \subseteq \Sigma_i$ and the graph isomorphisms $\alpha_i \colon G_{\Sigma_i} \cong \Gamma_i$.\qed

    Now to show that the square is a homotopy pullback, by Claims 1 and 2 we need only show it is a pullback of sets, which follows from rewriting points in~$E_n$ as tuples $\{ (\Sigma_i, \alpha_i)\}_{i=0}^n$.
    The remaining condition of the base-change lemma is that $\phi_0$ is surjective on connected components, which we have already established as part of Claim $1$.
\end{proof}

\begin{thm}\label{thm:hocolim-intro}
    For every $M \ncong S^1 \times S^2$ there is a functor $\rmD\colon \Gr(M) \to \Top$ such that 
    \[
        \BDiff^+(M)\simeq \hocolim_{\Gamma \in \Gr(M)} \rmD(\Gamma).
    \]
    If $\Gamma = G_\Sigma$ is the dual graph of a sphere system $\Sigma \subset M$, then $\rmD(\Gamma)$ is equivalent to the classifying space of the group of those orientation-preserving diffeomorphisms $\varphi\in \Diff^+(M)$ such that for all spheres $S \subseteq \Sigma$ we have $\varphi(S) = S$ and $\varphi_{|S}$ is orientation-preserving.
\end{thm}
\begin{proof}
    Let $\calS$ be the functor from \cref{defn:D-Gamma-functor}.
    By \cref{cor:rmD-classifying-space} this satisfies that $\calS(G_\Sigma)$ is the classifying space of $\Diff^+_{\pi_0(2\Sigma)}(M, \Sigma)$, which is exactly the group described in the statement of the theorem.

    To compute the homotopy colimit of $\calS$, we can commute homotopy colimits with homotopy quotients.
    (The homotopy colimit is defined as geometric realisation of disjoint unions and homotopy quotients commute with both of those constructions \cite[Lemma 2.12(1) and (5)]{BoydBregmanSteinebrunner-finiteness}.)
    Hence, we compute
    \begin{eqnarray*}
        \hocolim_{\Gamma \in \Gr(M)} \rmD(\Gamma) 
        &\simeq& \hocolim_{\Gamma \in \Gr(M)} \left(\sep^\Gamma(M) \hq \Diff^+(M)\right)\\
        &\simeq& \left(\hocolim_{\Gamma \in \Gr(M)} \sep^\Gamma(M)\right)\hq\Diff^+(M) \\
        &\simeq& \BDiff^+(M)
    \end{eqnarray*}
    using that $\hocolim \sep^\Gamma(M) \simeq *$ by \cref{prop:sep-contractible}.
\end{proof}

\section{The splitting map}\label{sec:wrongway map}

For any oriented 3-manifold $M$ with prime decomposition $M=P_1\sharp \cdots \sharp P_n\sharp (S^1\times S^2)^{\sharp g}$,
we set \[P_M=\amalg_{i=1}^nP_i,\] where $P_M$ inherits its orientation from $M$. In this section we define a functor $\calF^{-}(P_M)\colon \Gr(M) \to \Diff^+(P_M)\text{-}\Top$, which refines the functor $\rmD(-)$ in the sense that $\calF^\Gamma(P_M) \hq \Diff^+(P_M) \simeq \rmD(\Gamma)$.
This will allow us to construct a ``splitting map" $\W$ that fits into a homotopy fibre sequence:
\[
    \hocolim_{\Gamma \in \Gr(M)} \calF^\Gamma(P_M) \longrightarrow \BDiff^+(M) \xrightarrow[\qquad]{\W} \BDiff^+(P_M).
\]
In general, the functor $\calF^\Gamma(P_M)$ takes disconnected values.
By restricting to one connected component we are able to rewrite the homotopy colimit as the homotopy colimit of another functor $\rmC$ on the category $\Gr_{g,n}$, as stated in the introduction.
We denote the fibre of $\W$ by $\calH_g(P_1,\dots,P_n)$ and think of it as a ``space of handle-attachments to $P_M$''-- this perspective will be further substantiated by the descriptions in \cref{sec:describing the fibre}.
Finally, we also show that $\W$ commutes with restriction to the boundary and use this to obtain variants $\W_A$ for moduli spaces where some submanifold $A \subset \partial M$ is fixed.

\subsection{Overview of the strategy}Given a separating system $\Sigma\subset M$, let $(M\ca\Sigma)^{{\rm prime}}$ denote the union of the components of~$M \ca \Sigma$  not diffeomorphic to~$S^3\setminus \sqcup_k \interior{D}^3$  for some $k\geq 0$. Our na\"ive approach towards constructing $\W$ is to define a ``prime marking" to be an embedding $(M\ca\Sigma)^{\rm prime}\hookrightarrow P_M $. This leads to consider a space of all such ``prime markings" up to reparametrisations:
\[\calF_0(\Gamma) := \emb^{\scl=}((M\ca\Sigma)^{\rm prime},P_M)\hq \Diff^+_{\pi_0(2\Sigma)}(M, \Sigma).\]
(Recall the superscript $\scl{=}$ from \cref{defn: equal-after-scl-embeddings} indicates ``diffeomorphism after taking spherical closures".) This admits an action of $\Diff^+(P_M)$ by post-composition and we can use Lemma \ref{lem:embedding complement of discs} to argue that the homotopy quotient recovers $\rmD(\Gamma)$.
In particular, we obtain a homotopy fibre sequence 
\[
    \calF_0(\Gamma) \longrightarrow \calF_0(\Gamma)\hq \Diff^+(P_M) \simeq \rmD(\Gamma) \longrightarrow \BDiff^+(P_M).
\]
If we could now take homotopy colimits of the left and middle term over $\Gamma \in \Gr(M)$, this would result in the desired prime decomposition fibre sequence:
\[
    ``\hocolim_{\Gamma \in \Gr(M)}\calF_0(\Gamma)" \longrightarrow
    \BDiff^+(M) \simeq \hocolim_{\Gamma \in \Gr(M)} \rmD(\Gamma) \longrightarrow
    \BDiff^+(P_M).
\]

Unfortunately, this na\"ive approach suffers from a lack of functoriality with respect to morphisms in $\Gr(M)$, which prevents us from taking the homotopy colimit.
Concretely, given a graph $\Gamma_0$ and a separating system $\Sigma_0$ with $\Gamma_0\cong G_{\Sigma_0}$, an edge collapse $\Gamma_0\to \Gamma_1$ in~$\Gr(M)$ corresponds to forgetting a sphere in~$\Sigma_0$, yielding $\Sigma_0\supset \Sigma_1$. Consequently  there is an (in general, strict) inclusion $(M \ca \Sigma_0)^{\rm prime}\subset (M \ca \Sigma_1)^{\rm prime}$. Starting with an embedding $(M \ca \Sigma_0)^{\rm prime} \hookrightarrow P_M$, one thus has no functorial way of extending it to $(M \ca \Sigma_1)^{\rm prime}$.
This motivates us to define prime markings in terms of what we call the \emph{maximal prime piece} $P(\Sigma \subset M)$.
Crucially, this modification has the advantage that passing to a sub-sphere system $\Sigma_0 \supset \Sigma_1$ yields a smaller maximal prime piece, hence functoriality will follow from restriction of embeddings.

\subsection{The maximal prime piece}

Let $(M\setminus \Sigma)^{\rm prime}$ be defined in an analogous way
to $(M\ca\Sigma)^{\rm prime}$ above. The components of $M\setminus \Sigma$ that are diffeomorphic to punctured 3-spheres will be referred to as \emph{spherical components} of $M\setminus \Sigma$, while components of $(M\setminus \Sigma)^{\rm prime}$ will be called \emph{prime components}.

\begin{defn}
    A path $\gamma\colon [0,1]\to M$ is called \emph{$\Sigma$-primary} if $\gamma(0)\in (M\setminus \Sigma)^{\rm prime}$, and $\gamma^{-1}(C)$ is connected or empty for each component $C \subset M \setminus \Sigma$, and moreover $\gamma$ only intersects one component of $(M \setminus \Sigma)^{\rm prime}$.
\end{defn}

For a $\Sigma$-primary path $\gamma$, the pre-image  $\gamma^{-1}((M\setminus \Sigma)^{\rm prime})$ is either $[0,1]$ or $[0,\varepsilon)$ for some $\varepsilon>0$. In other words, $\gamma$ starts in a prime component and either remains in that component for the entire path, or leaves and never re-enters that component nor any other prime component.

\begin{defn}\label{defn - maximal prime piece}
    For a separating system $\Sigma \subset M$, define the \emph{maximal prime piece} $P(\Sigma \subset M)$ as 
    \[
        P(\Sigma \subset M) = \left\{ (p, [\gamma]) \, \Bigg| \begin{array}{l} p \in M,~
        \gamma\colon [0,1] \to M,~ \gamma(0)\in (M\setminus \Sigma)^{\rm prime},~ \gamma(1) = p,\\
        
        [\gamma] \text{ contains a $\Sigma$-primary representative}
         \end{array} \right\}
   \]
    where $[\gamma]$ is the equivalence class of $\gamma$ up to homotopies that fix $\gamma(1)$ and preserve the condition that $\gamma(0)\in (M\setminus \Sigma)^{\rm prime}$.
\end{defn}

\begin{figure}[ht]
    \centering
    \def\svgwidth{\linewidth}
    \import{}{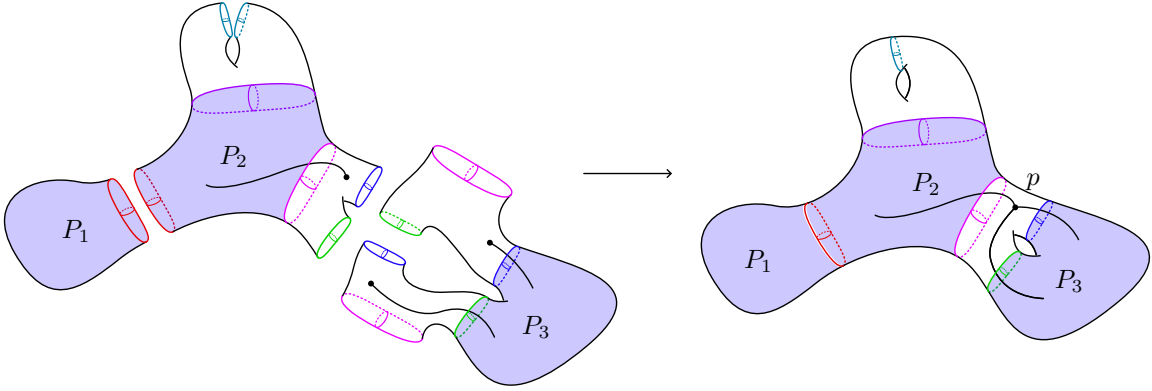}
    \caption{The maximal prime piece $P(\Sigma \subset M)$ of a separating system $\Sigma \subset M$, and the map $P(\Sigma \subset M) \to M$.
    The shaded region is the ordinary prime piece $(M \setminus \Sigma)^{\rm prime}$ on which the map restricts to a homeomorphism.
    Note that the point $p \in M$ has three preimages, depending on which primary path we equip it with.}
    \label{fig:maximal-prime-piece}
\end{figure}

The definition of $P(\Sigma\subset M)$ is reminiscent of the construction of the universal cover of $M$ as a set of based paths up to homotopy. 
We will show that $P(\Sigma\subset M)$ is a 3-manifold with boundary and that the projection map $\pi\colon P(\Sigma\subset M)\rightarrow M$ defined by $(p,[\gamma])\mapsto p$ is a local homeomorphism map away from the spherical boundary components. 
 
For this, it it will be useful to consider the variant
\[
    Q_\Sigma := \left\{ (p, [\gamma]) \;\big|\;  p \in M,~
        \gamma\colon [0,1] \to M,~ \gamma(0)\in (M\setminus \Sigma)^{\rm prime},~ \gamma(1) = p\right\}
\]
where we do not require $[\gamma]$ to contain a $\Sigma$-primary representative.
The projection map $\xi\colon Q_\Sigma \to M$ sending $(p,[\gamma]) \to p$ is a covering map and hence $Q_\Sigma$ is a smooth oriented $3$-manifold, though it will typically not be compact nor connected.
By definition, the maximal prime piece $P(\Sigma \subset M)$ is the subspace of those $(p,[\gamma])$ where the equivalence class $[\gamma]$ contains a $\Sigma$-primary representative.

To understand the connected components of $P(\Sigma \subset M)$,
note that there is a natural embedding $(M \setminus \Sigma)^{\rm prime} \hookrightarrow P(\Sigma \subset M)$ defined by sending $p$ to $(p, [\mathrm{cst}_p])$ where $\mathrm{cst}_p$ is the constant path at $p$.
This induces a bijection
    \[
        \pi_0((M \setminus \Sigma)^{\rm prime}) \cong \pi_0(P(\Sigma \subset M)) .
    \]
To see that it is bijective, we construct an inverse by sending the path component of $(p,[\gamma])$ in $P(\Sigma \subset M)$ to $[\gamma(0)]$, which is well-defined since $\gamma(0)$ must always remain in the same prime component under a homotopy. 
This is indeed an inverse as $(p,[\gamma])$ and $(\gamma(0), [\mathrm{cst}_{\gamma(0)}])$ are connected by a path defined by $\gamma$, if $\gamma$ is $\Sigma$-primary.
By the same reasoning, the inclusion into $Q_\Sigma$ also induces a bijection on path-components.

\begin{prop}\label{prop - max prime piece is diffeo to P minus some discs}
    For any sphere system $\Sigma\subset M$, the maximal prime piece is abstractly homeomorphic to a 3-manifold $P_M'$ obtained from $P_M$ by removing the interiors of finitely many discs. 
\end{prop}
\begin{proof}
    As noted above, $P(\Sigma \subset M)$ is a subspace of $Q_\Sigma$ and $\xi\colon Q_\Sigma \to M$ is a covering map.
    Let $X \subset P(\Sigma \subset M)$ be the homeomorphic image of $(M \setminus \Sigma)^{\rm prime}$ obtained by equipping points with the constant path.
    Let $\widetilde{\Sigma} := \xi^{-1}(\Sigma)$ be the collection of spheres obtained as the lifts of the separating system in $M$.
    
    The connected components of $P(\Sigma \subset M)$ and of $Q_\Sigma$ are in bijection with the components of $X \cong (M\setminus \Sigma)^{\rm prime}$, which in turn correspond to the prime factors $P_M = \amalg_{i=1}^n P_i$.
    The inclusion $X \to Q_\Sigma$ is an isomorphism on the fundamental group for each connected component.%
    \footnote{
        In fact, $X \to Q_\Sigma \to M$ is exactly the first Moore--Postnikov factorization of $X \to M$. 
    }
    Indeed, if we decompose $Q_\Sigma = \amalg_{i=1}^n Q_{\Sigma,i}$, then each $Q_{\Sigma,i} \to M$ is the covering space corresponding exactly to the inclusion $\pi_1(P_i) \hookrightarrow \pi_1(M)$.
    Form this it follows that each component of $\widetilde{\Sigma} \cap Q_{\Sigma,i}$ separates $Q_{\Sigma,i}$: 
    otherwise there would be a surjection $\pi_1(Q_{\Sigma,i}) \twoheadrightarrow \bbZ$ that kills $\pi_1(X \cap Q_{\Sigma,i})$, but this group is isomorphic to $\pi_1(Q_{\Sigma,i})$.

    The restriction on $P(\Sigma\subset M)$ that $[\gamma]$ contain a $\Sigma$-primary representative equivalently means that it consists of all points in $Q_\Sigma$ that can be connected to $X$ by a path that passes through at most one elevation of each component of $M\setminus \Sigma$, none of which are in $(M\setminus\Sigma)^{\rm prime}$ except $X$ itself. 
    Since $M\ca \Sigma$ only has finitely many components, $P(\Sigma \subset M)$ is a union of the closure of $X$ and a finite number of copies of the simply connected components of $M\ca\Sigma$, where any two such components meet along a separating 2-sphere in $\smash{\widetilde{\Sigma}}$.
    In particular, $P(\Sigma \subset M)$ is a compact submanifold of $Q_\Sigma$ and can be obtained from $(M\ca\Sigma)^{\rm prime}$ by successively attaching punctured 3-balls along boundary 2-spheres, so it is homeomorphic to $\amalg_{i=1}^n (P_i \setminus \amalg_{j_i} \interior{D^3})$.
\end{proof}

As alluded to above, the crucial property of the maximal prime piece is that refining the sphere system makes the maximal prime piece is larger. This is based on the following observation.

\begin{lem}\label{lem:Sigma-Upsilon-primary}
    Let $\Upsilon \subset \Sigma \subset M$ be two separating systems in $M$.
    Then every $\Upsilon$-primary path $\gamma$ can be homotoped to a $\Sigma$-primary path while fixing the endpoint $\gamma(1)$.
\end{lem}
\begin{proof}
    The inclusion $(M \setminus \Sigma)^{\rm prime} \subseteq (M \setminus \Upsilon)^{\rm prime}$ is a bijection on $\pi_0$.
    Therefore, we can homotope the starting point $\gamma(0)$ within $(M \setminus \Upsilon)^{\rm prime}$ such that it lies in $(M \setminus \Sigma)^{\rm prime}$ and $\gamma$ is still $\Upsilon$-primary.
    
    Working inductively, we may assume that $\Sigma = \Upsilon \sqcup S$ for a single $2$-sphere $S \subset M \setminus \Upsilon$.
    Let $C \subset N \setminus \Upsilon$ be the component containing $S$, so as a set $C = C_0 \sqcup S \sqcup C_1$ for two components $C_0, C_1 \subset N \setminus \varphi(\Sigma)$.
    There are two cases: either $C_0$ and $C_1$ both are a punctured spheres, or one of them is a punctured sphere and the other is a punctured $P_i$.
    Because at least one of the $C_i$ is simply connected and $S$ is connected, it is possible to homotope $\gamma_{|\gamma^{-1}(C)}$ (within $C$ and relative to its endpoints) to $\gamma'$ with $(\gamma')^{-1}(C_i)$ connected or empty.
    Extending this by the constant homotopy outside of $\gamma^{-1}(C)$ we obtain $\gamma \sim \gamma'$ such that $\gamma'$ is $\Sigma$-primary. 
\end{proof}

From the above lemma we can deduce a \emph{contravariant} functoriality of the maximal prime piece.

\begin{prop}\label{prop: restricting maximal prime piece}
    Every morphism $\varphi\colon (M, \Sigma) \to (N, \Upsilon)$ in $\SepMfd$, \emph{i.e.}~every diffeomorphism $\varphi\colon M \cong N$ with $\varphi(\Sigma) \supseteq \Upsilon$,
    induces an embedding of maximal prime pieces $P(\varphi)\colon P(\Upsilon \subset N) \hookrightarrow P(\Sigma \subset M)$.
    This embedding is uniquely characterised by the property that it makes the diagram
    \[\begin{tikzcd}[column sep = large]
        {(M \setminus \Sigma)^{\rm prime}} \rar[hook] \dar["{\varphi_{|(M \setminus \Sigma)^{\rm prime}}}"', hook]& 
        {P(\Sigma \subset M)}  \rar & 
        M \dar["{\cong}"', "\varphi"] \\
        {(N \setminus \Upsilon)^{\rm prime}} \rar[hook] & 
        {P(\Upsilon \subset N)}  \rar \uar["{P(\varphi)}", hook] & 
        N
    \end{tikzcd}\]
    commute.
    This promotes maximal prime piece construction into a continuous functor
    \[
        P\colon \SepMfd^{\rm op} \longrightarrow \mathrm{Mfd}_3
    \]
    into the category of $3$-manifolds and smooth embeddings.
\end{prop}
\begin{proof}
    By \cref{prop - max prime piece is diffeo to P minus some discs} the maximal prime piece is a compact submanifold with boundary $P(\Sigma \subset M) \subset Q_\Sigma$.
    There is a unique diffeomorphism $\psi$ making the square
    \[\begin{tikzcd}[column sep = large]
        {(M \setminus \Sigma)^{\rm prime}} \rar[hook] \dar["{\varphi_{|(M \setminus \Sigma)^{\rm prime}}}"', hook]& 
        {Q_\Sigma}  \rar \dar["\cong"', "\psi"] & 
        M \dar["{\cong}"', "\varphi"] \\
        {(N \setminus \Upsilon)^{\rm prime}} \rar[hook] & 
        {Q_\Upsilon}  \rar & 
        N
    \end{tikzcd}\]
    commute and it is given by $\psi(p, [\gamma]) = (\varphi(p), [\varphi \circ \gamma])$.

    From \cref{lem:Sigma-Upsilon-primary} it follows that $\psi(P(\Sigma \subset M)) \supseteq P(\Upsilon \subset N)$ as submanifolds of $Q_\Upsilon$.
    Therefore we may set $P(\varphi) := (\psi^{-1})_{|P(\Upsilon \subset N)}$.
    This is unique because every point in $P(\Upsilon \subset N)$ can be connected by a path to a point in $\varphi((M\setminus \Sigma)^{\rm prime})$, and every path in $M$ that does lift to $P(\Sigma \subset M)$ lifts uniquely as that map is a local homeomorphism.

    It follows from the uniqueness that for two composable morphisms $\varphi_1$ and $\varphi_2$ in $\SepMfd$ we must have $P(\varphi_1) \circ P(\varphi_2) = P(\varphi_2 \circ \varphi_1)$, and hence $P(-)$ is functorial.
    Moreover, $P(\varphi)$ depends continuously (in the $C^\infty$-topology) on $\varphi$, as it must cover $\varphi^{-1}$.
\end{proof}

\subsection{The space of prime markings}
Recall the definition of the embedding space $\emb^{\scl{=}}(N,Q)$ in \cref{defn: equal-after-scl-embeddings}. As a consequence of \cref{prop - max prime piece is diffeo to P minus some discs}, the spherical closure of $P(\Sigma\subset M)$ is diffeomorphic to $P_M$ and hence the embedding space $\emb^{\scl =}(P(\Sigma\subset M),P_M)$ is non-empty. 

\begin{defn}\label{defn:prime-marking}
    A \emph{prime marking} of $\Sigma\subset M$ is an embedding $\iota\in \emb^{\scl=}(P(\Sigma\subset M),P_M)$.
\end{defn}

As a direct application of \cref{lem:embedding complement of discs}, we obtain the following corollary, stating that the space of prime markings is (in a suitable homotopical sense) a torsor over the diffeomorphism group of $P_M$.
\begin{cor}\label{cor:homotopy-orbits-contractible}
    Let $\Sigma \subset M$ be a sphere system with maximal prime piece $P(\Sigma \subset M)$ and fix an embedding $\iota \colon P(\Sigma \subset M) \hookrightarrow P_M$ such that $P_M\setminus \im(\iota)=\sqcup_k\interior{D}^3$.
    Then the map 
    \begin{align*}
        \Diff^+(P_M) &\to \emb^{\scl{=}}(P(\Sigma \subset M), P_M) \\
        \varphi &\mapsto \varphi \circ \iota.
    \end{align*}
    is a homotopy equivalence.
    Moreover, the homotopy orbit space
    \[
        \emb^{\scl{=}}(P(\Sigma \subset M), P_M) \hq \Diff^+(P_M)
    \]
    is contractible.
\end{cor}

 Applying \cref{lem:embedding complement of discs} to $(M\ca\Sigma)^{\rm{prime}}$, we similarly see that the map defined by acting with $\Diff^+(P_M)$ on any point of $\emb^{\scl{=}}((M\ca \Sigma)^{\rm prime}, P_M)$ is an equivalence.
 From this it follows that we also have an equivalence between mapping spaces.
 
\begin{cor}\label{cor:maximal-prime=prime}
    Let $\Sigma \subset M$ be a sphere system, then the restriction map
    \[
        \emb^{\scl{=}}(P(\Sigma \subset M), P_M) \to 
        \emb^{\scl{=}}((M\ca \Sigma)^{\rm prime}, P_M) 
    \]
    is a $\Diff^+(P_M)$-equivariant homotopy equivalence.
\end{cor}

We are now ready to define the space $\sep^\Gamma(M; P_M)$ of $\Gamma$-separating systems with a prime marking. 
Recall that for any separating system $\Sigma$, a diffeomorphism $\varphi\colon M\rightarrow M$ sends $\Sigma$ to another separating system $\Upsilon=\varphi(\Sigma)$ and induces an isomorphism $\varphi_G\colon G_\Sigma\rightarrow G_\Upsilon$. 
Moreover, $\varphi$ induces a diffeomorphism 
\[
\varphi_P\colon P(\Sigma\subset M) \longrightarrow P(\Upsilon\subset M),~(p,[\gamma])\mapsto (\varphi(p),[\varphi(\gamma)]),
\] 
which is the inverse of the diffeomorphism $P(\varphi)$ induces by the functoriality of the maximal prime piece (\cref{prop: restricting maximal prime piece}).
For $\Upsilon = \Sigma$ this map gives a group homomorphism $\Diff(M, \Sigma) \to \Diff(P(\Sigma \subset M))$.

\begin{defn}\label{defn: Sep-PM}
    For a graph $\Gamma \in \Gr(M)$ we define
    \[
        \sep^\Gamma(M; P_M)
        := \{ (\Sigma, \alpha, \iota) \;|\; (\Sigma,\alpha) \in \sep^\Gamma(M),~  \iota \in \emb^{\scl{=}}(P(\Sigma \subset M), P_M) \}.
    \]
    As $\sep^\Gamma(M)$ can be identified with the quotient of $\Diff^+(M)$ by the free action of $\smash{\Diff_{\pi_0(2\Sigma)}^+(M, \Sigma)}$, we topologise $\smash{\sep^\Gamma(M; P_M)}$ via the bijection
    \begin{align*}
        \frac{\Diff^+(M) \times \emb^{\scl{=}}( P(\Sigma_0\subset M), P_M)}{\Diff^+_{\pi_0(2\Sigma_0)}(M, \Sigma_0)}
        &\cong \sep^\Gamma(M; P_M)  \\
        [\varphi, \iota\colon P(\Sigma_0\subset M) \hookrightarrow P_M] 
        & \longmapsto (\varphi(\Sigma_0), \alpha_0 \circ \varphi_G^{-1}, \iota \circ \varphi_P^{-1})
    \end{align*}
     for any choice of $(\Sigma_0, \alpha_0) \in \sep^\Gamma(M)$. 
     There is an action of $\Diff^+(M)\times\Diff^+(P_M)$  on $\sep^\Gamma(M; P_M)$ via the formula \[(\varphi,\psi)\colon(\Sigma,\alpha,\iota)\mapsto (\varphi(\Sigma),\alpha\circ\varphi_G^{-1},\psi\circ \iota\circ \varphi_P^{-1}),\]  for $(\varphi,\psi)\in \Diff^+(M)\times\Diff^+(P_M)$.
\end{defn}
Functoriality of $\sep^\Gamma(M; P_M)$ now follows from \cref{prop: restricting maximal prime piece}.
\begin{defn}\label{defn: Sep-PM-functor}
    For a map of graphs $f\colon \Gamma \to \Gamma'$ we obtain an induced map
    \begin{align*}
        f_*\colon \sep^{\Gamma}(M; P_M) &\longrightarrow \sep^{\Gamma'}(M; P_M) \\
        (\Sigma, \alpha, \iota) & \longmapsto (\Sigma',\alpha',\iota \circ j),
    \end{align*}
    where $\Sigma' \subseteq \Sigma$ and $\alpha'$ are defined as for $\sep^\Gamma(M)$ in \cref{defn:Sep-functor} and $j\colon P(\Sigma' \subset M) \hookrightarrow P(\Sigma \subset M)$ is the embedding induced by the functoriality of the maximal prime piece (\cref{prop: restricting maximal prime piece}).
\end{defn}

\begin{rem}
    The action of  $\Diff^+(M) \times \Diff^+(P_M)$ on $\sep^\Gamma(M;P_M)$ by pre- and post-composition is compatible with the functoriality.
\end{rem}

\subsection{The splitting map} Since $\sep^\Gamma(M;P_M)$ admits commuting actions by $\Diff^+(M)$ and $\Diff^+(P_M)$, we will obtain a space with maps to both $\BDiff^+(M)$ and $\BDiff^+(P_M)$ after taking homotopy quotients. This will be the bridge that allows us to define the splitting map $\W\colon \BDiff^+(M)\rightarrow \BDiff^+(P_M)$. We first show that the marking by the maximal prime piece in $\sep^\Gamma(M; P_M)$ can be cancelled by quotienting out the $\Diff^+(P_M)$-action.
\begin{lem}\label{lem:sep(M,P)/Diff(P)}
    There is a functorial equivalence
    \[
        \sep^\Gamma(M; P_M)\hq \Diff^+(P_M) \simeq \sep^\Gamma(M)
    \]
\end{lem}
\begin{proof}
    The projection $\sep^\Gamma(M;P_M) \to \sep^\Gamma(M)$ that forgets the prime marking is functorial and invariant under the action of $\Diff^+(P_M)$, as this only acts on the marking.
    Therefore we obtain a functorial map:
    \[
       \rho\colon \sep^\Gamma(M; P_M) \times_{\Diff^+(P_M)} E\Diff^+(P_M) \to \sep^\Gamma(M).
    \]
    Writing out the left term in full, this map becomes
    \[
        \rho\colon\Diff^+(M) \times_{\Diff^+_{\pi_0(2\Sigma)}(M,\Sigma)} \emb^{\scl =}(P(\Sigma \subset M), P_M) \times_{\Diff^+(P_M)} E\Diff(P_M) \to \sep^\Gamma(M).
    \]
    By \cref{lem:Diff(M)->sep(M) principal}, $\Diff^+(M)\rightarrow \sep^\Gamma(M)=\Diff^+(M) \times_{\Diff^+_{\pi_0(2\Sigma)}(M,\Sigma)}\{*\}$ is a principal $\Diff^+_{\pi_0(2\Sigma)}(M,\Sigma)$-bundle. Therefore $\rho$ is a locally trivial fibre bundle with fibre
    \[
        \emb^{\scl =}(P(\Sigma \subset M), P_M) \times_{\Diff^+(P_M)} E\Diff^+(P_M),
    \]
    which is contractible by \cref{cor:homotopy-orbits-contractible}. Hence $\rho$ is an equivalence.
\end{proof}

We now construct a variant $\calF^\Gamma(P_M)$ of $\rmD(\Gamma)$ that incorporates a prime marking.
\begin{defn}\label{defn:C-Gamma}
    Define the functor $\calF^{(-)}(P_M)\colon \Gr(M)\to \Diff^+(P_M)\text{-}\Top$ by
    \[
        \calF^\Gamma(P_M) := \sep^\Gamma(M; P_M) \hq \Diff^+(M).
    \]
\end{defn}

This definition allows us to connect $\calF^\Gamma(P_M)$ with the simpler, but non-functorial, definition $\calF_0(\Gamma)$ that we considered at the start of the section.
\begin{lem}\label{lem:rmC-simpler}
    For any graph $\Gamma$ there are $\Diff^+(P_M)$-equivariant equivalences
    \begin{align*}
        \calF^\Gamma(P_M) &\simeq \emb^{\scl{=}}(P(\Sigma \subset M), P_M) \hq \Diff^+_{\pi_0(2\Sigma)}(M, \Sigma) \\
        &\simeq \emb^{\scl{=}}((M\ca \Sigma)^{\rm prime}, P_M) \hq \Diff^+_{\pi_0(2\Sigma)}(M, \Sigma)
    \end{align*}
    where the action arises via the homomorphism
    $ \Diff^+_{\pi_0(2\Sigma)}(M, \Sigma) \to \Diff^+( (M \ca \Sigma)^{\rm prime})$.
\end{lem}
\begin{proof}
    By definition $\calF^\Gamma(P_M)$ is
    \[
         \left(\Diff^+(M) \times_{\Diff^+_{\pi_0(2\Sigma)}(M)} \emb^{\scl{=}}( P(\Sigma\subset M), P_M) \right) \hq \Diff^+(M).
    \]
    Spelling out the homotopy orbit construction and commuting out the quotient by $\Diff^+_{\pi_0(2\Sigma)}(M)$ we get
    \[
         \left(E\Diff^+(M) \times \Diff^+(M) \times \emb^{\scl{=}}( P(\Sigma\subset M), P_M) \right) \big/ (\Diff^+(M) \times \Diff^+_{\pi_0(2\Sigma)}(M))
    \]
    The two copies of $\Diff^+(M)$ cancel and we get
    \[
         \emb^{\scl{=}}( P(\Sigma_0\subset M), P_M) \hq \Diff^+_{\pi_0(2\Sigma)}(M),
    \]
    proving the first equivalence. 
    The second equivalence now follows from Corollary \ref{cor:maximal-prime=prime}.
\end{proof}

Since $\calF^\Gamma(P_M)$ retains its $\Diff^+(P_M)$-action, we can further quotient by $\Diff^+(P_M)$ to obtain $\rmD(\Gamma)$.

\begin{lem}\label{lem: C/Diff=D}
    There is a functorial equivalence $\calF^\Gamma(P_M) \hq \Diff^+(P_M)  \simeq  \rmD(\Gamma)$.
\end{lem}
\begin{proof}
    Recall from \cref{defn:D-Gamma-functor} that $\rmD(\Gamma)= \sep^\Gamma(M) \hq \Diff^+(M)$. Now applying $(-)\hq \Diff^+(M)$ to Lemma \ref{lem:sep(M,P)/Diff(P)} gives the desired result.
\end{proof}
Theorem~\ref{thm: wrong way map} follows almost immediately from \cref{lem: C/Diff=D}.

\begin{thm}\label{thm: wrong way map}
    Let~$M$ be a 3-manifold that is not diffeomorphic to~$S^1\times S^2$, and let~$P_M$ be the disjoint union of the irreducible factors of~$M$. 
    Then there is a homotopy fibre sequence
    \[
    \hocolim_{\Gamma \in \Gr(M)} \calF^\Gamma(P_M) \longrightarrow \BDiff^+(M) \xrightarrow[\qquad]{\W} \BDiff^+(P_M).
    \]
\end{thm}
\begin{proof}
    For any $G$-space $X$ there is a fibre sequence:
    \[
        X \longrightarrow X\hq G \longrightarrow *\hq G = BG.
    \]
    Setting $X = \hocolim_{\Gamma \in \Gr(M)} \calF^{\Gamma}(P_M)$ and $G = \Diff^+(P_M)$, all that remains is to identify $X\hq G$ with $\BDiff^+(M)$.
    For this we compute:
    \begin{align*}
        \left(\hocolim_{\Gamma \in \Gr(M)} \calF^{\Gamma}(P_M)\right) \hq \Diff^+(P_M) 
        &\simeq \hocolim_{\Gamma \in \Gr(M)} \left(\calF^{\Gamma}(P_M)\hq \Diff^+(P_M)\right) \\
        &\simeq \hocolim_{\Gamma \in \Gr(M)} \rmD(\Gamma) \\
        & \simeq \BDiff^+(M).
    \end{align*}
    Here the first step commutes the homotopy colimit with the homotopy quotient as in the proof of \cref{thm:hocolim-intro} (using \cite[Lemma 2.12(1) and (5)]{BoydBregmanSteinebrunner-finiteness}).
    The second step is \cref{lem: C/Diff=D}
    and the third step is \cref{thm:hocolim-intro}.
\end{proof}

\subsection{A simplified colimit}
We now rewrite the homotopy colimit that describes the fibre of $\W$ in slightly different terms.
The key idea is that instead of $\Gr(M)$ we want to use a different, more combinatorial category $\Gr_{g,n}$ of rank $g$ graphs with $n$ marked vertices.
This has two advantages: 
\begin{itemize}
    \item 
    The category $\Gr_{g,n}$ only depends on the number of $(S^1\times S^2)$-prime factors and the number of irreducible prime factors -- it does not depend on which prime factors appear nor their multiplicity.
    \item     
    The new functor $\rmC^{(-)}_{(-)}(P_1,\dots,P_n)\colon \Gr_{g,n} \to \Top$ over which we take the homotopy colimit takes values in connected spaces, whereas $\calF^\Gamma(P_M)$ typically has several connected components, all of which are homeomorphic.
\end{itemize}

\begin{notation}
    Throughout this section, fix integers $g,n\geq 0$ and compact oriented irreducible $3$-manifolds $P_1,\dots,P_n$  not diffeomorphic to $D^3$ or $S^3$.
    Write $M = P_1 \sharp \dots \sharp P_n \sharp (S^1 \times S^2)^{\sharp g}$ and $P_M := P_1 \sqcup \dots \sqcup P_n$.
\end{notation}

Recall from \cref{defn:Gr_gn} that an element $(\Gamma, \sigma) \in \Gr_{g,n}$ is a connected rank $g$ graph $\Gamma$ with a marking $\sigma \colon \{1,\dots,n\} \hookrightarrow V_\Gamma$ such that every non-marked vertex is at least trivalent. 
We let $V_\Gamma^{\rm sph} \subset V_\Gamma$ denote the set of unmarked vertices.
Given a choice of $P_1,\dots,P_n$, there is a functor
\[
    \lambda\colon \Gr_{g,n} \longrightarrow \Gr(M) 
\]
which sends $(\Gamma,\sigma)$ to the graph $\Gamma$ with the $3$-manifold labelling that assigns $[P_i]$ to the vertex $\sigma(i)$ and $[S^3]$ to all unmarked vertices.

\begin{defn}\label{defn:C-Gamma-sigma}
    For $(\Gamma, \sigma) \in \Gr_{g,n}$ we let 
    \[ 
        \sep^\Gamma_\sigma(M; P_1,\dots, P_n) \subseteq \sep^\Gamma(M; P_1 \sqcup \dots \sqcup P_n)
    \]
    denote the subspace of those $(\Sigma \subset M, \alpha\colon G_{\Sigma \subset M} \cong \Gamma, \iota \colon P(\Sigma \subset M) \hookrightarrow P_M)$ where $\iota$ sends the component of $P(\Sigma \subset M)$ that corresponds to $\alpha^{-1}(\sigma(i))$ to $P_i$.
    We further let
    \[
        \rmC^\Gamma_\sigma(P_1,\dots,P_n) := \sep^\Gamma_\sigma(M; P_1,\dots, P_n) \sslash \Diff^+(M) \subseteq \calF^\Gamma(P_1\sqcup\dots \sqcup P_n)
    \]
    denote the corresponding subspace.
\end{defn}

Both $\sep^\Gamma$ and $\calF^\Gamma$ were functors in $\Gamma \in \Gr(M)$ and these subspaces $\sep^\Gamma_\sigma$ and $\rmC^\Gamma_\sigma$ are still functorial in the more restrictive category $\Gr_{g,n}$.
The canonical inclusions of the subspaces yield a natural transformation.
\[\begin{tikzcd}
	{\Gr(M)} &&& {\;}\\
	{\Gr_{g,n}} &&& \Top
	\arrow[""{name=0, anchor=center, inner sep=0}, "{\calF^{-}(P_M)}", from=1-1, to=2-4]
	\arrow["\lambda", from=2-1, to=1-1]
	\arrow[""{name=1, anchor=center, inner sep=0}, "{\rmC^{-}_{-}(P_1,\dots, P_n)}"', from=2-1, to=2-4]
	\arrow["\alpha \Uparrow", phantom, near start, from=2-1, to=1-4]
\end{tikzcd}\]

Now suppose we fix $\Gamma \in \Gr(M)$, a graph labelled by diffeomorphism classes of irreducible manifolds.
Then we call an injection $\sigma \colon \{1,\dots, n\} \hookrightarrow V_\Gamma$ \emph{compatible} if it sends $i$ to a vertex labelled by the diffeomorphism class $[P_i]$.
By construction of $\rmC^\Gamma_\sigma$, we can recover $\calF^\Gamma$ by taking the disjoint union over all choices of compatible $\sigma$.
That is, the combined map
\[
    \alpha^\Gamma = \amalg_\sigma \alpha^\Gamma_\sigma \colon 
    \coprod_{\substack{\sigma\colon \{1,\dots,n\} \hookrightarrow V_\Gamma \\ \text{compatible}}} 
    \rmC^\Gamma_\sigma(P_1,\dots, P_n) \cong
    \calF^{\Gamma}(P_1 \sqcup \dots \sqcup P_n)
\]
is a homeomorphism.
This implies that our new functor still has the same homotopy colimit.
\begin{lem}\label{lem:hocolim-on-Gr_gn}
    The natural transformation $\alpha$ induces a homeomorphism
    \[
        \hocolim_{(\Gamma,\sigma) \in \Gr_{g,n}} \rmC^\Gamma_\sigma(P_1,\dots,P_n) 
        \cong
        \hocolim_{\Lambda \in \Gr(M)} \calF^\Lambda(P_1\sqcup \dots \sqcup P_n).
    \]
\end{lem}   
\begin{proof}
    For the purpose of this proof, it will be useful to think of an object of $\Gr_{g,n}$ as a graph $\Gamma \in \Gr(M)$ together with an compatible map $\sigma \colon \{1,\dots,n\} \to V_\Gamma$.
    This only changes the category up to isomorphism, as the $3$-manifold marking on $\Gamma$ is uniquely determined by $\sigma$.
    
    Let $E_\bullet \to E_\bullet'$ be the map (induced by $\alpha$) between simplicial spaces that compute the two homotopy colimits as in \cref{defn - hocolim}.
    In simplicial degree $k$ this map is 
    \[
        \coprod_{(\Gamma_0,\sigma_0) \to \dots \to (\Gamma_k, \sigma_k)} \rmC^{\Gamma_0}_{\sigma_0}(P_1,\dots,P_n) 
        \longrightarrow
        \coprod_{\Gamma_0\to \dots \to \Gamma_k} \calF^{\Gamma_0}(P_1 \sqcup \dots \sqcup P_n).
    \]
    In the chain of maps $(\Gamma_0,\sigma_0) \xrightarrow{f_1} \dots \xrightarrow{f_k} (\Gamma_k, \sigma_k)$ the markings $\sigma_1, \dots, \sigma_k$ are completely determined by $\sigma_0$ as we must have $\sigma_i = f_i \circ \sigma_{i-1}$.
    Hence we can decompose the left term as
    \[
        \coprod_{\Gamma_0 \to \dots \to \Gamma_k} 
        \coprod_{\substack{\sigma\colon \{1,\dots,n\} \hookrightarrow V_\Gamma \\ \text{compatible}}} 
        \rmC^{\Gamma_0}_{\sigma_0}(P_1,\dots,P_n) 
        \longrightarrow
        \coprod_{\Gamma_0\to \dots \to \Gamma_k} \calF^{\Gamma_0}(P_1 \sqcup \dots \sqcup P_n),
    \]
    which is a homeomorphism by the display above the lemma.
    This shows that $E_\bullet \cong E_\bullet'$ and thus the homotopy colimits are homeomorphic.
\end{proof}

\begin{rem}
    From a categorical perspective, the above lemma follows from the fact that $\Gr_{g,n} \to \Gr(M)$ is (equivalent to) a left fibration whose fibre at $\Gamma$ is the finite set of compatible marking of $\Gamma$.
\end{rem}

\begin{defn}\label{defn:calH_g}
    For $P_1,\dots,P_n$ as before and $g\ge 0$, we define the space of \emph{genus $g$ $1$-handle attachments} to $P_1,\dots,P_n$ to be the homotopy colimit
    \[
        \calH_g(P_1,\dots,P_n) := \hocolim_{(\Gamma, \sigma) \in \Gr_{g,n}} \calH^\Gamma_\sigma(P_1,\dots,P_n).
    \]
\end{defn}

\begin{cor}\label{cor:fibre-sequence-H_g}
    For every compact, connected, oriented 3-manifold $M$ that is not diffeomorphic to~$S^1\times S^2$, there is a homotopy fibre sequence
    \[
        \calH_g(P_1,\dots,P_n) \longrightarrow \BDiff^+(M) \xrightarrow[\qquad]{\W} \BDiff^+(P_1 \sqcup \dots \sqcup P_n).
    \]
\end{cor}

\subsection{Geometric description of \texorpdfstring{$\W$}{W} and \texorpdfstring{$\pi_1$}{pi-1}-surjectivity} Since the existence of $\W$ in \cref{thm: wrong way map} was derived from the homotopy colimit, its construction may seem somewhat opaque.  The purpose of this section is to provide a more geometric interpretation of $\W$, and give some applications to mapping class groups. 

Our aim will be to explicitly describe the restriction of $\W$ to $\Diff^+(M,\Sigma)$. Consider the chain of homomorphisms \[\Diff^+(M,\Sigma)\to \Diff^+(M\ca\Sigma)\rightarrow \Diff^+((M\ca\Sigma)^{\rm prime})\]
obtained by cutting along $\Sigma$ and then restricting to the prime components. Smale's theorem \cite{Smale59} tells us we can extend any fixed element of $\Diff^+((M\ca\Sigma)^{\rm prime})$, or indeed any $S^d$-family of such diffeomorphisms over $\sqcup_kD^3$ to obtain a family of diffeomorphisms of $P_M$. Thus, we should expect to obtain a map on the level of classifying spaces: $\BDiff^+(M,\Sigma)\to \BDiff^+(P_M)$. To make this discussion more concrete, we introduce the following subgroup of $\Diff^+(M,\Sigma)$.

\begin{defn}\label{defn:O(3)-restriction-subgroup}
    Fix a collar $C_\Sigma\cong [-1,1]\times \Sigma$ of $\Sigma$, where $\Sigma$ is identified with $\Sigma\times\{0\}$.  If we equip $C_\Sigma$ with the product of the standard Euclidean metric on $[-1,1]$ and the round metric on each sphere $S\in \Sigma$, its  orientation-preserving isometry group can be written as 
    \begin{equation}\label{eqn:Collar-Isom-Sigma-Isom}
        \Isom^+(C_\Sigma)=\left(\prod_{S\in\pi_0(\Sigma)}\Or(3)\right)\rtimes \Sym_{\pi_0(\Sigma)}=\Isom(\Sigma)
    \end{equation} where  $\Sym_{\pi_0(\Sigma)}$ is the group of permutations of $\pi_0(\Sigma)$, and each orientation-reversing element of $\Or(3)$ acts in the usual way on $S\in\pi_0(\Sigma)$ but as $-1$ on $[-1,1]$.  Define $\Diff^{\Or(3)}(M, \Sigma) \subset \Diff(M,\Sigma)$ to be the subgroup of diffeomorphisms which agree with an element of $\Isom^+(C_\Sigma)$ on $C_\Sigma$.
\end{defn}

\begin{lem}\label{lem:O(3)-subgroup-equivalence}
    The inclusion $\Diff^{\Or(3)}(M,\Sigma)\hookrightarrow \Diff^+(M,\Sigma)$ is a homotopy equivalence.
\end{lem}
\begin{proof}
    Consider the following commutative diagram of fibre sequences:
    \[\begin{tikzcd}
    {\Diff^+_\Sigma(M)} \ar[r]& {\Diff^+(M,\Sigma)}\ar[r]& {\Diff(\Sigma)}\\
    {\Diff_{C_\Sigma}(M)}\ar[r]\ar[u, "\simeq"]& {\Diff^{\Or(3)}(M,\Sigma)}\ar[r]\ar[u]& {\Isom(\Sigma)}\ar[u, "\simeq"]
    \end{tikzcd}\]
    Both the top and bottom horizontal sequences come from restriction to $\Sigma$, while the vertical maps are all inclusions. In the bottom sequence, we have identified the fibre with those diffeomorphisms that act trivially on $C_\Sigma$ and the image with $\Isom(\Sigma)$ via the isomorphism in \cref{eqn:Collar-Isom-Sigma-Isom}.

    By Smale's theorem, the right hand vertical map is a homotopy equivalence.  In general, neither $\Diff^+(M,\Sigma)\to\Diff(\Sigma)$ nor $\Diff^{\Or(3)}(M,\Sigma)\to\Isom(\Sigma)$ will be surjective since only those permutations induced by $\Aut(G_\Sigma)$ can be realised.  Nevertheless, they both surject onto exactly the same components, and hence the middle vertical map will be a homotopy equivalence as long as the left hand vertical map is a homotopy equivalence.

    To see this, let $\Coll_\Sigma := \emb^+_{\Sigma \times \{0\}}(\Sigma \times [-1,1], M)$ be the space of collars of $\Sigma$, \emph{i.e.}~the space of those orientation-preserving embeddings $\Sigma\times[-1,1]\hookrightarrow M$ that are the identity on $\Sigma \cong \Sigma \times \{0\}$.
    This space is contractible by a result of Cerf \cite[5.2.1, Corollaire 1]{Cerf}.
    The action of $\Diff_\Sigma^+(M)$ on $\Coll_\Sigma$ is locally retractile and transitive, hence 
    by taking the stabiliser at some collar $C_\Sigma \in \Coll_\Sigma$ we obtain a fibre sequence
    \[\Diff_{C_\Sigma}(M)\hookrightarrow \Diff_\Sigma^+(M)\rightarrow \Coll_{C_\Sigma}(\Sigma).\]
    where the base is contractible and hence the fibre inclusion is an equivalence, as desired.
\end{proof}

We now define a homomorphism $\Phi\colon\Diff^{\Or(3)}(M,\Sigma)\to \Diff^+(P_M)$ by first cutting to get $M\ca \Sigma$, capping each $S^2$ component of the boundary with a round $D^3$ to get $\smash{\widehat{M\ca\Sigma}}$, then forgetting $S^3$ components to get $P_M$. The fact that each element of $\Diff^{\Or(3)}(M,\Sigma)$ agrees with an element of $\Or(3)$ in a neighbourhood of $\Sigma$ guarantees that this extension is smooth. 
See \cref{fig:cutting-and-capping} in the introduction.

\begin{defn}
    We define the \emph{cut-and-cap map} $\kappa\colon \BDiff^+(M,\Sigma)\to \BDiff^+(P_M)$ to be the composition of the equivalence  $\BDiff^+(M,\Sigma)\simeq \BDiff^{\Or(3)}(M,\Sigma)$ from \cref{lem:O(3)-subgroup-equivalence} with the induced map $B\Phi\colon \BDiff^{\Or(3)}(M,\Sigma)\rightarrow \BDiff^+(P_M)$. 
\end{defn}
The next proposition gives a geometric interpretation to the restriction of the splitting map $\W$ to $\BDiff^+(M,\Sigma)$ in terms of the capping map $\kappa$:

\begin{prop}\label{prop:geometric-wrong-way-map}
    The following diagram commutes up to homotopy
    \[
        \begin{tikzcd}
        \BDiff^+(M,\Sigma)\ar[dr, "\kappa"] \ar[d]&\\\BDiff^+(M)\ar[r, "\W"] &\BDiff^+(P_M).
        \end{tikzcd}
     \]
\end{prop}
\begin{proof}
To define the splitting map, we introduced the intermediate space $\sep^\Gamma(M;P_M)$ which has an action of $\Diff^+(M)\times\Diff^+(P_M)$. From the construction of the homomorphism $\Phi$, we have a fixed $\Phi$-equivariant embedding $\iota_0\colon(M\ca\Sigma)^{\rm prime}\rightarrow P_M$.  By \cref{lem:rmC-simpler}, there are equivalences \[\rmD(\Gamma)\simeq\calF^\Gamma(P_M)\hq\Diff^+(P_M)\simeq \emb^{\scl=}((M\ca\Sigma)^{\rm prime},P_M)\hq(\Diff^+_{\pi_0(2\Sigma)}(M,\Sigma)\times\Diff^+(P_M)). \]It follows that we get a homotopy commutative diagram:
\[\begin{tikzcd}
{\BDiff^{\Or(3)}(M,\Sigma)\simeq\ast\hq\Diff^{\Or(3)}(M,\Sigma)} \ar[dr, "{*\hq\Phi}"] \ar[d, "{\iota_0\hq(\rm{Id}\times\Phi)}"'] & \\
{\emb^{\scl=}((M\ca\Sigma)^{\rm prime},P_M)\hq(\Diff^+_{\pi_0(2\Sigma)}(M,\Sigma)\times\Diff^+(P_M))} \ar[r] & 
{\ast\hq\Diff^+(P_M)=\BDiff(P_M)}\\
{\rmD(\Gamma)} \ar[u, "{\text{\cref{lem:rmC-simpler}}~\simeq}"] \ar[ur]\ar[d] & \\
{\BDiff^+(M)\simeq\displaystyle \hocolim_{\Gamma \in \Gr(M)} \rmD(\Gamma)} \ar[uur, "\W"] & 
\end{tikzcd}\]
By definition, the equivalence $\BDiff^+(M,\Sigma)\simeq\BDiff^{\Or(3)}(M,\Sigma)$ followed by $B\Phi=\ast\hq\Phi$ in the diagram is $\kappa$, while going the other way around the diagram is the map $\BDiff^+(M,\Sigma)\rightarrow \BDiff^+(M)$ induced by inclusion followed by $\W.$
\end{proof}

Our main application of the previous proposition will be to mapping class groups.
Recall that for $N$ an oriented manifold, the \emph{mapping class group} of $N$ is the group of components $\pi_0(\Diff^+(N))$.

Because $P_M$ is a disjoint union, we can decompose $\Diff^+(P_M)$, and hence the mapping class group of $P_M$, as follows. Partition $P_M=(\sqcup^{n_1}Q_1)\sqcup\cdots\sqcup(\sqcup^{n_k} Q_k)$ where for $i\neq j$, $Q_i$ is not orientation-preserving diffeomorphic to $Q_j$. This yields a product decomposition \begin{equation}\label{eqn:Diff-product-decomp}
    \Diff^+(P_M)\cong \prod_{i=1}^k\left(\left(\prod^{n_1}\Diff^+(Q_i)\right)\rtimes \Sym_{n_i}\right),
\end{equation}
and hence a corresponding decomposition of mapping class groups.

Since $\pi_0(\Diff^+(M))\cong \pi_1(\BDiff^+(M))$, Proposition \ref{prop:geometric-wrong-way-map} will imply the following surjectivity result for mapping class groups.

\begin{cor}\label{cor:W-pi1-surjective}
    The splitting map induces a surjection
    \[
        \pi_1\W\colon \pi_0(\Diff^+(M)) \twoheadrightarrow \pi_0(\Diff^+(P_M))
    \]
    from the mapping class group of $M$ onto the mapping class group of $P_M$.
\end{cor}
\begin{proof}
    Let $P_M=(\sqcup^{n_1}Q_1)\sqcup\cdots\sqcup(\sqcup^{n_k} Q_k)$ be as above. In each $Q_i$, choose a disc $D_i\cong D^3$ as well as a collar $  C_i\cong D^3$ containing $D_i$. Consider the subgroup $H\leq \Diff^+(P_M)$ defined in terms of the decomposition in \cref{eqn:Diff-product-decomp} by  \[H= \prod_{i=1}^k\left(\left(\prod^{n_1}\Diff_{C_i}^+(Q_i)\right)\rtimes \Sym_{n_i}\right).\]
    Since every diffeomorphism of $Q_i$ can be isotoped to fix $C_i$ pointwise, the inclusion of $H$ into $\Diff^+(P_M)$ is $\pi_0$-surjective. 
    
    Let $n=\sum_{i=1}^kn_i$ and let $M_0=(S^1\times S^2)^{\sharp^{g}}\setminus (\sqcup_n\interior{D}^3).$ We can build $M$ by attaching $n_i$ copies of $Q_i\setminus \interior{D}_i$ along boundary spheres of $M_0$ for $1\leq i\leq k$. Let $\Sigma$ be the separating system formed by a choice sphere in each $(S^1\times S^2)$-factor together with the $n$ spheres coming from the $\partial D_i$. In particular, $(M\ca\Sigma)^{\rm prime}=\sqcup_{i=1}^k(\sqcup_{n_i}(Q_i\setminus \interior{D}_i))$. 
    
    Now choose a collar $C_\Sigma$ of $\Sigma$ that agrees with $C_i\setminus \interior{D}_i$ on each prime component. With respect to $C_\Sigma$, we obtain a homomorphism $\Phi\colon \Diff^{\Or(3)}(M,\Sigma)\rightarrow \Diff^+(P_M)$. Let $\widehat{H}\leq\Diff^{\Or(3)}(M,\Sigma)$ consist of those diffeomophisms that restrict to  $C_\Sigma$ as  a permutation in $S_{\pi_0(\Sigma)}\leq \Isom^+(C_\Sigma)$. Since every such permutation can be achieved, $\Phi$ maps $\widehat{H}$ surjectively onto $H$, and hence $\pi_0(\widehat H)$ surjects onto $\pi_0(\Diff^+(P_M))$. On the other hand, by \cref{prop:geometric-wrong-way-map}, the restriction of $\W$ to $B\widehat{H}$ agrees with $B\Phi$, so $\pi_1\W$ is surjective.
\end{proof}

\subsection{The splitting map and boundary}
In this subsection we construct variants of the splitting map that fix parts of the boundary of $M$.
All of the boundary of $M$ comes from its prime pieces, so we have an identification $\partial M = \partial P_M$.
For every submanifold $A \subset \partial M = \partial P_M$ we will construct a variant of the splitting map 
\[
    \W_A \colon \BDiff_A(M) \longrightarrow \BDiff_A(P_M)
\]
and show that its homotopy fibre is equivalent to that of $\W$ and can hence be described by the same homotopy colimit.
The key observation that allows us to do this is that $\W$ commutes with restriction to the boundary in the following sense.

\begin{prop}\label{prop:W-on-boundary}
    There is a homotopy commutative square 
    \[\begin{tikzcd}
        {\BDiff^+(M)} \rar["{\W}"] \dar["{\partial}"] & {\BDiff^+(P_M)} \dar["{\partial}"] \\
        {\BDiff^+(\partial M)} \rar[equal] & {\BDiff^+(\partial P_M)} 
    \end{tikzcd}\] 
    where the vertical arrows restrict to the boundary and the bottom equality comes from the identification $\partial M = \partial P_M$. 
\end{prop}
\begin{proof}
    Let $\calD$ be the space of orientation-preserving diffeomorphisms from $\partial M$ to $\partial P_M$.
    Given our canonical identification $\partial M = \partial P_M$ we can identify $\calD$ with either $\Diff^+(\partial M)$ or $\Diff^+(\partial P_M)$.
    However, it will be convenient to think of it not as a group but rather as a space with commuting actions of $\Diff^+(\partial M)$ and $\Diff^+(\partial P_M)$.
    For every $\Gamma \in \Gr(M)$ we have a $\Diff^+(M) \times \Diff^+(P_M)$-equivariant map
    \[
        \delta \colon \sep^\Gamma(M; P_M) \longrightarrow \calD 
    \]
    defined by restricting the maximal prime marking $\iota\colon P(\Sigma \subset M) \hookrightarrow P_M$ to the diffeomorphism it induces between $\partial M \subset \partial P(\Sigma \subset M)$ and $\partial P_M$. 
    Taking homotopy orbits by $\Diff^+(M) \times \Diff^+(P_M)$ on the source and by $\Diff^+(\partial M) \times \Diff^+(\partial P_M)$ on the target, we obtain a map (natural in $\Gamma$)
    \[
        \delta'\colon \calF^\Gamma(P_M) \sslash \Diff^+(P_M) = \sep^\Gamma(M; P_M) \sslash (\Diff^+(M) \times \Diff^+(P_M)) 
        \longrightarrow \calD\sslash (\Diff^+(\partial M) \times \Diff^+(\partial P_M)).
    \]
    Note that the target of this map is equivalent both to $*\sslash \Diff^+(\partial M)$ and to $*\sslash\Diff^+(\partial P_M)$ by sending $\calD \to *$ and projecting to one of the factors.
    This $\delta'$ fits into a commutative square
    \[\begin{tikzcd}
        \rmD(\Gamma) \dar["{\partial}"'] & {\calF^\Gamma(P_M) \sslash \Diff^+(P_M)} \lar["\simeq"'] \rar \dar["{\delta'}"'] & {*\sslash \Diff^+(P_M)} \dar["{\partial}"] \\
        {* \sslash \Diff^+(\partial M)} & 
        {\calD \sslash (\Diff^+(\partial M) \times \Diff^+(\partial P_M))} \rar["{\simeq}"] \lar["{\simeq}"'] & {* \sslash \Diff^+(\partial P_M)} .
    \end{tikzcd}\] 
    Taking homotopy colimits over $\Gamma \in \Gr(M)$ the top row recovers the zig-zag that defines $\W$ while the bottom row recovers the equivalence between $\BDiff^+(\partial M)$ and $\BDiff^+(\partial P_M)$.
\end{proof}

We let $\Diff^{+}(\partial M, [M]) \le \Diff^+(\partial M)$ denote the subgroup of those diffeomorphisms that can be extended to the interior. 
This is always a union of path components, as any isotopy on $\partial M$ can be pushed to a collar of the boundary in $M$.
The advantage of this subgroup is that we have a homotopy fibre sequence
\[
    \BDiff^+_\partial(M) \longrightarrow 
    \BDiff^+(M) \longrightarrow 
    \BDiff^{+}(\partial M, [M]) .
\]
Note that a diffeomorphism of $\partial M = \partial P_M$ extends to $M$ if and only if it extends to $P_M$:
this is the claim that the vertical maps in \cref{prop:W-on-boundary} have the same image in $\pi_1$, which is indeed true as $\W$ is $\pi_1$-surjective by \cref{cor:W-pi1-surjective}.
Therefore, the square from \cref{prop:W-on-boundary} lifts to a square
    \[\begin{tikzcd}
        {\BDiff^+(M)} \rar["{\W}"] \dar["{\partial}"] & {\BDiff^+(P_M)} \dar["{\partial}"] \\
        {\BDiff^{+}(\partial M, [M])} \rar[equal] & {\BDiff^{+}(\partial P_M, [P_M]).} 
    \end{tikzcd}\] 
We can now use this modified square to obtain the variant $\W_A$ of the splitting map where we fix a submanifold of the boundary.

\begin{cor}\label{cor:splitting-fixing-A}
    For any submanifold $A \subset \partial M$ of the boundary, there is a variant $\W_A$ of the splitting map for diffeomorphisms that fix $A$, and there is a map of homotopy fibre sequences
    \[
    \begin{tikzcd}[column sep = large]
        {\mathrm{hofib}(\W_A)} \rar \dar["\simeq"] &
        {\BDiff_A^+(M)} \rar["{\W_A}"] \dar &
        {\BDiff_A^+(P_M)} \dar \\
        {\calH_g(P_1,\dots,P_n)} \rar &
        {\BDiff^+(M)} \rar["{\W}"] &
        {\BDiff^+(P_M)}
    \end{tikzcd}
    \]
    such that the map on fibres is an equivalence.
\end{cor}
\begin{proof}
Let $\Diff^+_A(\partial M, [M]) \le \Diff^+(\partial M)$ denote the subgroup of those diffeomorphisms that fix $A$ pointwise and also can be extended to the interior of $M$, and similarly for $P_M$.
Consider the cube
\[\begin{tikzcd}
	{\BDiff^+_A(M)} && {\BDiff_A^+(\partial M, [M])} & \\
	& {\BDiff^+(M)} && {\BDiff^+(\partial M, [M])} \\
	{\BDiff^+_A(P_M)} && {\BDiff^+_A(\partial P_M, [P_M])} \\
	& {\BDiff^+(P_M)} && {\BDiff^+(\partial P_M, [P_M])}
	\arrow[from=1-1, to=1-3]
	\arrow[from=1-1, to=2-2]
	\arrow["{\W_A}"', dashed, from=1-1, to=3-1]
	\arrow[from=1-3, to=2-4]
	\arrow[equals, from=1-3, to=3-3]
	\arrow[equals, from=2-4, to=4-4]
	\arrow[from=3-1, to=3-3]
	\arrow[from=3-1, to=4-2]
	\arrow[from=3-3, to=4-4]
	\arrow[from=4-2, to=4-4]
	\arrow["\W"'{pos=0.3}, from=2-2, to=4-2, crossing over]
	\arrow[from=2-2, to=2-4, crossing over]
\end{tikzcd}\]
where the front face is obtained from \cref{prop:W-on-boundary} and the top and bottom faces are given by restrictions and subgroup inclusions.
    Both the top face and the bottom face are homotopy pullback squares, as can be seen by comparing their horizontal homotopy fibers:
    In the case of the top face both horizontal homotopy fibers are $\BDiff_\partial^+(M)$ and for the bottom face both are $\BDiff_\partial^+(P_M)$.
    We therefore obtain the dashed map $\W_A$ as the induced map on the homotopy pullbacks.

    Finally, consider all the vertical homotopy fibres of the cube: these also must form a homotopy pullback square.
    The homotopy fibres of the two vertical maps on the right are contractible, so it follows that the homotopy fibres of $\W_A$ and $\W$ are equivalent, as claimed.
\end{proof}

\section{Describing the fibre of the splitting map}\label{sec:describing the fibre}
The purpose of this section is to describe $\rmC^\Gamma_\sigma$, and thus the fibre of the splitting map, in more detail.
We will give a concrete formula for $\rmC^\Gamma_\sigma(P_1,\dots,P_n)$ as the homotopy quotient of an action of a compact Lie group on a product of framed configuration spaces.
For $n>0$ we will see that this action is free and argue that the homotopy quotient is equivalent to the strict quotient, which admits a compactification as a manifold (with corners).
This will allow us to conclude that for $n>0$ the fibre of the splitting map is homotopy finite.
For general $n$, we will also use the formula for $\rmC^\Gamma_\sigma$ to exhibit it as a homotopy pullback and thus set up an explicit spectral sequence converging to its cohomology.

Given $(\Gamma,\sigma)\in \Gr_{g,n}$ and oriented irreducible 3-manifolds $P_1,\ldots, P_n$, our goal will be to parametrise the minimal data necessary to prescribe a gluing of the $P_i$ and copies of $S^3$ in a pattern described by $(\Gamma,\sigma)$. We think of each vertex $v\in V_\Gamma$ as being labelled by some $P_i$ or $S^3$ according to $\sigma$. Then for each edge $e\in E_\Gamma$, a gluing is specified by a choice of two points in the manifold corresponding to the endpoints of $e$, together with an orientation-reversing isomorphism between the tangent spaces at these chosen points. 

We are thus led to define a space of framed configurations (\cref{defn:action-on-Conf}):
\[
    \Conf^{\fr+}_{(\Gamma,\sigma)}(P_1,\ldots,P_n):=\prod_{i=1}^n \Conf_{H_{\sigma(i)}}^{\fr+}(P_i) \times \prod_{v \in V_\Gamma^{\rm sph}} \Conf_{H_v}^{\fr+}(S^3),
\]
where $\Conf^{\fr+}_A(M)$ denotes the space of positively framed configurations of a finite set $A$ in $M$, and for $v\in V_\Gamma$, $H_v$ is the set of half-edges at $v$. This configuration space comes naturally equipped with an action of a group 
\[
    \left(\Aut(\Gamma,\sigma)\times\prod_{i=1}^n\Diff^+(P_i)\right) \ltimes\SO(\Gamma,\sigma),
\]
where  $\SO(\Gamma,\sigma)$ is the compact Lie group $ \smash{\prod_{E_\Gamma}\SO(3)\times\prod_{V_{\Gamma}^{\rm sph}}\SO(4)}$.
(More generally, this extends to an action of $\prod_{E_\Gamma} \GL_3^+(\bbR) \times \prod_{V_\Gamma^{\rm sph}} \Diff^+(S^3)$, but we act with the smaller group so that later we can restrict the action to orthogonal frames.)
The finite group $\Aut(\Gamma,\sigma)$ acts by permuting the framed points via their corresponding half-edges, the $\Diff^+(P_i)$ and $\SO(4)$-factors act on the framed points in each $P_i$ and copy of $S^3$, respectively, and the $\SO(3)$-factors act by simultaneously rotating the framings at either end of an edge. Acting by $\SO(3)\simeq \Diff^+(S^2)$ identifies equivalent gluings along each edge, while acting by $\SO(4)\simeq \Diff^+(S^3)$ accounts for equivalent configurations in the $S^3$-factors. After modding out by the action of $\SO(\Gamma,\sigma)$, we obtain:

\begin{prop}\label{prop:C-Gamma configuration model}
    For all $(\Gamma,\sigma)$ and $P_1,\dots, P_n$ as above, 
    there is an $\Aut(\Gamma,\sigma) \times \prod_{i=1}^n \Diff^+(P_i)$ equivariant zig-zag of homotopy equivalences
    \[
        \rmC^{\Gamma}_\sigma(P_1,\dots,P_n) 
        \simeq \dots \simeq \Conf^{\fr+}_{(\Gamma,\sigma)}(P_1,\ldots,P_n)\hq\SO(\Gamma,\sigma)
    \]
\end{prop}

Together with \cref{cor:fibre-sequence-H_g} this implies our main theorem.
\begin{thm}\label{thm:fibre-sequence-intro}
    For every compact, connected, oriented 3-manifold $M$ that is not diffeomorphic to~$S^1\times S^2$, there is a homotopy fibre sequence
    \[
    \hocolim_{(\Gamma, \sigma) \in \Gr_{g,n}} \rmC_\sigma^\Gamma(P_1, \dots, P_n) \longrightarrow \BDiff^+(M) \xrightarrow[\qquad]{\W} \BDiff^+(P_1 \sqcup \dots \sqcup P_n)
    \]
    where for each $(\Gamma, \sigma) \in \Gr_{g,n}$ there is a homotopy equivalence
    \[
        \rmC_\sigma^\Gamma(P_1,\dots,P_n) \simeq \Conf^{\fr+}_{(\Gamma,\sigma)}(P_1,\dots,P_n) \sslash 
        \prod_{E_\Gamma}\SO(3)\times\prod_{V_{\Gamma}^{\rm sph}}\SO(4).
    \]
\end{thm}

We will prove \ref{prop:C-Gamma configuration model} in subsection \ref{subsec:proof-of-config}, but first we record some consequences.

\begin{cor}\label{cor:C-connected}
    The space $\rmC^\Gamma_\sigma(P_1,\dots,P_n)$ is connected.
\end{cor}
\begin{proof}
    The space of \emph{positively} framed configurations on a connected oriented $3$-manifold is always connected and hence so is the homotopy quotient in \cref{prop:C-Gamma configuration model}.
\end{proof}

This yields another proof of \cref{cor:W-pi1-surjective}. 
\begin{cor}
    The fibre of the splitting map $\W$ is connected, and thus $\pi_1\W$ is surjective.
\end{cor}
\begin{proof}
    By \cref{thm: wrong way map} and \cref{lem:hocolim-on-Gr_gn} the fibre of $\W$ is a homotopy colimit, over $\Gr_{g,n}$, of the spaces $\rmC^\Gamma_\sigma$, which are connected by \cref{cor:C-connected}.
    Since $\pi_0$ of a homotopy colimit is the colimit of the $\pi_0$'s, the connectedness of the fibre follows from the connectedness of $\Gr_{g,n}$ in \cref{lem:depth-of-Gr_gn}(iii).
    Surjectivity of $\pi_1\W$ then results from the long exact sequence of homotopy groups.
\end{proof}

It turns out that when $M$ has at least one prime factor that is not $S^1 \times S^2$, \emph{i.e.}~when there is at least one $P_i$, the configuration space model is particularly well-behaved. More precisely, we will show that the fibre of  $\W$ is homotopy finite in this case. This applies to the majority of 3-manifolds; indeed, the only cases not covered occur when $M = U_g = (S^1 \times S^2)^{\sharp g}$ and hence $P_M=\emptyset$. However, in these cases the target of $\W$ is a point, whence the fibre of the splitting map is just all of $\BDiff^+(U_g)$. At the other extreme, when $g=0$ and $n=1$, $M=P_M$ is irreducible and the fibre of $\W$ is a point, hence finite. For all remaining cases, we have:
\begin{prop}\label{prop:Conf-strict-quotient}
    For $n>0$ and $n+g>1$, the action in \cref{defn:action-on-Conf} is free and principal, and hence its strict quotient is equivalent to the homotopy quotient.
    This strict quotient is homotopy equivalent to a compact manifold with corners of dimension at most $15(g-1)+12n$  and thus to a finite cell complex.
\end{prop}
As a consequence, we obtain the following finiteness result:
\begin{cor}\label{cor:fibre-finite-and-dimension-bound}
    For $n>0$ and $n+g>1$, 
    the fibre $\calH_g(P_1,\dots,P_n)$ of the splitting map $\W$ is homotopy equivalent to a finite cell complex of dimension at most $17(g-1)+13n$.
\end{cor}
\begin{proof}
    According to \cref{thm: wrong way map} and \cref{lem:hocolim-on-Gr_gn}, the fibre of the $\W$ can be written as a homotopy colimit of spaces $\rmC_{\sigma}^\Gamma(P_1,\ldots,P_n)$ over the finite EI-category $\Gr_{g,n}$. The depth of $\Gr_{g,n}$ is $2(g-1)+n$ by \cref{lem:depth-of-Gr_gn}(i), while the dimension of $\rmC^\Gamma_\sigma(P_1,\ldots,P_n)\hq\Aut(\Gamma,\sigma)$ is at most $15(g-1)+12n$ by \cref{prop:C-Gamma configuration model} and  \cref{prop:Conf-strict-quotient}. Combined with the estimate in \cref{prop:EI-finite}, this yields a bound of $17(g-1)+13n$ for the dimension of the fibre. 
\end{proof} 

To illustrate this finiteness, consider the case where all of the $P_i$ are hyperbolic manifolds.
\begin{ex}\label{ex:hyperbolic}
    If $P_1,\dots,P_n$ are hyperbolic manifolds, then $\Diff^+(P_M)$ is equivalent to the finite groups of isometries $\Isom(P_M)$ by work of Gabai \cite{Gabai01}.
    Therefore the ``fibre inclusion'' map from $\calH_g(P_1,\dots,P_n)$ to $\BDiff(M)$ is equivalent to a finite cover with fibre $\Isom(P_M)$.
    Together with \cref{cor:fibre-finite-and-dimension-bound}, this shows that $\BDiff(M)$ admits a finite cover that is equivalent to a finite CW complex of dimension at most $17(g-1)+13n$.
\end{ex}

\subsection{Connected sums of two prime manifolds}
To underscore the applicability of \cref{prop:C-Gamma configuration model}, we explicitly describe the homotopy type of $\rmC_g(P_1,\ldots, P_n)$ in the two simplest situations where $n>0$, \emph{i.e.}~when $n=2,~g=0$ and $n=1,~g=1$. 
Applying \cref{prop:C-Gamma configuration model}, in both of these cases the homotopy colimit is taken over one labelled graph, hence $\rmC_{g}(P_1,\ldots,P_n)$ is equivalent to one of the compact manifolds given in \cref{prop:Conf-strict-quotient}. 

Previously, Hatcher \cite{Hatcher1981,Hatcher1981revised} proved that when $M=P_1\sharp P_2$, the space $\sepNP(M)$ is contractible, which can be leveraged to get information about the homotopy type and compute cohomology of $\BDiff^+(M)$ in this case; for instance, see work of Lelkes \cite{Lelkes-lens-spaces}. In contrast,  our description for $M=P\sharp (S^1\times S^2)$ seems to be wholly new. We will also get a complete description of $\BDiff_0(M)$ (and hence, after looping, $\Diff_0(M)$), provided that the irreducible factors are aspherical and not homeomorphic to $S^1\times D^2$. In this setting, we will extend some results of Kalliongis--McCullough \cite{KalliongisMcCullough} on $\pi_2(\BDiff_0(M))$ in the case where $M$ has two prime factors. For the $\BDiff_0(M)$ computation, we will need the following lemma:
\begin{lem}\label{lem:aspherical-quotient}
    If $P$ is aspherical and not diffeomorphic to $S^1\times D^2$, then $P\hq \Diff_0(P)$ is aspherical 
    and its fundamental group is the quotient of $\pi_1(P)$ by its centre.
\end{lem}
\begin{proof}
    Since $P$ is aspherical it is irreducible. By results of Hatcher \cite{Hatcher76} and Ivanov \cite{Ivanov76} in the Haken case, and  Gabai \cite{Gabai01} and McCullough--Soma \cite{McCulloughSoma13} in the non-Haken case,  $\Diff_0(P)\simeq \SO(2)^r$ where $0\leq r\leq 3 $ is equal to the rank of the centre of $\pi_1(P)$ (except when $P=S^1\times D^2$, which is in part why this is excluded).

    Let $\pi=\pi_1(P)$ and let $Z\cong \Z^r$ denote the centre of $\pi$.  If $r=0$, then $P\hq \Diff_0(P)\simeq P$, which is aspherical by assumption. If $r>0$, then $P$ is Seifert-fibred and there is a free $\SO(2)^r$-action on $P$.  When $r=1$, apart from two cases, there is a unique fibring up to isotopy and $\SO(2)$ acts by a full rotation of the generic fibre, which represents $Z\leq \pi$. The two exceptional cases occur when $P=\Kx$, the twisted $I$-bundle over the Klein bottle, and its double $D(\Kx)$. For $\Kx$, the $\SO(2)$-action extends the $\SO(2)$-action on $K$, and represents a full rotation along the central element in $\pi_1(K)$. Similarly, for $D(\Kx)$, the $\SO(2)$-actions on each $\Kx$ match along the boundary, hence glue up to an $\SO(2)$-action on the double which rotates along the common central elements (which have been identified) on both sides.
    When $r=2$, the only possibility is $T^2\times I$ and for $r=3$, the only possibility is $T^3$, with $\SO(2)^r$ acting in the obvious way. 
    
    In each case, the $\SO(2)^r$-action is covered by a free and proper $\bbR^r$-action on the universal cover $\smash{\widetilde{P}}$, which commutes with the action of $\pi$ and satisfies $\pi\cap \bbR^r=Z.$  In particular, as $\widetilde{P}$ is contractible, so is the quotient $\widetilde{F}:=\widetilde{P}/\bbR^r$, which inherits a properly discontinuous action by $\pi':=\pi/Z$.   Therefore we obtain:
    \[P\hq \Diff_0(P)\simeq P\hq \SO(2)^r\simeq\widetilde{P}\hq(\bbR^r\cdot\pi)\simeq\widetilde{F}\hq \pi'\simeq *\hq \pi'=K(\pi',1).\qedhere\]
\end{proof}

\begin{ex}[$n=2$, $g=0$]\label{ex:n=2-g=0} Up to labelled graph isomorphism, $\Gr_{2,0}$ has exactly one graph $(\Gamma,\sigma)$, consisting of a single edge with both endpoints labelled.  Thus, $\Gr_{2,0}$ has no non-trivial morphisms, and in particular, $\Aut(\Gamma,\sigma)=\{1\}$. If $M=P_1\sharp P_2$, then, after choosing a framing on the $P_i$,
\[\Conf^{\fr+}_{(\Gamma,\sigma)}(P_1,P_2) \simeq\Conf_1^{\fr+}(P_1)\times\Conf_1^{\fr+}(P_2)\simeq (P_1\times \SO(3))\times (P_2\times \SO(3)),\] while $\SO(\Gamma,\sigma)\cong \SO(3)$, which acts diagonally on the two framings on the right. 
Consequently, after choosing a framing on the $P_i$, we have: 
\[\rmC_0(P_1,P_2)=\rmC^\Gamma_\sigma(P_1,P_2)\simeq\frac{\Conf_1^{\fr+}(P_1)\times\Conf_1^{\fr+}(P_2)}{\SO(3)}\\
\simeq P_1\times P_2\times \SO(3),\]
where the resulting $\SO(3)$-factor records the difference of the two framings. The prime decomposition fibre sequence is: \[P_1\times P_2\times \SO(3)\hookrightarrow\BDiff^+(P_1\sharp P_2)\xrightarrow[]{\W}\BDiff^+(P_1\sqcup P_2),\]
which has fibre a compact 9-manifold.

Recall now that  $\BDiff_0(M)$ is the universal cover of $\BDiff^+(M)$. By \cref{prop:C-Gamma configuration model}, we can identify  the intermediate cover of $\BDiff^+(M)$ corresponding to the surjection $\pi_1\W$ with the homotopy quotient $(P_1\times P_2\times \SO(3))\hq \Diff_0(P_1\sqcup P_2)$. The projection onto the first two factors is equivariant with respect to the action of $\Diff_0(P_1\sqcup P_2)\cong \Diff_0(P_1)\times\Diff_0(P_2)$, which yields a fibre sequence: 
\[\SO(3)\hookrightarrow \frac{\underline{P_1\times P_2\times \SO(3)}}{ \Diff_0(P_1\sqcup P_2)}\rightarrow \frac{P_1}{\overline{ \Diff_0(P_1)}}\times \frac{P_2}{\overline{ \Diff_0(P_2)}}. \]
When the $P_i$ are aspherical, and neither is diffeomorphic to $S^1\times D^2$, \cref{lem:aspherical-quotient} implies the base of this fibration is aspherical. Passing to universal covers, we see that $\BDiff_0(P_1\sharp P_2)$ is homotopy equivalent to the universal cover of $\SO(3) \cong \mathbb{RP}^3$, namely $S^3$. 
In particular, $\Diff_0(P_1\sharp P_2)\simeq \Omega S^3$.
\end{ex}

\begin{figure}[hbt!]

    \begin{subfigure}{0.4\textwidth}
    \centering
        \begin{tikzpicture}
            \draw[red,thick] (0,0)--(1,0);
            \draw[red,thick] (1.5,0) circle (.5);
            \filldraw[red] (0,0) circle (.1);
            \filldraw[red] (1,0) circle (.1);
            \node[anchor=east] at (-.05,0) {$1$};
            \node at (1,-1){$(\Lambda,\rho) $};
        \end{tikzpicture}
    \label{fig:subim1}
    \end{subfigure}
   \begin{subfigure}{0.4\textwidth}
   \centering
        \begin{tikzpicture}
            \draw[red,thick] (1.5,0) circle (.5);
            \filldraw[red] (1,0) circle (.1);
            \node[anchor=east] at (.95,0) {1};
            \node at (1.5,-1){$(\Delta,\tau) $};
        \end{tikzpicture}
    \label{fig:subim2}
    \end{subfigure}
    \caption{The two objects of $\Gr_{1,1}$ are pictured. The graph $(\Lambda,\rho)$ at the left has a redundant edge and admits an edge collapse onto the graph $(\Delta,\tau)$ at the right, which has no redundant edges.  }
    \label{fig:Gr11-graphs}
\end{figure}
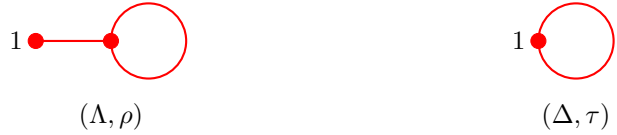

\begin{ex}[$n=1$, $g=1$]\label{ex:n=g=1}
There are only two labelled graphs $(\Lambda,\rho), (\Delta,\tau)$ in $\Gr_{1,1}$ (as pictured in Figure \ref{fig:Gr11-graphs}), but only $(\Delta,\tau)$ is in $\Gr_{1,1}^{\rm nr}$. In this case $\Aut(\Delta,\tau)\cong \Z/2$, which acts by inverting the loop. If $M=P\sharp (S^1\times S^2)$, then: \begin{align*}\Conf^{\fr+}_{(\Delta,\tau)}(P)=\Conf_2^{\fr+}(P)&\simeq \Conf_2(P)\times \SO(3)\times\SO(3)\\&=\{((p_1,p_2),\theta_1,\theta_2)\mid p_1\neq p_2\in P,~\theta_i\in \SO(3)\},\end{align*} where $\SO(\Delta,\tau)\cong \SO(3)$ again acts diagonally on the framings on the right. After quotienting by the action of $\SO(\Delta,\tau)$, we can write:
\[\Conf^{\fr+}_2(P)/\SO(\Delta,\tau)\simeq \Conf_2(P)\times \SO(3)=\{((p_1,p_2,\eta)\in \Conf_2(P)\times\SO(3)\}, \] where $\eta= \theta_1\circ \theta_2^{-1}$ measures the difference between the two framings, and is invariant under the action of $\SO(\Delta,\tau)$ by construction. The automorphism group $\Z/2$ acts freely on $\Conf_2^{\fr+}(P)$ by  exchanging points $p_1\leftrightarrow p_2$ and frames $\theta_1\leftrightarrow\theta_2$, hence it acts in the same way (freely) on $\Conf_2(P)$ but by $\eta\mapsto \eta^{-1}$ on the $\SO(3)$-factor. Thus the prime decomposition fibre sequence is: \[\frac{\Conf_2(P)\times\SO(3)}{\Z/2}\hookrightarrow \BDiff^+(P\sharp (S^1\times S^2))\rightarrow \BDiff^+(P).\]
This fibre is noncompact, so we now describe the construction of the compact manifold arising in \cref{prop:Conf-strict-quotient} that is equivalent to the fibre. The configuration space $\Conf_2(P)$ is the interior of a smooth 6-manifold with boundary, denoted by $\Conf_2[P]$, which can be obtained blowing up the diagonal of $P\times P$ to a copy of the unit tangent bundle of $P$ (identified with the unit normal bundle to the diagonal). The action of $\Z/2$ extends to the blown up manifold by taking its derivative along the diagonal, which is the antipodal map on each sphere of  the unit normal bundle. This action is clearly free on $\Conf_2(P)$ as well as on the boundary of $\Conf_2[P]$, since the antipodal map acts freely on $S^2$. Thus the action of $\Z/2$ is free on $\Conf_2[P]\times \SO(3)$ and the quotient is a compact 9-manifold with boundary, homotopy equivalent to the fibre $\rmC_1(P)$.

For the computation of $\BDiff_0(M)$, we follow the same argument as in \cref{ex:n=2-g=0}. Now the cover of $\BDiff^+(M)$ corresponding to $\pi_1\W$ is the homotopy quotient of $\Conf_2(P)\times\SO(3)$ by $\Z/2\times\Diff_0(P).$ 
We can pass to the further 2-fold cover which just the quotient $(\Conf_2(P)\times \SO(3))\hq\Diff_0(P).$ 
Projecting onto the first of the two configuration points in $\Conf_2(P)$ is $\Diff_0(P)$-equivariant, hence yields a fibre sequence: \begin{equation}\label{eqn:punctured-P-sequence}(P\setminus *)\times \SO(3)\hookrightarrow \frac{\underline{\Conf_2(P)\times \SO(3)}}{\Diff_0(P)}\rightarrow \frac{P}{\overline{\Diff_0(P)}}.\end{equation}
Setting $\pi=\pi_1(P)$,  the universal cover of $(P\setminus *)$ is $\widetilde{P}\setminus \pi$, where $\pi$ is identified with the orbit of a base-point in the interior of $\smash{\widetilde{P}}$ under the deck group action. Assuming that $P$ is aspherical and not $S^1\times D^2$, the base of the above fibre sequence is aspherical by \cref{lem:aspherical-quotient}, hence passing to universal covers we obtain: 
\[\BDiff_0(P\sharp (S^1\times S^2))\simeq \left(\widetilde{P}\setminus \pi\right)\times S^3\simeq \left(~\bigvee_{i=1}^\infty S^2\right)\times S^3, \] 
where the second equivalence comes from the fact that $\pi$ is infinite and the interior of $\widetilde{P}$ is homeomorphic to $\bbR^3$. 
It follows that $\pi_2\BDiff^+(M)=\pi_2\BDiff_0(M)=H_2(\widetilde{P}\setminus \pi)\cong \Z\pi$, regarded as a module over the integral group ring of $\pi=\pi_1(P\setminus *)$.
This recovers the result of \cite{KalliongisMcCullough} that $\pi_2$ is infinitely generated in this case.
Note that, by comparison, we are able to calculate $\pi_2$ exactly, and since $\pi=\pi_1(P\setminus *)$ injects into  the mapping class group $\pi_1\BDiff^+(M)$ by \cref{eqn:punctured-P-sequence}, our computation additionally shows that $\pi_2\BDiff^+(M)$ is finitely generated as a $\Z\pi_1\BDiff^+(M)$-module.
In fact, it is a cyclic $\Z \pi_1\BDiff^+(M)$-module.
\end{ex}

\subsection{\texorpdfstring{$\rmC^\Gamma_\sigma$}{C-Gamma} in terms of configuration spaces}

Fix $(\Gamma, \sigma) \in \Gr_{g,n}$ and $P_1,\dots,P_n$ irreducible manifolds that are neither $D^3$ nor $S^3$. Via the labelling $\sigma$, we associate the oriented diffeomorphism class $[P_i]$ to the vertex $\sigma(i)$. For each half-edge $h\in H_\Gamma$ we set $P_{r(h)}=P_i$ if $r(h)=\sigma(i)$ and $P_{r(h)}=S^3$ otherwise. 

For an oriented 3-manifold $M$, let $\Fr^+(M)$ denote the space of positively oriented frames on $M$. Thus, $\Fr^+(M)$ is the total space of a principal $\GL_3^+(\bbR)$-bundle $\pi\colon \Fr^+(M)\rightarrow M$. For  a finite set $A$ we define the configuration space of $A$ in $M$ to be \[\Conf_A(M):=\left\{(p_a)_{a\in A}\in \prod_A M\;\bigg|\; a \neq b \Rightarrow p_a \neq p_b\right\}.\]
We define the \emph{space of positively framed configurations of the set $A$ in $M$} via the pullback square:\[\begin{tikzcd}
    \Conf_A^{\fr^+}(M)\ar[r,hook]\ar[d]&\prod_A\Fr^+(M)\ar[d]\\
\Conf_A(M)\ar[r,hook]&\prod_A M
\end{tikzcd}\]
where $\prod_A\Fr^+(M)\rightarrow \prod_A M$ is the product bundle. Setwise, $\Conf_A^{\fr^+}(M)$ may be described as follows:
\[
    \Conf_A^{\fr+}(M) = \left\{(p_a,\theta_a)_{a \in A} \in \prod_A \Fr(M) \;\bigg|\; a \neq b \Rightarrow p_a \neq p_b \text{ and } \theta_a \text{ is positive} \right\}.
\]
We also write $\Conf_n^{\fr+}(M)$ when $A = \{1,\dots,n\}$, but it will usually be helpful to keep track of the name of the set we are embedding.
The wreath product  $ \SO(3) \wr\Aut(A) = \Map(A,\SO(3))  \rtimes \Aut(A) $ acts on this space by relabelling points and rotating framings.

\begin{defn}\label{defn:action-on-Conf}
For $(\Gamma,\sigma)\in \Gr_{g,n}$ and irreducible manifolds $P_1,\ldots, P_n$ as above we define 
\[\Conf^{\fr+}_{(\Gamma,\sigma)}(P_1,\ldots,P_n):=\prod_{i=1}^n \Conf_{H_{\sigma(i)}}^{\fr+}(P_i) \times \prod_{v \in V_\Gamma^{\rm sph}} \Conf_{H_v}^{\fr+}(S^3).\]
A point in $\Conf^{\fr+}_{(\Gamma,\sigma)}(P_1,\ldots,P_n)$ is a tuple $\mathbf{x}=(p_h,\theta_h)_{h\in H_\Gamma}$ where $p_h\in P_{r(h)}$, $\theta_h$ is a positive frame at $p_h$, and $p_h\neq p_{h'}$ for $h\neq h'$. An element $\pmb{\gamma}=(\gamma_e)_{e\in E_\Gamma}\in \prod_{E_\Gamma} \SO(3) $ acts on $\mathbf{x}$ by pre-composition:
\[(\pmb{\gamma}.\mathbf{x})_h=(p_h,\theta_h\circ\gamma_{[h]}^{-1}),\]
where $[h] \in E_\Gamma$ denotes the edge containing the half-edge $h$.
An element $\pmb{\alpha}=(\alpha_v)_{v\in V_\Gamma^{\rm sph}}\in\prod_{V_\Gamma^{\rm sph}} \SO(4)$ acts on $\mathbf{x}$ via: 
\[(\pmb{\alpha}.\mathbf{x})_h=\left\{\begin{array}{ll}(\alpha_v(p_h),D_{p_h}(\alpha_v)\circ\theta_h)& r(h)=v\in V_\Gamma^{\rm sph}\\
(p_h,\theta_h), &r(h)\in V_\Gamma^{\rm prime}
\end{array}\right.\]
where $D_{p_h}(\alpha_v)\colon T_{p_h}S^3\rightarrow T_{\alpha_v(p_h)}S^3$ is the derivative. In particular, $\pmb{\alpha}$ acts trivially on framed points in the $P_i$.
Since $\prod_{V_\Gamma^{\rm sph}} \SO(4)$ acts by post-composition, it commutes with the action of $\prod_{E_\Gamma} \SO(3)$. 
We therefore obtain an action of their product which we denote by: 
\[
    \SO(\Gamma,\sigma):=\prod_{E_\Gamma} \SO(3) \times \prod_{V_\Gamma^{\rm sph}} \SO(4).
\]
Lastly, an element $\varphi\in\Aut(\Gamma,\sigma)$ permutes the half-edges $H_\Gamma$ and therefore acts on $\mathbf{x}$ by \[(\varphi.\mathbf{x})_h=(p_{\varphi(h)},\theta_{\varphi(h)}).\]
Since the action of $\Aut(\Gamma,\sigma)$  preserves both $E_{\Gamma}$ and $V_{\Gamma}^{\rm sph}$, it normalises the actions of $\prod_{E_\Gamma}\SO(3)$ and $\prod_{V_\Gamma^{\rm sph}} \SO(4)$. Explicitly, we compute that \[(\varphi\circ\pmb{\gamma}\circ\varphi^{-1})_e=\gamma_{\varphi(e)}~~\text{ and } ~~(\varphi\circ\pmb{\alpha}\circ\varphi^{-1})_v=\alpha_{\varphi(v)}.\]

In summary, we obtain a group action
\[
    \Aut(\Gamma,\sigma) \ltimes \SO(\Gamma,\sigma) 
    \curvearrowright \Conf^{\fr+}_{(\Gamma,\sigma)}(P_1,\ldots,P_n).
\]
Moreover, this action commutes with the $\prod_{i=1}^n \Diff^+(P_i)$-action on  $\prod_{i=1}^n\Conf_{H_{\sigma(i)}}^{\fr+}(P_i)$.
\end{defn}
\begin{rem}
Regarding $S^3$ as the unit quaternions, the left action of $S^3$ on itself induces a faithful representation $\ell\colon S^3\hookrightarrow \SO(4)$, $p\mapsto \ell_p$. 
Equip $S^3$ with the Lie group framing.
Identifying the $\SO(4)$-stabiliser of $1\in S^3$ with $\SO(3)$, an explicit formula for the derivative of $\alpha\in \SO(4)$ at a point $p\in S^3$ is given by:
\[D_p(\alpha)=\ell_{\alpha(p)}^{-1}\circ \alpha\circ\ell_p.\] 
Note that for example that $D_p(\ell_q) = \id_{\bbR^3}$ as we used the left-multiplication to identify the tangent spaces with $\bbR^3$.
\end{rem}

We will break up the proof of \cref{prop:Conf-strict-quotient} into several steps.

\begin{lem}\label{lem:free-action-on-Conf-interior}
    For $n>0$, the action of $\Aut(\Gamma,\sigma)\ltimes \SO(\Gamma,\sigma)$ on $\Conf_{(\Gamma,\sigma)}(P_1,\ldots,P_n)$ is free.
\end{lem}
\begin{proof}
    Using the semi-direct product decomposition, any element  $\pmb{\mu} \in \Aut(\Gamma,\sigma)\ltimes \SO(\Gamma,\sigma)$ can be written uniquely as $\pmb{\mu}=(\varphi; \pmb{\gamma},\pmb{\alpha})$, where $\varphi\in \Aut(\Gamma,\sigma)$, $\pmb{\gamma}=(\gamma_e)_{e\in E_\Gamma}$,
    and $\smash{\pmb{\alpha}=(\alpha_v)_{v\in V_\Gamma^{\rm sph}}}$.  Suppose $\mathbf{x}=(p_h,\theta_h)_{h\in H_\Gamma}$ is a point in $\smash{\Conf_{(\Gamma,\sigma)}^{\fr+}(P_1,\ldots,P_n)}$ such that $\pmb{\mu}.\mathbf{x}=\mathbf{x}$. We will show that $\pmb{\mu}$ is trivial.

    First, let $v\in V_{\Gamma}^{\rm prime}$, and let $h_1,\ldots, h_d$ be the half-edges based at $v$.  By definition, $\varphi(v)=v$ and $\pmb{\alpha}$ acts trivially on $p_{h_j}$ for $1\leq j\leq d$. Since $p_{h_1},\ldots,p_{h_d}$ are pairwise distinct, we must have $\varphi(h_j)=h_j$ for $1\leq j\leq d$. Lastly, since $\gamma_{[h]}$ acts freely on $\theta_{h}$ for each $h\in H_\Gamma$, we see that $\gamma_{e_j}=1$ for the edges $e_j = [h_j]$ incident at $v$. 

    Set $V_0=V_{\Gamma}^{\rm prime}$ and for $k\geq 1$, inductively define $V_k=\{v\in V_\Gamma\mid \textrm{dist}_{\Gamma}(v,V_{k-1})\leq 1\}$. We will show, by induction on $k$, that for each $v\in V_k$ and each half-edge $h$ based at $v$:
    \begin{enumerate}
        \item $\varphi(v)=v$ and $\varphi(h)=h$,
        \item $\alpha_v=1$ and $\gamma_{[h]}=1.$
    \end{enumerate}
    The base case $k=0$ of this induction was proved in the previous paragraph. For the inductive step, let $v\in V_k$ for some $k\geq 1$. In particular, $v\in V_\Gamma^{\rm sph}$. Let $h_0$ be a half-edge at $v$ such that $w=r(h_0^\dagger)\in V_{k-1}$. 
    Since $\varphi(w)=w$ and $\varphi(h_0^\dagger)=h_0^\dagger$, we see that $\varphi(v)=v$ and $\varphi(h_0)=h_0$. 
    Because $\varphi$ fixes $h_0$ we must have $\alpha_v(p_{h_0}) = p_{h_0}$.
    But the $\SO(4)$-stabiliser of $p_{h_0}$ acts freely on frames, so we get that $D_{p_{h_0}}(\alpha_v)=1$ as well.  Since $\alpha_v$ fixes the point $p_{h_0}\in S^3$ and the frame $\theta_{h_0}$ at $p_{h_0}$, it must be the identity. It follows that for each half-edge $h$ based at $v$ we have $(\pmb{\mu}.p)_h=p_{\varphi(h)}$, and since the points $p_h$ are pairwise distinct, we conclude that $\varphi(h)=h$ for every $h$ based at $v$. 
    
    We have thus proved (1) and the first half of (2) holds at $v$. Then for each half-edge $h$ based at $v$, we have: \[(\pmb{\mu}.\mathbf{x})_h=(p_h,\theta_h\circ \gamma_{[h]}^{-1}).\] 
    Again using that $\gamma_{[h]}$ acts freely on $\theta_h$, we obtain that $\gamma_{[h]}=1$, which proves the second half of (2) and completes the inductive step. 
    Since $\Gamma$ is finite and connected, $V_\Gamma=V_R$ for some $R\leq \textrm{diam}(\Gamma)$, and thus we have proved that $\pmb{\mu}=1$ when $k=R$.
\end{proof}

\begin{defn}\label{defn:orthogonal-framed-configuration-space}
    For a finite set $A$ and an oriented 3-manifold $M$ equipped with a Riemannian metric define the subspace $\Conf_A^{\SO}(M)\subset \Conf_A^{\fr+}(M)$ to be 
    \[\Conf_A^{\SO}(M)=\{(p_a,\theta_a)\in \Conf^{\fr+}_A(M)\mid \theta_a \text{ is an isometry}\},\]
    meaning that each $\theta_a$ defines an orthonormal frame with respect to the metric.
    We then obtain a corresponding subspace of $\Conf^{\fr+}_{(\Gamma,\sigma)}(P_1,\ldots,P_n)$: \[\Conf^{\SO}_{(\Gamma,\sigma)}(P_1,\ldots P_n):=\prod_{i=1}^n \Conf_{H_{\sigma(i)}}^{\SO}(P_i) \times \prod_{v \in V_\Gamma^{\rm sph}} \Conf_{H_v}^{\SO}(S^3).\]
\end{defn}
    The action of $\Aut(\Gamma,\sigma)\ltimes \SO(\Gamma,\sigma)$ preserves $\Conf^{\SO}_{(\Gamma,\sigma)}(P_1,\ldots,P_n)$ and there is a strong $\Aut(\Gamma,\sigma)\ltimes \SO(\Gamma,\sigma)$-equivariant deformation retraction of $\Conf_{(\Gamma,\sigma)}^{\fr+}(P_1,\ldots, P_n)$ onto $\Conf^{\SO}_{(\Gamma,\sigma)}(P_1,\ldots,P_n)$, \emph{i.e.} $\Conf_{(\Gamma,\sigma)}^{\fr+}(P_1,\ldots, P_n)\simeq \Conf^{\SO}_{(\Gamma,\sigma)}(P_1,\ldots,P_n)$.

The framed configuration space $\Conf_A^{\fr+}(M)$ is the total space of a principal $\prod_A\GL_3^+(\bbR)$-bundle over $\Conf_A(M)$. By replacing each $\GL_3^+(\bbR)$ with $\SO(3)$, we have at least made the fibres of this bundle compact. However, the base, being a configuration space, is still fundamentally noncompact.

Recall that $\Conf_A(M)$ embeds as the interior of a compact manifold with corners $\Conf_A[M]$ known as the \emph{Fulton--MacPherson compactification}. Roughly speaking, a point in the boundary of $\Conf_A[M]$ corresponds to a sequence of points in $\Conf_A(M)$ which limit to the big diagonal in $\prod_AM$.  Each boundary stratum keeps track of the relative rate and tangent direction at which these points collide. For more details and properties of this construction we refer the reader to Sinha's article \cite{Sinha04}. For us, the most relevant properties will be:
\begin{enumerate}[(i)]
\item $\Conf_A[M]$ is a compact manifold with corners and $\Conf_A(M)\hookrightarrow \Conf_A[M]$ embeds as its interior. In particular, this inclusion is a homotopy equivalence  \cite[Theorem 4.4 and Corollary 4.5]{Sinha04}. 
\item There is a quotient map $\rho\colon \Conf_A[M]\rightarrow \prod_AM$ which restricts to a homeomorphism on $\Conf_A(M)$ (this follows directly from \cite[Definition 1.3]{Sinha04}).
\item $\Conf_A[M]$ is functorial in the sense that if $f\colon M\rightarrow N$ is an embedding, then there is an embedding $\Conf_A[M]\rightarrow \Conf_A[N]$. In particular, $\Diff(M)$ acts on $\Conf_A[M]$ by diffeomorphisms \cite[Corollary 4.9]{Sinha04}.
\end{enumerate}

\begin{lem}\label{lem:compactified-action}
    Let $(M,g)$ be an oriented Riemannian 3-manifold. For any finite set $A$, there is a compactification of $\Conf_A^{\SO}(M)$ as the interior of a manifold with corners $\Conf^{\SO}_A[M]$ that fits into a composition of pullback squares
    \[\begin{tikzcd}
        \Conf_A^{\SO}(M)\arrow[r,hook]\arrow[d,"\pi"']&
        \Conf^{\SO}_A[M]\arrow[d,"\pi"'] \rar
        & \prod_A \Fr^{\SO}(M) \arrow[d,"\pi"']
         \\
        \Conf_A(M)\arrow[r,hook,"j"]&
        \Conf_A[M] \rar["\rho"] & 
        \prod_A M
    \end{tikzcd}\]
    In particular, $\pi\colon \Conf_A^{\SO}[M]\rightarrow\Conf_A[M]$ is a principal $\prod_A\SO(3)$-bundle over $\Conf_A[M]$. Moreover, the action of $\Isom^+(M,g)$ on $\Conf^{\SO}_A(M)$ extends smoothly to $\Conf^{\SO}_A[M]$ and acts equivariantly on the diagram. 
\end{lem}
\begin{proof}
Let $j\colon \Conf_A(M)\hookrightarrow \Conf_A[M]$ be one inclusion, and let $i\colon \Conf_A(M)\hookrightarrow \prod_A M$ be the other. By property (ii) above, there is a natural projection $\rho\colon \Conf_A[M]\rightarrow \prod_A M$ which restricts to a homeomorphism on the interior, hence yields a factorisation $i=\rho\circ j$.   We can regard  $\Conf_A^{\SO}(M)$ as the pullback of the principal $\prod_A\SO(3)$-bundle $\pi\colon \prod_A\Fr^{\SO}(M)\rightarrow \prod_A M$ along $i$, where $\Fr^{\SO}(M)$ is the space of positive orthonormal frames on $M$. Thus if we define $\Conf_A^{\SO}[M]$ as the pullback of $\pi$ along $\rho$, then we obtain a composition of pullback squares as in the statement of the proposition.
In particular,  $\Conf_A^{\SO}[M]$ is a compact manifold with corners whose interior is $\Conf_A^{\SO}(M)$, and $\pi\colon \Conf_A^{\SO}[M]\rightarrow \Conf_A[M]$ is a principal $\prod_A\SO(3)$-bundle. By property (iii), the action of $\Diff^+(M)$, and hence $\Isom^+(M,g)$, extends to $\Conf_A[M]$.  On the other hand, the $\Isom^+(M,g)$ action on $\Conf_A(M)$ also lifts to $\Conf_A^{\SO}(M)$, acting on fibres via the derivative map.  By continuity we obtain a smooth extension to $\Conf_A^{\SO}[M]$  making the left hand square $\Isom^+(M,g)$-equivariant. Equivariance of the outer square is clear, so the proposition follows.
\end{proof}
\begin{rem}\label{rem:free-SO(3)-action}
    Since $\pi\colon \Conf_A^{\SO}[M]\rightarrow \Conf_A[M]$ is a principal $\prod_A\SO(3)$-bundle, it comes with a free right action of $\prod_A\SO(3)$.
\end{rem}

The last ingredient we will need to prove \cref{prop:Conf-strict-quotient} is the following general fact about actions of compact Lie groups on manifolds with corners.
\begin{thm}\label{thm:compact-Lie-action}
    Let $G$ be a compact Lie group acting smoothly on a compact manifold with corners $M$.
    If $G$ acts freely on the interior of $M$, then $G$ acts freely on $M$ and the quotient map $M \to M/G$ is a principal $G$-bundle.
    Moreover, $M/G$ is a compact manifold with corners and hence both it and $M\sslash G$ are homotopy finite.
\end{thm}
\begin{proof}
    Since $G$ is compact and $M$ is Hausdorff, the action is automatically proper (see \emph{e.g.} \cite[Ch.III.4, Proposition 2]{Bourbaki-Topologie-I-IV}. We first show that the action of $G$ is free on all of $M$. Suppose $\xi\in \partial M$ is a corner of dimension $k$, and let $G_\xi$ be the stabiliser of $\xi$. By \cite[Proposition 5.1]{Albin-Melrose-2011}, there exists a chart $U\ni \xi$ diffeomorphic to $\bbR^{n,k}$ on which the action of $G_\xi$ is linear, and in fact orthogonal. Thus, on $U$, the action of $G_\xi$ preserves the $k$-simplex of inward pointing unit vectors. In particular, $G_\xi$ fixes the ray corresponding to the barycenter of this simplex pointwise. Since these points are interior, we conclude that $G_\xi=\{1\}$. Thus the action of $G$ on $M$ is free and proper. 
    
    Since $G$ is compact, it has only finitely many connected components. Let $G_0$ be the connected component of the identity. We claim that $G_0$ satisfies the conclusion of the theorem. Indeed, as $G_0$ is connected, the action of $G_0$ on $M$ is boundary-intersection free in the sense of \cite[Definition 1.4]{Albin-Melrose-2011}. Let $\xi\in M$ be a corner of dimension $k$. We apply the tube theorem \cite[Proposition 5.3]{Albin-Melrose-2011} with $G_\xi=\{1\}$ to find a $G_0$-invariant neighbourhood $U$ of $\xi$, an $m$-submanifold with corners $V^+\cong\bbR^{m,k}$ (for $m=n-\dim(G_0)$) containing $\xi$ with trivial $G_0$-stabiliser and a $G_0$-equivariant diffeomorphism $\psi\colon G_0\times V^+\rightarrow U$. This provides a local trivialisation for the quotient map $q_0\colon M\rightarrow M_0 := M/G_0$ in a neighbourhood of $q_0(\xi)$, which maps $V^+$ diffeomorphically to $M_0$.  Since $\xi$ was arbitrary, $M_0$ has the structure of a smooth manifold with corners and $q_0$ is a principal $G_0$-bundle. 

    To finish the proof, observe that $M_0$ inherits a free action of the finite group $F:=G/G_0$. Since $M_0$ is Hausdorff, the quotient $q_1\colon M_0\rightarrow M_1=M_0/F$ is automatically a covering action.  Therefore $M_1$ is also a manifold with corners and composition $q=q_1\circ q_0\colon M\rightarrow M_1$ is a principal $G$-bundle, as required.
\end{proof}

\begin{proof}[Proof of \cref{prop:Conf-strict-quotient}]

By compactifying each factor in the product, we obtain a homotopy equivalent compactification of $\Conf^{\SO}_{(\Gamma,\sigma)}(P_1,\ldots,P_n)$ as: \[\Conf^{\SO}_{(\Gamma,\sigma)}[P_1,\ldots,P_n]:=\prod_{i=1}^n\Conf^{\SO}_{H_{\sigma(i)}}[P_i]\times\prod_{v\in V^{\rm sph}_\Gamma}\Conf^{\SO}_{H_v}[S^3],\]
which is a smooth manifold with corners of dimension \[6\left(\sum_{i=1}^n|H_{\sigma(i)}|+\sum_{v\in V^{\rm sph}_\Gamma}|H_v|\right)=12|E_\Gamma|.\]
By \cref{lem:compactified-action} and \cref{rem:free-SO(3)-action} for $\SO(\Gamma,\sigma)$ and by \cite[Theorem 4.10]{Sinha04} for $\Aut(\Gamma,\sigma)$, the action of $\Aut(\Gamma,\Sigma)\ltimes\SO(\Gamma,\sigma)$ extends smoothly to this compactification.  By \cref{lem:free-action-on-Conf-interior} and \cref{thm:compact-Lie-action}, the extension is free and principal, hence the quotient is a compact manifold with corners. In particular, the homotopy quotient is equivalent to the strict quotient, which, being a homeomorphic to a compact manifold with boundary, has the homotopy type of a finite CW complex. 

Recalling that $\Aut(\Gamma,\sigma)$ is discrete and that as a Lie group there is an isomorphism $\SO(\Gamma,\sigma)\cong\prod_{e\in E_\Gamma}\SO(3)\times\prod_{v\in V^{\rm sph}_\Gamma}\SO(4)$, the quotient has dimension:
\begin{align*}12|E_\Gamma|-(3|E_\Gamma|+6|V_\Gamma^{\rm sph}|)&= 9|E_\Gamma|-6|V_\Gamma^{\rm sph}|\\
&=9|E_\Gamma|-6|V_\Gamma|+6n\\
&=9|E_\Gamma|-6(|E_\Gamma|-(g-1))+6n\\
&=3|E_\Gamma|+6(g-1)+6n,
\end{align*}
where we have used that $1-g=\chi(\Gamma)=|V_\Gamma|-|E_\Gamma|$ on the third line. The maximal number of edges $\Gamma$ can have is $3(g-1)+2n$ by \cref{lem:depth-of-Gr_gn}(i), hence the dimension of the quotient is at most $15(g-1)+12n$, as claimed.
\end{proof}

\subsection{Thickened separating systems}
To prove \cref{prop:C-Gamma configuration model} and certain other descriptions of $\rmC^\Gamma_\sigma$, we will construct a map that roughly takes a separating system $\Sigma \subset M$ and sends it to the various path components of $M\ca \Sigma$.
In defining this on the level of moduli spaces we run into a technical issue: 
points in $\BDiff$ can be described as unparametrised submanifolds of $\bbR^\infty$, but if $M \subset \bbR^\infty$, then there is no canonical way for us to think of $M \ca \Sigma$ as a submanifold of $\bbR^\infty$.
We will resolve this by remembering for each separating system $\Sigma \subset M$ a parametrised tubular neighbourhood $\Sigma^\varepsilon$ of $\Sigma$ in $M$, so that instead of $M \ca \Sigma$ we can simply take the diffeomorphic manifold $M \setminus \Sigma^\varepsilon$ that is literally a submanifold of $\bbR^\infty$.
This is a purely technical point and will not really affect any of our arguments.

\begin{defn}\label{defn:thickened-spheres}
    A \emph{thickened separating system} is a separating system $\Sigma$ together with an equivalence class of embeddings $i\colon \Sigma \times [-1,1] \hookrightarrow M$ satisfying $i(x, 0) = x$ for all $x \in \Sigma$.
    Here we consider two embeddings equivalent if they differ by flipping the $[-1,1]$ factor for some of the components of $\Sigma$.
    We let $\Sigma^\varepsilon_\circ := i(\Sigma\times (-1,1)) \subset M$ and $\Sigma^\varepsilon := i(\Sigma \times [-1,1])$ denote the open and closed tubular neighbourhoods obtained from the thickening.
    By abuse of notation, we often write $\Sigma^\varepsilon$ to denote the entire datum $[\Sigma, i]$.
\end{defn}

\begin{defn} 
    A diffeomorphism $\psi\in \Diff^+(\Sigma\times[-1,1])$ is called \emph{cylindrical} if $\psi(x,t)=(\psi_0(x),t)$ for some $\psi_0\in \Diff^+(\Sigma)$.
    We denote the subgroup of cylindrical diffeomorphisms of $\Sigma^\varepsilon\cong \Sigma\times[-1,1]$ by $\Diff^{\cyl}(\Sigma^{\varepsilon})$.
    Given a thickened separating system $\Sigma^\varepsilon \subset M$,  we let 
    $\Diff^\cyl(M, \Sigma^\varepsilon) \leq \Diff(M, \Sigma^\varepsilon)$
    denote the subgroup of those diffeomorphisms that are cylindrical on $\Sigma^\varepsilon$, \emph{i.e.}~those diffeomorphisms such that $i^{-1}\circ\varphi_{|\Sigma^{\varepsilon}} \circ i$ is cylindrical. In particular, $\varphi(\Sigma)=\Sigma.$
    That is, we take those $\varphi\colon M \cong M$ such that $\varphi(\Sigma) = \Sigma$ and $\varphi(i(x,t)) = i(\varphi(x),t)$ holds for $i$ part of the thickening.
    (Note that it does not matter which $i$ we pick from the equivalence class of embeddings $\Sigma \times [-1,1] \hookrightarrow M$.)
\end{defn}

\begin{lem}\label{lem:cyl==2Sigma}
    The subgroup inclusion $\Diff^\cyl(M, \Sigma^\varepsilon) \subset \Diff^+_{\pi_0(2\Sigma)}(M, \Sigma)$ is admissible and a homotopy equivalence.
\end{lem}
\begin{proof}
    As in the proof of \cref{lem:O(3)-subgroup-equivalence} we let $\Coll_\Sigma := \emb^+_{\Sigma \times \{0\}}(\Sigma \times [-1,1], M)$ be the (contractible) space of orientation-preserving collars of $\Sigma$.
    The group $\Diff^+_{\pi_0(2\Sigma)}(M, \Sigma)$ acts on $\Coll(\Sigma)$ by
    \[
        \varphi.\iota := \varphi \circ \iota \circ (\varphi_{|\Sigma}^{-1} \times \id_{[-1,1]}).
    \]
    Note this indeed preserves the condition that $\iota_{|\Sigma} = \id_\Sigma$.
    Restricting to the subgroup $\Diff^+_\Sigma(M)$ this action becomes the usual action by post-composition, which we know to be locally retractile.
    Using the same local retraction it follows that the action of the larger group $\Diff^+_{\pi_0(2\Sigma)}(M, \Sigma)$ is also locally retractile.
    The stabiliser of this action at some point $\Sigma^\varepsilon \in \Coll(\Sigma)$ is exactly the group of cylindrical diffeomorphisms, so this subgroup is admissible by \cref{obs:admissible}.
    Moreover, we have a fibre sequence
    \[
        \Diff^\cyl(M, \Sigma^\varepsilon) \longrightarrow \Diff^+_{\pi_0(2\Sigma)}(M, \Sigma) \longrightarrow \Coll(\Sigma)
    \]
    and as the base is contractible it follows that the subgroup inclusion is an equivalence.
\end{proof}

From \cref{lem:cyl==2Sigma} and \cref{obs:admissible} we can deduce that the restriction of a principal $\Diff^+(M)$-action to $\Diff^\cyl(M,\Sigma^\varepsilon)$ is always principal.
\begin{cor}\label{lem:Diff-cyl-principal}
    The action of $\Diff^\cyl(M,\Sigma^\varepsilon)$ on $\emb(M, \bbR^\infty)$ is principal.
\end{cor}

We use the thickened separating systems to define a variant $\Cthick^\Gamma_\sigma$ of $\rmC^\Gamma_\sigma$.
It will be convenient to define this variant directly rather than as a homotopy orbit construction of a variant of $\sep^\Gamma_\sigma(M)$, as this removes the reference to $M$ and also makes it easier to describe the $\Aut(\Gamma,\sigma)$-action.

\begin{notation}\label{notation:dual-graph-indices}
    Consider a separating system $\Sigma \subset M$ with dual graph $G_{\Sigma \subset M}$.
    For every vertex $v$ of $G_{\Sigma \subset M}$ we let $(M \ca \Sigma)_v \subset M \ca \Sigma$ denote the component corresponding to $v$.
    If the vertex $v$ is such that $(M\ca \Sigma)_v$ is not a punctured sphere, we also write $P(\Sigma \subset M)_v \subset P(\Sigma \subset M)$ for the component of the maximal prime piece that contains $(M\ca \Sigma)_v$.
\end{notation}

\begin{defn}\label{defn:thick}
    A point in $\Cthick^\Gamma_\sigma(P_1,\dots, P_n)$ consists of the data 
    \begin{itemize}
        \item $N \subset \bbR^\infty$ a compact, oriented, $3$-dimensional submanifold,
        \item $\Sigma^\varepsilon\subset N$ a thickened separating system,
        \item $\alpha \colon \Gamma \cong G_{\Sigma \subset N}$ an isomorphism of graphs such that for every vertex $v \in V_\Gamma$ the component
            $(N \ca \Sigma)_{\alpha(v)}$ is abstractly diffeomorphic to $S^3 \setminus \amalg_k \interior{D^3}$ if $v$ is unmarked and to $P_i \setminus \amalg_k \interior{D^3}$ if $v=\sigma(i)$.
        \item an embedding $\iota \colon P(\Sigma \subset N) \hookrightarrow \coprod_{i=1}^n P_i$ sending the component $P(\Sigma \subset N)_{\alpha(\sigma(i))}$ to $P_i$.
    \end{itemize}
\end{defn}

If we pick a thickened separating system $\Sigma^\varepsilon \subset M$ and a graph isomorphism $\alpha\colon G_{\Sigma \subset M} \cong \Gamma$, then we can parametrise this space via the map
\begin{align*}
    q\colon \frac{\emb(M, \bbR^\infty) \times \prod_{i=1}^n \emb(P(\Sigma\subset M)_i, P_i)}{\Diff^\cyl(M,\Sigma^\varepsilon)}
    &\xrightarrow[\qquad]{\cong} \Cthick^\Gamma_\sigma(P_1,\dots,P_n)\\
    (j\colon M \hookrightarrow \bbR^\infty, (\iota_i)_{i=1}^n)
    &\longmapsto 
    (j(M), j(\Sigma^\varepsilon), \alpha \circ j_G^{-1}, (\sqcup_{i=1}^n \iota_i) \circ j_P^{-1})
\end{align*}
    where $j_G\colon G_{\Sigma \subset M} \cong G_{j(\Sigma) \subset j(M)}$ is the graph isomorphism induced by $j$ and $j_P\colon P(\Sigma \subset M) \cong P(j(\Sigma) \subset j(M))$ is the diffeomorphism induced on maximal prime pieces.

\begin{lem}
    The map $q$ is a bijection and the topology it induces on $\Cthick^\Gamma_\sigma(P_1,\dots,P_n)$ is independent of the choice of $\Sigma^\varepsilon \subset M$ and $\alpha$.
\end{lem}
\begin{proof} 
    For surjectivity, let $(N, \Upsilon^\varepsilon, \beta, \nu) \in \Cthick^\Gamma_\sigma(P_1,\dots,P_n)$.
    The composite $\beta\circ \alpha^{-1}$ defines an isomorphism of dual graphs $G_{\Sigma \subset M} \cong G_{\Upsilon \subset N}$ and hence \cref{lem:graph-functor-is-full} constructs a diffeomorphism $\psi\colon M \cong N$ with $\psi(\Sigma) = \Upsilon$ such that the induced map on dual graphs is exactly $\beta \circ \alpha^{-1}$.
    In fact, the proof of \cref{lem:graph-functor-is-full} constructs $\psi$ such that it preserves the thickenings.
    Then we have 
    \[
        q(\psi\colon M \cong N \subset \bbR^\infty, \nu \circ \psi_P) 
        = (\psi(M), \psi(\Sigma^\varepsilon), \alpha \circ \psi_G^{-1}, (\nu \circ \psi_P) \circ \psi_P^{-1})
        = (N, \Upsilon^\varepsilon, \beta, \nu),
    \]
    which shows that $q$ is surjective.
    Similarly, we see that the quotient topology does not depend on the choice of $\Sigma^\varepsilon \subset M$, as any other choice differs by a diffeomorphism that preserves the thickened separating system.
    
    Two pairs $(j, \iota)$ and $(j', \iota')$ map to the same point in $\Cthick^\Gamma_\sigma(P_1,\dots,P_n)$ if and only if $j(M) = j'(M)$, and the thickened sphere systems, identification of dual graphs, and the maximal prime marking agree.
    Hence we can find a diffeomorphism $\varphi\in \Diff^+(M)$ with $j'= j \circ \varphi$.
    This diffeomorphism automatically lies in $\Diff^\cyl(M,\Sigma^\varepsilon)$ because it must preserve the thickened sphere systems and identification of dual graphs.
    Moreover, $\iota \circ j_P^{-1} = \iota' \circ (j')_P^{-1}$ implies that $\iota \circ \varphi_P = \iota'$, so we have $(j',\iota') = (j \circ \varphi, \iota \circ \varphi_P)$.
    This shows that $q$ is a bijection. 
\end{proof}

The groups $\Aut(\Gamma,\sigma)$ and $\prod_i \Diff^+(P_i)$ act on the space $\Cthick^\Gamma_\sigma(P_1,\dots,P_n)$ by post-composition on the graph isomorphism and the maximal prime embedding, respectively.
The map
\begin{equation}\label{eqn:Cthick-to-Cregular}
    \Cthick^\Gamma_\sigma(P_1,\dots,P_n) \longrightarrow
    \rmC^\Gamma_\sigma(P_1,\dots,P_n)
\end{equation}
that forgets the thickening, sending $\Sigma^\varepsilon$ to $\Sigma$, is equivariant for the action of these groups.

\begin{cor}\label{cor:Cthick-to-C}
    The map in \cref{eqn:Cthick-to-Cregular} is a homotopy equivalence.
\end{cor}
\begin{proof}
    By choosing some $(M,\Sigma^\varepsilon,\alpha)$ we can write this map as
    \[
        \frac{\emb(M, \bbR^\infty) \times \prod_{i=1}^n \emb(P(\Sigma\subset M)_i, P_i)}{\Diff^\cyl(M, \Sigma^\varepsilon)}
        \longrightarrow
        \frac{\emb(M, \bbR^\infty) \times \prod_{i=1}^n \emb(P(\Sigma\subset M)_i, P_i)}{\Diff_{\pi_0(2\Sigma)}^+(M, \Sigma)}.
    \]
    On both sides the action of the group on the numerator is principal:
    on the left this follows from \cref{lem:Diff-cyl-principal} and on the right it follows from \cref{lem:Diff(M)->sep(M) principal}.
    The claim now follows from \cref{lucis lemma} and \cref{lem:cyl==2Sigma}.
\end{proof}

\subsection{Proving the configuration space description}\label{subsec:proof-of-config}
We will now construct the zig-zag of equivalences between $\rmC^\Gamma_\sigma$ and the configuration space model, as promised in \cref{prop:C-Gamma configuration model}.

We will not be able to directly map from $\Cthick^\Gamma_\sigma$ to the configuration space model, but if we construct more flexible models for the configuration spaces, we will be able to map into those.
Instead of considering configurations of points, we will consider configurations of $2$-spheres that can be filled, which will give an equivalent model via Hatcher's proof of the Smale conjecture \cite{Hatcher}.

\begin{defn}\label{defn:emb-D}
    For a finite set $A$ and an irreducible $3$-manifold $P$ we let 
    \[
        \emb^{\rm D}(A \times S^2, P) \subset \emb(A\times S^2, P)
    \]
    denote the subspace of those embeddings $\iota\colon A \times S^2 \hookrightarrow P$ that extend to an \emph{orientation preserving} embedding of $A \times D^3$.
    Since $P$ is irreducible we already know that each of the $\iota_a(S^2)$ bounds a ball, but the above condition requires that it bounds an oriented ball with respect to the preferred orientation on $S^2$.
\end{defn}

\begin{lem}\label{lem:embplus}
    For any $A$ and $P$ as above, there is an $\Aut(A) \wr \SO(3)$ and $\Diff^+(P)$ equivariant zig-zag of equivalences
    \[
        \emb^{\rm D}(A \times S^2, P) \xleftarrow{\simeq} \emb^+(A \times D^3, P) \xrightarrow{\simeq} \Conf_A^{\fr+}(P).
    \]
\end{lem}
\begin{proof}
    The left map is surjective by definition of $\emb^{\rm D}(\dots)$ in \cref{defn:emb-D}.
    It is a Serre fibration by \cite{Palais60,Cerf} and its fibres are $\Diff_\partial(A \times D^3)$, which is contractible by the Smale conjecture \cite{Hatcher}.
    The right map is obtained by taking the derivative at the centre of each disc and is an equivalence by a standard argument: we can construct a homotopy inverse by choosing a metric and using exponential maps. Equivariance is clear by inspection.
\end{proof}

We will also need a specific model for the moduli space of punctured $3$-spheres.
\begin{defn}\label{defn:BDiff_A^sph}
    For $A$ a finite set we define $\BDiff_A^{\rm sph}(S^3)$ to be the space of pairs $(W, \psi)$ of $W \subset \bbR^\infty$ a compact oriented manifold abstractly diffeomorphic a $3$-sphere with open balls removed and $\psi\colon \partial W \cong A \times S^2$ an orientation-preserving diffeomorphism.
    To topologise this, we pick an embedding $\iota\colon A \times D^3 \hookrightarrow S^3$, and require that the canonical bijection
    \[
        \BDiff_A^{\rm sph}(S^3) \cong \frac{\emb(S^3 \setminus \iota(A \times \interior{D^3}), \bbR^\infty)}{\Diff_\partial(S^3 \setminus \iota(A \times (D^3)^\circ))}
    \]
    be a homeomorphism.
\end{defn}

By construction, $\BDiff_A^{\rm sph}(S^3)$ is equivalent to $\BDiff_\partial(S^3 \setminus \iota(A \times \interior{D^3}))$, but we chose the above construction to make it independent of a choice of embedding $\iota$ and to make explicit the action
\[
    \Diff^+(A \times S^2) = \Aut(A) \ltimes \Map(A, \Diff^+(S^2)) \curvearrowright \BDiff_A^{\rm sph}(S^3)
\]
by changing the parametrisation.
Writing $V := S^3 \setminus \iota(A \times \interior{D^3})$, we can identify the above action with the quotient of the principal action $\Diff^+(V) \curvearrowright \emb(V, \bbR^\infty)$ by the admissible, normal subgroup $\Diff_\partial(V) \unlhd \Diff^+(V)$, 
and as a consequence the above action is principal by \cref{obs:admissible}.\ref{it:adm-quotient}.

\begin{lem}\label{lem:punctured-sphere-config}
    There is an $\Aut(A) \wr \SO(3)$ equivariant zig-zag of equivalences
    \[
        \BDiff_A^{\rm sph}(S^3) \xleftarrow{\simeq} 
        \frac{\emb^+(A \times D^3, S^3) \times \emb(S^3, \bbR^\infty)}{\SO(4)}
        \xrightarrow{\simeq} \Conf_A^{\fr+}(S^3)\sslash \SO(4)
    \]
\end{lem}
\begin{proof}
    The right arrow is induced by the homotopy equivalence $\emb^+(A \times D^3, S^3) \to \Conf_A^{\fr+}(S^3)$ that records the derivative at centre of each disc, as in \cref{lem:embplus}.
    We construct the left arrow as
    \[
        [\iota\colon A \times D^3 \hookrightarrow S^3, j\colon S^3 \hookrightarrow \bbR^\infty] \longmapsto
        (j(S^3 \setminus \iota(A \times \interior{D^3})), (j \circ \iota)_{|A \times S^2}^{-1}),
    \]
    which is $\Aut(A) \wr \SO(3)$-equivariant by inspection.
    To see that this is an equivalence, fix some $\iota\colon A \times D^3 \hookrightarrow S^3$ and factor the above map as
    \begin{align*}
        \frac{\emb^+(A \times D^3, S^3) \times \emb(S^3, \bbR^\infty)}{\SO(4)}
        & \longrightarrow
        \frac{\emb^+(A \times D^3, S^3) \times \emb(S^3, \bbR^\infty)}{\Diff^+(S^3)}
        \cong
        \frac{\emb(S^3, \bbR^\infty)}{\Diff_{\iota(A \times D^3)}(S^3)}\\
        & \longrightarrow
        \frac{\emb(S^3 \setminus \iota(A \times \interior{D^3}), \bbR^\infty)}{\Diff_\partial(S^3 \setminus \iota(A \times (D^3)^\circ))}
        \cong \BDiff_A^{\rm sph}(S^3)
    \end{align*}
    where the first map is an equivalence by the Smale conjecture,
    the second map is a homeomorphism, and the third map is an equivalence as both numerators are contractible and the map on denominators is an equivalence.
\end{proof}

Having established these alternative models for configuration spaces and punctured spheres, we can now compare $\Cthick^\Gamma_\sigma$ to a quotient defined in terms of these.
The final piece of notation needed is the group
\[
    \Map^{C_2}(H_\Gamma, \Diff^+(S^2)) \le 
    \Map(H_\Gamma, \Diff^+(S^2)) = \prod_{H_\Gamma} \Diff^+(S^2)
\]
whose elements are those families of diffeomorphisms $(\gamma_h)_{h \in H_\Gamma}$ such that $\gamma_{h^\dagger} = (-\id_{S^2}) \circ \gamma_h \circ (-\id_{S^2})$ holds for every half-edge $h$, where $(-\id_{S^2})$ denotes the antipodal map on $S^2$.
If we choose a preferred half-edge in each edge, \emph{i.e.}~an orientation of the graph, then we obtain an isomorphism between this group and $\prod_{E_\Gamma} \Diff^+(S^2)$, but this identification is not compatible with the $\Aut(\Gamma,\sigma)$-action. Note that, homotopically we can make this identification canonical: the group $\Map^{C_2}(H_\Gamma, \Diff^+(S^2))$ is equivalent to its subgroup $\Map^{C_2}(H_\Gamma, \SO(3))$, and this subgroup is just $\prod_{E_\Gamma} \SO(3)$ because the antipodal map is central in $\Or(3)$.%

\begin{lem}\label{lem:Cthick-to-modified-Conf}
    There is an $\Aut(\Gamma,\sigma) \times \prod_i \Diff^+(P_i)$-equivariant homotopy equivalence 
    \[
        \Theta \colon \Cthick^\Gamma_\sigma(P_1,\dots,P_n) \longrightarrow
        \frac{\prod_{v \in V_\Gamma^{\rm sph}} \BDiff_{H_v}^{\rm sph}(S^3) \times \prod_{i=1}^n \emb^{\rm D}(S^2 \times H_{\sigma(i)}, P_i)}{
            \Map^{C_2}(H_\Gamma, \Diff^+(S^2))
        }
    \]
    where the action by $\Map^{C_2}(H_\Gamma, \Diff^+(S^2))$ re-parametrises, for each $h \in H_\Gamma$, the $h$-labelled sphere in either $\BDiff_{H_v}^{\rm sph}(S^3)$ if $r(h)=v\in H_\Gamma^{\rm sph}$, or in $\emb(S^2 \times H_{\sigma(i)}, P_i)$ if $r(h) = \sigma(i)$.
\end{lem}
\begin{proof}
    We begin by constructing $\Theta$ as a map of sets (continuity is checked below)
    \[
        [\Sigma^\varepsilon \subset M \subset \bbR^\infty, \alpha, \iota\colon P(\Sigma \subset M) \hookrightarrow \amalg_i P_i]
        \longmapsto
        [
           ((M \setminus \Sigma^\varepsilon)_{\alpha(v)}, \amalg_{h \in H_v} \psi_h)_{v \in V_\Gamma^{\rm sph}}, 
           (\iota \circ (\amalg_{h \in H_{\sigma(i)}} \psi_h))_{i=1}^n
        ]
    \]
    where we chose orientation-preserving diffeomorphism $\psi_h \colon S^2 \cong \Sigma_{\alpha(h)}$ for all $h \in H_\Gamma$ such that $\psi_h \circ (-\id_{S^2}) = \psi_{h^\dagger}$.
    The map does not depend on the choice of the $\psi_h$ because we take the quotient by the group $\Map^{C_2}(H_\Gamma, \Diff^+(S^2))$. It follows from the above formula that the map is automatically equivariant for the action of $\Aut(\Gamma, \sigma)$ (which just reindexes) and $\prod_i \Diff^+(P_i)$ (which acts on the embeddings into the $P_i$).

    To check that the map is continuous and that it is a weak equivalence, we pick a preferred $\Sigma^\varepsilon \subset M$, which will allow us to write $\Cthick$ as a quotient of embedding spaces by $\Diff^\cyl(M,\Sigma^\varepsilon)$.
    Similarly, we can write $\Map^{C_2}(H_\Gamma, \Diff^+(S^2)) \cong \Diff^\cyl(\Sigma^\varepsilon)$.
    We further pick $\psi_h\colon S^2 \cong \Sigma_h$ as before.
    In this context the map $\Theta$ described above is 
    \[
        \frac{\emb(M, \bbR^\infty) \times \prod_{i=1}^n \emb(P(\Sigma \subset M)_i, P_i)}{
            \Diff^\cyl(M, \Sigma^\varepsilon)
        }
        \xrightarrow[\qquad]{\Theta}
        \frac{\prod_{v \in V_\Gamma^{\rm sph}} \BDiff_{H_v}^{\rm sph}(S^3) \times \prod_{i=1}^n \emb^{\rm D}(S^2 \times H_{\sigma(i)}, P_i)}{
            \Diff^\cyl(\Sigma^\varepsilon)
        },
    \]
    which comes from a map on the numerators that is equivariant for the denominators.
    Indeed, the map on numerators is
    \[
        [j\colon M \hookrightarrow \bbR^\infty, (\iota_i\colon P(\Sigma \subset M)_i \hookrightarrow P_i)_{i=1}^n]
        \longmapsto
        [
           (j((M \setminus \Sigma^\varepsilon)_{\alpha(v)}), \amalg_{h \in H_v} j \circ \psi_h)_{v \in V_\Gamma^{\rm sph}}, 
           (\iota_i \circ (\amalg_{h \in H_{\sigma(i)}} \psi_h))_{i=1}^n
        ].
    \]
    This is continuous, and hence so is the map on quotients.
    To prove that the map on quotients is an equivalence, we will compute its homotopy fibre.
    The space on the right is connected and 
    both group-actions are principal (see \cref{lem:cyl==2Sigma} and the text below \cref{defn:BDiff_A^sph}). Hence, by \cref{lucis lemma}, the homotopy fibre of our map is the quotient of the fibre of the numerators by the fibre of the denominators.

    On denominators we have a homotopy fibre sequence
    \[
        \prod_{v \in V_\Gamma} \Diff_\partial( (M \setminus \Sigma^\varepsilon_\circ)_v ) \simeq \Diff_{\Sigma^\varepsilon}(M) \longrightarrow 
         \Diff^\cyl(M, \Sigma^\varepsilon) \longrightarrow 
         \Diff^\cyl(\Sigma^\varepsilon) .
    \]
    On numerators the map splits as a product of the map on the left factor and the map on the right factor.
    For the left factor we pick a base-point consisting of some $(W_v,\varphi_v)_{v \in V_\Gamma^{\rm sph}}$ in $\prod_{v \in V_\Gamma^{\rm sph}}\BDiff_{H_v}^{\rm sph}(S^3)$ and we have a homotopy fibre sequence
    \[
    \begin{tikzcd}
        {\prod_{v \in V_\Gamma^{\rm sph}} \emb_{\partial}((M \setminus \Sigma^\varepsilon_\circ)_v, W_v)} \rar &
        {\emb(M, \bbR^\infty)} \rar &
        {\prod_{v \in V_\Gamma^{\rm sph}} \BDiff_{H_v}^{\rm sph}(S^3)}
    \end{tikzcd}
    \]
    For the right term we pick as a base-point embeddings $\iota_i^0\colon S^2 \times H_{\sigma(i)} \hookrightarrow P_i$ and we have a homotopy fibre sequence
    \[
    \begin{tikzcd}
        {\prod_{i=1}^n \emb_{\iota_i^0}(P(\Sigma \subset M)_i, P_i)} \rar &
        {\prod_{i=1}^n \emb(P(\Sigma \subset M)_i, P_i)} \rar &
        {\prod_{i=1}^n \emb^{\rm D}(H_{\sigma(i)} \times S^2, P_i)} 
    \end{tikzcd}
    \]
    where the embeddings $\iota\colon P(\Sigma \subset M)_i \hookrightarrow P_i$ must agree with $\iota_i^0$ on the boundary.
    Let $P_i' \subset P_i$ be the submanifold obtained by removing the open discs bounded by $\iota_i^0(H_{\sigma(i)} \times S^2)$ -- then $\iota$ exactly induces a diffeomorphism $P(\Sigma \subset M)_i \cong P_i$ with a prescribed value on the boundary.
    In summary the homotopy fibre of $\Theta$ can be written as the quotient
    \[
        \frac{
            \prod_{v \in V_\Gamma^{\rm sph}} \emb_{\partial}((M \setminus \Sigma^\varepsilon_\circ)_v, W_v)
            \times 
            \prod_{i=1}^n \emb_\partial(P(\Sigma \subset M)_i, P_i')
        }{
            \prod_{v \in V_\Gamma} \Diff_\partial( (M \setminus \Sigma^\varepsilon_\circ)_v )
        }.
    \]
    Here numerator and denominator are equivalent via the action map, so the quotient is contractible,
    which shows that $\Theta$ is a homotopy equivalence.
\end{proof}

\begin{rem}
    In the above lemma we crucially used that there is an isomorphism $\Or(3) \cong \SO(3) \times \bbZ/2$, which allowed us to pick a central orientation reversing element of order $2$ -- the antipodal map.
    In even dimensions we would instead have to pick a reflection $R$, choose the maps $\psi_h\colon S^2 \cong \Sigma_{\alpha(h)}$ to satisfy $\psi_h \circ R = \psi_{h^\dagger}$, and instead of taking the quotient by $\prod_{E_\Gamma} \Diff^+(S^2)$ we would have to quotient by the group of families $(\gamma_h)_{h \in H_\Gamma}$ with $\gamma_{h^\dagger} = R \circ \gamma_{h} \circ R$.
    Luckily, in dimension $3$ we can just take $R$ to be the antipodal map, which is central in $\Or(3)$.
\end{rem}

We are now ready to prove \cref{prop:C-Gamma configuration model}.
\begin{proof}[Proof of \cref{prop:C-Gamma configuration model}]
    By \cref{cor:Cthick-to-C} and \cref{lem:Cthick-to-modified-Conf} we already have equivalences 
    \[
        \rmC^\Gamma_\sigma(P_1,\dots,P_n) \leftarrow 
        \Cthick^\Gamma_\sigma(P_1,\dots,P_n) \to
        \frac{\prod_{v \in V_\Gamma^{\rm sph}} \BDiff_{H_v}^{\rm sph}(S^3) \times \prod_{i=1}^n \emb^{\rm D}(H_{\sigma(i)} \times S^2, P_i)}{ \prod_{E_\Gamma} \Diff^+(S^2) }
    \]
    that are $\Aut(\Gamma,\sigma) \times \prod_i \Diff(P_i)$ equivariant.
    We can further replace the group in the denominator by the equivalent group $\prod_{E_\Gamma} \SO(3)$.
    Then we can use \cref{lem:punctured-sphere-config} and \cref{lem:embplus} to rewrite the quotient as 
    \[
        \frac{\underline{\prod_{v \in V_\Gamma^{\rm sph}} \left(\Conf_{H_v}^{\fr+}(S^3) \sslash \SO(4) \right) \times \prod_{i=1}^n \Conf_{H_{\sigma(i)}}^{\fr+}(P_i)}}{ \prod_{E_\Gamma} \SO(3) }.
    \]
    This is then equivalent to 
    \[
        \frac{\underline{\prod_{v \in V_\Gamma^{\rm sph}} \Conf_{H_v}^{\fr+}(S^3) \times \prod_{i=1}^n \Conf_{H_{\sigma(i)}}^{\fr+}(P_i)}}{ \prod_{v \in V_\Gamma^{\rm sph}} \SO(4) \times \prod_{E_\Gamma} \SO(3) },
    \]
    which is the desired configuration space model.
\end{proof}

\subsection{\texorpdfstring{$C^\Gamma_\sigma$}{C-Gamma} as a homotopy pullback}\label{subsec:C-pullback}
  
From our discussion in the previous subsection, we can construct two types of maps out of $\Cthick_\sigma^\Gamma$.
\begin{enumerate}
    \item 
    For every $v \in V_\Gamma^{\rm sph}$ we have
    \begin{align*}
        \varphi_v\colon \Cthick_\sigma^\Gamma(P_1,\dots, P_n) & 
        \longrightarrow \frac{\BDiff_{H_v}^{\rm sph}(S^3)}{\Map(H_v, \Diff^+(S^2))}  
        \simeq \dots \simeq \Conf_{H_v}(S^3)\sslash \SO(4)\\
        [\Sigma \subset W, \alpha, \iota] & \longmapsto [(W \setminus \Sigma^\varepsilon_\circ)_{\alpha(v)}, \alpha_{|H_v}]
    \end{align*}
    that takes the component of $W \setminus \Sigma^\varepsilon_\circ$ corresponding to $v$ and remembers the identification of its set of boundary components with the set of half-edges $H_v$ at $v$.
    \item 
    For $1 \le i \le n$ we define a map 
    \begin{align*}
        \psi_i\colon \Cthick_\sigma^\Gamma(P_1,\dots, P_n) & \longrightarrow 
        \frac{\emb^{\rm D}(H_{\sigma(i)} \times S^2, P_i)}{\Map(H_{\sigma(i)}, \Diff^+(S^2))} \simeq \dots \simeq \Conf_{H_{\sigma(i)}}(P_i)\\
        [\Sigma \subset W, \alpha, \iota] & \longmapsto [\iota( \amalg_{h \in H_{\sigma(i)}} \Sigma_{\alpha(h)})]
    \end{align*}
    that records the (unparametrised) bounding spheres obtained as the image of $\Sigma$ in $P_i$.
    \item 
    Similarly to \cref{lem:Cthick-to-modified-Conf}, we have a map 
    \begin{align*}
        \xi\colon \Cthick^\Gamma_\sigma(P_1,\dots,P_n) &\longrightarrow \Map^{C_2}(H_\Gamma, \BDiff^+(S^2)) \\
        [\Sigma \subset W, \alpha, \iota] & \longmapsto (h \mapsto \Sigma_{h})
    \end{align*}
    that records the oriented sphere associated to each half-edge, with the orientation being induced by the coorientation and the orientation of $W$.
    Here $C_2$ acts on $\BDiff^+(S^2)$ by orientation reversal.
    
\end{enumerate}

Together, these maps form the top map in a homotopy pullback square describing $\Cthick^\Gamma_\sigma$ in terms of more local data associated to vertices and edges of $\Gamma$.
This pullback square is closely related to the pullback square that witnesses the Segal condition for the modular $\infty$-operads of $3$-manifolds under connected sum.

\begin{prop}\label{prop:rmC-pullback}
    For all $(\Gamma, \sigma)  \in \Gr_{g,n}$ and $P_1,\dots,P_n$, the $\Aut(\Gamma,\sigma) \times \prod_{i=1}^n \Diff^+(P_i)$-equivariant commutative square 
    \begin{equation}\label{eqn:C-pullback-precise}
    \begin{tikzcd}[column sep=large]
    	\Cthick^\Gamma_\sigma(P_1,\dots,P_n) & 
        \prod\limits_{v \in V^{\rm sph}} \frac{\BDiff_{H_v}^{\rm sph}(S^3)}{\Map(H_v, \Diff^+(S^2))} \times \prod\limits_{i=1}^n \frac{\emb^{\rm D}(H_{\sigma(i)} \times S^2, P_i)}{\Map(H_{\sigma(i)}, \Diff^+(S^2))}\\ 
    	{\Map^{C_2}(H_\Gamma, \BDiff^+(S^2))} & {\prod_{H_\Gamma} \BDiff^+(S^2)}.
    	\arrow[from=2-1, to=2-2, hook, "\Delta^\pm"]
    	\arrow[from=1-1, to=2-1, "\xi"']
    	\arrow[from=1-1, to=1-2, "{(\prod_v \varphi_v, \prod_i \psi_i)}"]
    	\arrow[from=1-2, to=2-2]
    \end{tikzcd}
    \end{equation}
    is a homotopy pullback square.
    Here, the right vertical map records boundary spheres.
\end{prop}
\begin{proof}
     The square is commutative and equivariant by inspection.
     Indeed, the action of $\Aut(\Gamma,\sigma)$ is given by permuting $S^3$-vertices and half-edges and does not change the underlying manifolds, and the action of $\Diff^+(P_i)$ is trivial on the spaces in the bottom row.

    To see that the square is a homotopy pullback, it suffices to compare the homotopy fibres of the vertical maps. 
    The spaces in the bottom row (and in fact also those in the top row) are connected, so it will suffice to consider the homotopy fibre at one point.
    As $B(-)$ commutes with finite products we may rewrite $\prod_{H_\Gamma} \BDiff^+(S^2) \cong B\big(\prod_{H_\Gamma} \Diff^+(S^2)\big)$ and similarly for $\prod_{E_\Gamma}$.
    Next we replace $\Cthick$ by the homotopy equivalent quotient from \cref{lem:Cthick-to-modified-Conf} -- the maps in the square clearly factor through this.
    Now both vertical maps are of the form $X \to *\hq G$ for a principal $G$-action on $X$ and as such their (homotopy) fibre is $X$.
    In both cases the space $X$ is the same, namely the product of the $\BDiff_{H_v}^{\rm sph}(S^3)$ and the $\emb^{\rm D}(H_{\sigma(i)} \times S^2, P_i)$, and in particular the fibres are indeed the same.
\end{proof}

By taking the homotopy fibre of the top map in \cref{prop:rmC-pullback}, we can fit $\rmC^\Gamma_\sigma(\dots)$ into a homotopy fibre sequence that will be useful for computations.
To better understand the resulting fibre sequence and spectral sequence we first consider the homotopy fibre sequence coming from the bottom map in \cref{prop:rmC-pullback}.

Let $\rho \in H^4(\BDiff^+(S^2))$ denote the generator that corresponds to the first Pontryagin class on $\BSO(3)$, and let $\beta \in H^3(\SO(3);\bbQ)$ denote the generator that transgresses to the first Pontryagin class.
(That is, $\beta$ is a fixed non-zero multiple of the dual of the fundamental class of $\SO(3) \cong \mathbb{RP}^3$, but the exact prefactor will not matter.)

\begin{lem}\label{lem:bottom-fibre-sequence}
    The bottom map in \cref{prop:rmC-pullback} fits into an $\Aut(\Gamma,\sigma)$-equivariant homotopy fibre sequence
    \[
    \Map^{C_2}(H_\Gamma, \SO(3))
    \longrightarrow
    \Map^{C_2}(H_\Gamma, \BDiff^+(S^2)) 
    \xrightarrow[\qquad]{\Delta^\pm}
    \prod_{H_\Gamma} \BDiff^+(S^2),
    \]
    where $C_2$ acts on $\SO(3)$ by inverting.
    The fibre is (non-canonically) equivalent to $\prod_{E_\Gamma} \SO(3)$.
    The cohomology of the fibre and base are
    \[
        \frac{\Lambda_\bbQ \{ \beta_h : h \in H_\Gamma \}}{\langle \beta_h + \beta_{h^\dagger} : h \in H_\Gamma \rangle}
        \qquad\text{ and }\qquad
        \bbQ[\rho_h : h \in H_\Gamma].
    \]
    In the Serre spectral sequence for this fibre sequence $d_2$ and $d_3$ are trivial, and $d_4$ is determined by $d_4(\beta_h) = \rho_h - \rho_{h^\dagger}$.
\end{lem}
\begin{proof}
    The map $\Delta^\pm$ is a subspace inclusion, namely the subspace of those $H_\Gamma$-tuples of oriented spheres $(S_h)_{h \in H_\Gamma}$ where for every  half-edge $h \in H_\Gamma$ we have $S_{h^\dagger}= S_h^-$, \emph{i.e.}~$S_h$ and $S_{h^\dagger}$ are the same submanifold of $\bbR^\infty$, but with opposite orientation.
    To compute the homotopy fibre of $\Delta^\pm$, we factor it $\Aut(\Gamma,\sigma)$-equivariantly as a composite of an equivalence and a Serre fibration
    \begin{align*}
        \Map^{C_2}(H_\Gamma, \BDiff^+(S^2)) 
        & \hookrightarrow
        X = \{ (S_h, \psi_h)_{h \in H_\Gamma}  \;|\; S_h \in \BDiff^+(S^2), \psi_h\colon S_h \cong S_{h^\dagger}^-, \psi_{h^\dagger} = \psi_h^{-1}\} \\
        & \xrightarrow{\Delta'} \prod_{H_\Gamma} \BDiff^+(S^2)
    \end{align*}
    where the first map includes the subspace where $\psi_h$ is the identity map for all $h$, and the second map forgets the $\psi_h$.
    (To see that the inclusion is an equivalence, note that, after choosing an orientation of the graph, $X$ is equivalent $\prod_{E_\Gamma} \BDiff^+(S^2)$.)
    We can now compute the homotopy fibre of $\Delta^\pm$ as the strict fibre of $\Delta'$.

    As a base-point in $\prod_{H_\Gamma} \BDiff^+(S^2)$ we choose the constant tuple at the standard sphere $S^2$ with its standard orientation, this is an $\Aut(\Gamma,\sigma)$ fixed point and hence the fibre 
    \[
        F = \{ (\psi_h)_{h \in H_\Gamma} \in \prod_{H_\Gamma} \Diff(S^2) \;|\; \psi_h \text{ orientation-reversing and } \psi_{h^\dagger} = \psi_h^{-1}\}.
    \]
    still has an $\Aut(\Gamma, \sigma)$-action.
    Using that the antipodal map $(-\id_{S^2})$ is a central element in $\Or(3)$, we can construct an $\Aut(\Gamma,\sigma)$-equivariant map
    \begin{align*}
        \Map^{C_2}(H_\Gamma, \SO(3)) & \longrightarrow F \\
        (\gamma_h)_{h \in H_\Gamma}  & \longmapsto (\psi_h = \gamma_h \circ (-\id_{S^2}))_{h \in H_\Gamma}
    \end{align*}
    which is moreover a homotopy equivalence.
    Here, as in the statement of the lemma, $C_2$ acts on $\SO(3)$ by inverting.
    This identifies the homotopy fibre $\Aut(\Gamma,\sigma)$-equivariantly, as claimed.

    To compute the cohomology, recall that 
        $H^*(\SO(3)) \cong \Lambda_\bbQ\langle \beta \rangle$
        and 
        $H^*(\BSO(3)) \cong \bbQ[\rho]$.
    The inversion map $(-)^{-1}\colon \SO(3) \to \SO(3)$ acts by $\beta \mapsto -\beta$ on the cohomology: we can compute its degree by noting that its derivative at the identity is $(-1)\colon \mathfrak{so}(3) \to \mathfrak{so}(3)$.
    The description of the cohomology of the base and fibre now follows from the K\"unneth formula.

    In the Serre spectral sequence we must have $d_2(\beta_i)=0$ and $d_3(\beta_i)=0$, which via the Leibniz rule implies $d_2 \equiv 0 \equiv d_3$.
    To compute the $d_4$-differentials, it will similarly suffice to determine $d_4(\beta_h)$.
    Choose an orientation on each edge and let $H_\Gamma^+ \subset H_\Gamma$ be the subset of positively oriented half-edges.
    Consider the map of homotopy fibre sequences
    \[
    \begin{tikzcd}
        \prod_{H_\Gamma^+} \SO(3) \rar \dar["\cong", "i"'] & P \rar \dar \ar[dr, phantom, very near start, "\lrcorner"] & \prod_{H_\Gamma^+} \BSO(3) \dar["j"] \\
        \Map^{C_2}(H_\Gamma, \SO(3)) \rar & X \rar & {\prod_{H_\Gamma} \BDiff^+(S^2)}
    \end{tikzcd}
    \]
    where the map $j$ labels every negative half-edge by the standard sphere and $P$ denotes the pullback.
    The top fibre sequence is equivalent to the product, over the set of edges, of the standard fibre sequence $\SO(3) \to \mathrm{ESO}(3) \to \BSO(3)$.
    We defined $\beta$ such that it transgresses to $\rho$ in this fibre sequence, so for every $h \in H_\Gamma^+$ we have $j^*d_4(\beta_h) = d_4(i^*\beta_h) = \rho_h$.
    On cohomology $j^*$ sends $\rho_h$ to $\rho_h$ if $h \in H_\Gamma^+$ and $0$ otherwise.
    Therefore, we get $d_4(\beta_h) = \rho_h + \varepsilon$ where $\varepsilon$ is in the span of $\{\rho_{h'} \;|\; h' \in H_\Gamma^-\}$.
    Repeating the same argument for $H_\Gamma^-$ in place of $H_\Gamma^+$ we see $d_4(\beta_{h^\dagger}) = \rho_{h^\dagger} + \varepsilon'$.
    Since $\beta_{h^\dagger} = -\beta_h$, this implies $d_4(\beta_h) = \rho_h - \rho_{h^\dagger}$ as claimed.
\end{proof}

\begin{cor}\label{cor:C-fibre-sequence}
    For every $(\Gamma,\sigma)$ and $P_1,\dots,P_n$, there is a homotopy fibre sequence
    \[
    \Map^{C_2}(H_\Gamma, \SO(3))
    \longrightarrow
	\rmC^\Gamma_\sigma(P_1,\dots,P_n) 
    \longrightarrow
    \prod\limits_{v \in V_\Gamma^{\rm sph}} \Conf_{H_v}(S^3) \hq \SO(4) \times \prod\limits_{i=1}^n \Conf_{H_{\sigma(i)}}(P_i)
    \]
    whose fibre is (non-canonically) equivalent to $\prod_{E_\Gamma} \SO(3)$.
\end{cor}
\begin{proof}
    Because the square in \cref{prop:rmC-pullback} is a homotopy pullback, the horizontal homotopy fibres are equivalent, and we can read of the homotopy fibre of the bottom map from \cref{lem:bottom-fibre-sequence}.
\end{proof}

The Serre spectral sequence for the fibre sequence in \cref{cor:C-fibre-sequence} will be an effective tool for computing the (rational) cohomology of $\rmC^\Gamma_\sigma$.
In \cref{prop:C-spectral-sequence} we record the basic structure of this spectral sequence and show that it is $\Aut(\Gamma,\sigma) \times \prod_i \pi_0\Diff^+(P_i)$-equivariant.
First, we establish a lemma about suitable generalities of such spectral sequences.

\begin{lem}\label{lem:sp-seq-with-action}
    Let $G$ be a topological group and suppose we have a $G$-equivariant square of spaces
    \[\begin{tikzcd}
        E \rar["\pi"] \dar["F"'] & B \dar["f"] \\
        {E'} \rar["{\pi'}"] & {B'}
    \end{tikzcd}\]
    that is also a homotopy pullback square.
    Suppose further that $B'$ is simply connected and admits a $G$-fixed point $b_0 \in B'$, so the homotopy fibre $\hofib_{b_0}(\pi')$ inherits a $G$-action.
    Then there is a convergent multiplicative spectral sequence
    \[
        E_2^{p,q} = H^p(B; H^q(\hofib_{b_0}(\pi')) \Rightarrow H^{p+q}(E)
    \]
    and this spectral sequence is $\pi_0(G)$-equivariant for the diagonal $\pi_0(G)$-action on $H^p(B; H^q(\hofib_{b_0}(\pi'))$ and the usual action on $H^{p+q}(E)$.
\end{lem}
\begin{proof}
    For every map $\pi\colon E \to B$ there is a convergent multiplicative spectral sequence
    \[
        E_2^{p,q} = H^p(B; H^q(\pi)) \Rightarrow H^{p+q}(E)
    \]
    where $H^q(\pi)$ is the local system on $B$ that assigns to each point $b \in B$ the abelian group $H^q(\hofib_b(\pi))$, see \cite{Dress1967}.
    Note that we do not have to pick a base-point in $B$ to make sense of this spectral sequence.
    This spectral sequence is natural by construction:
    any square as above (without assuming that it is a homotopy pullback square) induces a map of local coefficient systems $\alpha_f\colon f^*H^q(\pi') \to H^q(\pi)$ and a map on $E_2$-pages 
    \[
        H^p(B'; H^q(\pi')) \xrightarrow[\qquad]{f^*} H^p(B; f^*H^q(\pi')) 
        \xrightarrow[\qquad]{(\alpha_f)_!} H^p(B; H^q(\pi)),
    \]
    which induces a map of spectral sequences compatible the the map $F^*$ in the abutment.

    Via naturality, left-multiplication $\lambda_g$ by an element $g \in G$ induces an automorphism of the spectral sequence.
    These assemble into the desired $\pi_0(G)$-action on the aforementioned spectral sequence.
    The action on the $E_2$-page comes from maps $\alpha_{\lambda_g}\colon \lambda_g^* H^q(\pi) \to H^q(\pi)$.
    (In other words, $H^q(\pi)$ has $\pi_0(G)$-fixed point data in the $\pi_0(G)$-category $\mathrm{Fun}(\Pi B, \mathrm{Mod}_\bbQ)$ of local systems on $B$.)
    To prove the lemma, it will suffice to show that the local coefficient system $H^p(\pi)$ on $B$ is the trivial coefficient system with value $H^q(\hofib_{b_0}(\pi))$ and the maps $\alpha_{\lambda_g}$ are just the maps coming from acting with $g$ on $\hofib_{b_0}(\pi)$.

    The map $f$ induces a map of local coefficient systems $\beta\colon f^*H^q(\pi') \to H^q(\pi)$.
    Because we assumed that the square is a homotopy pullback square, $f$ induces a homotopy equivalence $\hofib_b(\pi) \simeq \hofib_{f(b)}(\pi')$ for all $b \in B$, which implies that $\beta$ is an isomorphism.
    But since $B'$ is assumed to be simply connected, $H^q(\pi')$ is a constant coefficient system and hence its pullback is also the constant coefficient system with the same value.
    It remains to determine the $\pi_0(G)$-action on this coefficient group.
    Because $b_0$ is fixed by $G$ we have a $G$-equivariant homotopy pullback diagram 
    \[\begin{tikzcd}
        {\hofib_{b_0}(f)} \dar \rar & \{b_0\} \dar["i"', hook] \\
        {E'} \rar["{\pi'}"] & {B'},
    \end{tikzcd}\]
    where the induced map in local coefficient systems is a $\pi_0(G)$-equivariant isomorphism $i^*H^q(\pi') \cong H^q(\hofib_{b_0}(\pi'))$.
    This shows that $H^q(\pi)$ is indeed the constant coefficient system at $H^q(\hofib_{b_0}(\pi'))$ with its usual $\pi_0(G)$-action.
\end{proof}

\begin{prop}\label{prop:C-spectral-sequence}
    There exists a multiplicative $\Aut(\Gamma, \sigma) \times \prod_i \pi_0\Diff^+(P_i)$-equivariant Serre spectral sequence converging to $H^*(\rmC_\sigma^\Gamma(P_1,\dots,P_n); \bbQ)$, with $E_4$-page
    \[
        E_4^{*,*} \cong 
        \frac{\Lambda_\bbQ \{ \beta_h : h \in H_\Gamma \}}{\langle \beta_h + \beta_{h^\dagger} : h \in H_\Gamma \rangle}
        \otimes \bigotimes_{v \in V_\Gamma^{\rm sph}} H^*\left(\frac{\underline{\Conf_{H_v}(S^3)}}{\SO(4)}; \bbQ\right) \otimes \bigotimes_{i=1}^n H^*(\Conf_{H_{\sigma(i)}}(P_i); \bbQ)
    \]
    where each $\beta_h$ has bidegree $(0,3)$ and the other tensor factors are contained in the $0$th row $E_4^{*,0}$.
    The $d_4$ differential is determined by
    \[
        d_4(\beta_h) = \delta_{r(h)} - \delta_{r(h^\dagger)}
    \]
    where we set $\delta_v = 0$ if $v$ is a marked vertex, and otherwise $\delta_v$ is the pullback of the first Pontryagin class $p_1 \in H^4(\BSO(4))$ along the map $\Conf_{H_v}(S^3)\sslash \SO(4) \to * \sslash \SO(4)$.
\end{prop}
\begin{proof}
    We apply \cref{lem:sp-seq-with-action} to the square from \cref{eqn:C-pullback-precise}, by taking vertical homotopy fibres in \cref{prop:rmC-pullback} and taking $G := \Aut(\Gamma, \sigma) \times \prod_{i=1}^n \Diff^+(P_i)$.
    Note that $\prod_{H_\Gamma} \BDiff^+(S^2)$ is indeed simply connected because $\BSO(3)$ is, and that we can pick a $G$-invariant base-point $b_0$ in the bottom right by assigning the same oriented sphere to all half-edges, so that $\mathrm{hofib}(\Delta^\pm)$ inherits a $G$-action.
    (However, we will not generally be able to pick a $G$-invariant base-point in the top right, and thus the top map does not give a (homotopy) fibre sequence in spaces with $G$-action.)
    Nevertheless, \cref{lem:sp-seq-with-action} applies and we obtain a $\pi_0(G)$-equivariant multiplicative Serre spectral sequence
    \[
        E_2^{p,q} = H^p(B ; H^q(\mathrm{hofib}_{b_0}(\Delta^\pm); \bbQ)) \Rightarrow H^{p+q}(E)
    \]
    We saw in \cref{lem:bottom-fibre-sequence} that for the bottom fibre sequence we have
    \[
        \hofib(\Delta^\pm) \simeq 
        \frac{\prod_{H_\Gamma} \SO(3)}{\prod_{E_\Gamma} \SO(3)}
        \qquad\text{and}\qquad
        H^*(\hofib(\Delta^\pm); \bbQ) \cong \frac{\Lambda_\bbQ\{ \beta_h : h \in H_\Gamma\}}{ \langle \beta_h + \beta_{h^\dagger} : h \in H_\Gamma \rangle}.
    \]
    We thus get a $\pi_0(G)$-equivariant spectral sequence converging to the cohomology of $\Cthick^\Gamma_\sigma$, with $E_2$-page
    \[
        E_2^{*,*} \cong 
        \frac{\Lambda_\bbQ \{ \beta_h : h \in H_\Gamma \}}{\langle \beta_h + \beta_{h^\dagger} : h \in H_\Gamma \rangle}
        \otimes \bigotimes_{v \in V_\Gamma^{\rm sph}} H^*\left(\frac{\underline{\Conf_{H_v}(S^3)}}{\SO(4)}; \bbQ\right) \otimes \bigotimes_{i=1}^n H^*(\Conf_{H_{\sigma(i)}}(P_i); \bbQ).
    \]
    Here, to describe the cohomology of the base we used the same identifications as before and then applied the K\"unneth theorem.
    Note that $\pi_0\Diff^+(P_i)$ only acts on $\Conf_{H_{\sigma(i)}}(P_i)$.

    The $d_2$ and $d_3$ differentials are $0$ on $\beta_h$ for degree reasons, and thus the Leibniz rule implies that all $d_2$ and $d_3$ differentials are $0$, so the $E_4^{*,*}$ are still of the above form.
    On the $E_4$-page all differentials are determined, through the Leibniz rule, by $d_4(\beta_h)$, which we will now describe.
    For the bottom fibre sequence, we saw in \cref{lem:bottom-fibre-sequence} that $d_4(\beta_h) = \rho_h - \rho_{h^\dagger}$.

    It hence remains to verify that the pullback of each $\rho_h$ gives $\delta_v$ for $v=r(h)$ as defined in the statement.
    For this, we have to take the pullback of the first Pontryagin class in $H^4(\BDiff^+(S^2))$ (the factor corresponding to $h$) under the right vertical map in \cref{prop:rmC-pullback}.
    First consider the case of an unmarked vertex.
    We have a homotopy commutative diagram where $s_h$ records the sphere associated to $h \in H_v$ (we are trying to show $s_h^*\rho_h = \delta_v$) and $T_h$ records the tangent space at the configuration point associated to $h$.
    \[
    \begin{tikzcd}
        \frac{\BDiff_{H_v}^{\rm sph}(S^3)}{\prod_{H_v} \Diff^+(S^2)} \dar["s_h"] &
        {\Conf_{H_v}(S^3) \sslash \SO(4)} \rar \dar["T_h"] \lar["\simeq"] & {* \sslash \SO(4)} \dar[equal] \\
        {\BDiff^+(S^2)} & 
        {\BSO(3)} \rar["-\oplus \bbR"] \lar["\simeq"] & {\BSO(4)}
    \end{tikzcd}
    \]
    To construct the left square, we proceed as follows:
    the bottom map sends an oriented $3$-dimensional sub vector space $V \subset \bbR^\infty$ to its unit sphere $S(V)$.
    A point in $\Conf_{H_v}(S^3)$ can be thought of as an oriented $4$-dimensional sub vector space $W \subset \bbR^\infty$ together with an embedding $\iota\colon H_v \hookrightarrow S(W)$. 
    The top map takes such a tuple $(W,\iota)$ and sends it to the submanifold of $\bbR^\infty$ obtained by removing open $\varepsilon$-discs around each point of the configuration.
    (We have to vary $\varepsilon$ continuously with the configuration, but this is not an issue.)
    The middle vertical map sends $(W,\iota)$ to the tangent space $T_{\iota(h)}S(W)$ of $S(W)$ at the point $\iota(h)$.
    The right square commutes up to homotopy because the two ways of going around the square either send $(W,\iota)$ to $W$ or to $T_{\iota(h)}S(W) \oplus \bbR$, but these are canonically isomorphic as the normal bundle of $S(W)$ in $W$ is trivial.
    On cohomology this diagram shows us that $s_h^*(\rho_h)$ is indeed the pullback of the first Pontryagin class.
    
    For marked vertices we can similarly argue that the map 
    \[
        s_h\colon \frac{\emb^{\rm D}(H_{\sigma(i)} \times S^2, P_i)}{\prod_{H_{\sigma(i)}} \Diff^+(S^2)} \longrightarrow \BDiff^+(S^2)
    \]
    is equivalent to the map $\Conf_{H_{\sigma(i)}}(P_i) \to \BSO(3)$ that records the tangent space $T_{\iota(h)}P_i$.
    Like every orientable $3$-manifold, $P_i$ admits a framing, so we can choose a (non-canonical) null-homotopy of this map and thus $s_h^*(\rho_h) = 0$ in this case.
\end{proof}

\section{Rational cohomology computations}\label{sec: computation}

The aim of this section is to lay the foundations for a computational  framework that can be used to calculate the rational cohomology of the fibre $\mathcal{H}_g(P_1,\dots,P_n)$ of $\W$. More precisely, we provide explicit presentations for the rational cohomology ring of $\mathcal{H}_\sigma^\Gamma(P_1,\ldots,P_n)$ as an $\Aut(\Gamma,\sigma) \times \prod_{i=1}^n \pi_0\Diff^+(P_i)$-module, together with a description of the induced maps from morphisms in $\Gr_{g,n}$. Throughout this section, all cohomology groups will be with rational coefficients, unless otherwise noted.

We begin by describing the rational cohomology rings of three spaces, with each description building on the previous. First, in \cref{sec:conf-cohomology}, we consider the cohomology ring of the ordered configuration space $\Conf_d(S^3)$ as a representation of the symmetric group $\Sym_d$.  Then in \cref{sec:punctured-moduli}, we use this to compute the cohomology ring of the classifying space of the diffeomorphism group of~$S^3$ fixing $d$ points, which can be identified with the homotopy quotient of $\Conf_d(S^3)\hq \SO(4)$. 
Thirdly, in \cref{subsection:cohomology of D}, we input the previous step into the spectral sequence from \cref{prop:C-spectral-sequence} to obtain a presentation for the rational cohomology of $\mathcal{H}^\Gamma_\sigma(P_1,\ldots,P_n)$ in \cref{thm:DGamma Cohomology Ring}. The final section describes the effect of an edge collapse in terms of the generators of the cohomology ring. This is most concrete in the case when there are no irreducible prime factors (see \cref{rem:beta-ses-splitting}), and in \cref{sec:extended example}, we will apply the results obtained here to calculate the rational cohomology ring of $\BDiff^+(U_2)$, where $U_g=(S^1\times S^2)^{\sharp g}$.

Taking a broader view, configuration spaces of points in $S^3$ assemble into a functor
    \[
        \Conf_\bullet(S^3) \colon \FI^{\rm op} \longrightarrow \Top, \qquad
        A \mapsto \emb(A, S^3) 
    \]
    from the opposite of the category of finite sets and injections to the category of topological spaces.
    This functor sends the set $\{1,\dots, d\}$ to the ordered configuration space $\Conf_d(S^3)$, and remembers the $\Sym_d$-action given by permuting points, as well as the maps $\Conf_d(S^3) \to \Conf_k(S^3)$ given by forgetting some of the points.
    Taking cohomology we obtain a functor
    \[
        H^*(\Conf_\bullet(S^3)) \colon \FI \longrightarrow \mathrm{Ring},
    \]
    \emph{i.e.}~an $\FI$-ring. 
    Similarly, $H^*(\Conf_d(S^3)\hq \SO(4)) \cong H^*(\BDiff^+_{\{1,\dots,d\}}(S^3))$ forms an $\FI$-ring that we study in \cref{sec:punctured-moduli}.
    While the study of configuration spaces of points in manifolds is classical and dates back to early last century (see, for example, \cite{Fadell-Husseini-2001} and the references therein), the $\FI$-perspective is relatively recent \cite{Church-Ellenberg-Farb-2015, Church-Ellenberg-Farb-Nagpal-2014}. In fact, one knows from general results of Church--Ellenberg--Farb \cite{Church-Ellenberg-Farb-2015} that the $\FI$-ring $H^*(\Conf_\bt(S^3))$ is finitely generated and for each fixed degree $q$ the module $H^q(\Conf_\bt(S^3))$ exhibits representation stability as a representation of the symmetric group $\Sym_d$. 
    Although we will not explicitly rely on such results here, the $\FI$-perspective is useful to keep in mind as a guiding principle.
    We will periodically return to it throughout the section, recasting results in these terms.
    For example, we give presentation of $H^*(\Conf_\bt(S^3))$ as an $\FI$-ring with two generators and four relations in \cref{rem:FI}.

\subsection{The rational cohomology ring of $\Conf_d(S^3)$}\label{sec:conf-cohomology}

We first calculate the rational cohomology of the configuration space of $d$ ordered points $\Conf_d(S^3)$ as a representation of the symmetric group $\Sym_d$. Let the $d$ marked points be $\{x_1,\ldots, x_d\}$. Recalling that $S^3\cong \SU_2$, we can use the left action of $\SU_2$ on itself to obtain diffeomorphisms
\begin{equation}\label{eqn:Conf Isomorphism}
    \varphi_i\colon\Conf_d(S^3)\rightarrow S^3\times \Conf_{d-1}(\bbR^3)
\end{equation}
where the $S^3$-factor corresponds to $x_i$, while the remaining coordinates in $\Conf_{d-1}(\bbR^3)$ correspond to $x_i^{-1}x_j$ for $i\neq j$. We will refer to the $\varphi_i$ as \emph{trivialisations}.

For $1\leq i\leq d$ we also have projections $\pi_i\colon \Conf_d(S^3)\rightarrow \Conf_1(S^3)\cong S^3$ that forget all but the point $x_i$. Let $\alpha_i=\pi_i^*(\mu)\in H^3(\Conf_d(S^3))$ be the pullback of the fundamental class $\mu\in H^3(S^3).$

\begin{lem}\label{lem:H3 Class} The classes $\alpha_i$, $\alpha_j$ are equal for all $i,j$. 
\end{lem}
\begin{proof}
    It suffices to prove the lemma when $d=2$ and the two marked points are $\{x_1,x_2\}$.  The trivialisations $\varphi_1,\varphi_2$ in  Equation \ref{eqn:Conf Isomorphism} yield diffeomorphisms 
    \begin{eqnarray*}
      \bbR^3\times S^3\cong\Conf_1(\bbR^3)\times S^3  \xleftarrow{\varphi_2}\Conf_2(S^3)\xrightarrow{\varphi_1} S^3\times \Conf_1(\bbR^3)\cong S^3\times \bbR^3
    \end{eqnarray*}
Using the trivialisation $\varphi_1$, the class $\alpha_1$ is the pullback of $\mu$ under projection onto the $S^3$-factor. Hence $\alpha_1$ is hom dual to the class $S^3\times \{0\}\subseteq S^3\times \bbR^3$. In $\Conf_2(S^3)$, this corresponds to the 3-sphere $\{(x_1,-x_1)\mid x_1\in S^3\}$, which projects diffeomorphically under $\varphi_2$ to $\{0\}\times S^3$. Since the latter is dual to $\alpha_2$, we just need to compute the degree of the induced map on $S^3$. From the description above the composite map $S^3\rightarrow S^3$ is given by $x_1\mapsto -x_1$, \emph{i.e.} the antipodal map. This has degree 1 since $S^3$ is odd-dimensional.  Therefore, $\alpha_1=\alpha_2$, as claimed.
\end{proof}

For $i,j,k$ pairwise distinct consider the composition  \[\Conf_d(S^3)\rightarrow \Conf_3(S^3)\xrightarrow[]{\varphi_k} S^3\times \Conf_2(\bbR^3)\rightarrow \Conf_2(\bbR^3)\]  
where the first map forgets all but $x_i$, $x_j$ and $x_k$, and the $S^3$-factor in the trivialisation corresponds to $x_k$.  In other words, $x_k$ corresponds to the point at infinity in $\Conf_2(\bbR^3)$, while the two points correspond to $x_i$ and $x_j$. Using the action of $\bbR^3$ on itself by translation, we can move the image of $x_i$ to the origin in $\bbR^3$ to obtain a further trivialisation
\begin{eqnarray*}
\tau_i^k\colon\Conf_2(\bbR^3)&\rightarrow& \bbR^3\times (\bbR^3\setminus \{0\})\\
(v_i,v_j)&\mapsto&  (v_i,v_j-v_i)
\end{eqnarray*}
where $v_i$, $v_j$ are the image of $x_i$ and $x_j$ respectively. Similarly, we have a trivialisation $\tau_j^k$ translating the image of $x_j$ to the origin in $\bbR^3.$ Let $\pi_{ij}^k$ denote the composition  \[\Conf_d(S^3)\rightarrow \Conf_2(\bbR^3)\xrightarrow{\tau_i^k} \bbR^3\times (\bbR^3\setminus \{0\})\]  
If $\omega\in H^2(S^2)\cong H^2(\bbR^3\setminus \{0\})$ is the fundamental class, we define $\omega_{ij}^k=(\pi_{ij}^k)^*(\omega)$ to be the pullback. Concretely, $\omega_{ij}^k$ takes that value 1 on the 2-sphere that separates $x_i$ from $x_k$ in $\bbR^3\setminus \{0\}$ with the outward orientation. In particular, we have that \begin{equation}\label{eqn:SwitchInfinity}\omega_{ij}^k=-\omega_{kj}^i\end{equation}
Moreover, the two trivialisations $\tau_i^k, \tau_j^k$ induce a diagram \begin{align*}
(\bbR^3\setminus \{0\})\times \bbR^3&\xleftarrow{\tau_{j}^k}\Conf_2(\bbR^3)\xrightarrow{\tau_i^k}\bbR^3\times (\bbR^3\setminus \{0\})\\
    (v_i-v_j,v_j)&\longmapsfrom (v_i,v_j) \longmapsto (v_i,v_j-v_i)
\end{align*}
Thus the induced map on $\bbR^3\setminus \{0\}$ is multiplication by $-1$, and we obtain \begin{equation}\label{eqn:SwitchBasepoint}
    \omega_{ji}^k=-\omega_{ij}^k
\end{equation}
Therefore any permutation of $\{i,j,k\}$ acts via the sign representation on $\omega_{ij}^k.$ 

\begin{lem}\label{lem:4PointRelation} The second cohomology group $H^2(\Conf_4(S^3))$ is 3-dimensional and generated by the four classes $\omega_{12}^3$, $\omega_{12}^4$, $\omega_{13}^4$ and $\omega_{23}^4$ with the relation 
\begin{equation}\label{eqn:4PointRelation}
    \omega_{12}^4+\omega_{23}^4=\omega_{13}^4+ \omega_{12}^3
\end{equation}
\end{lem}
\begin{proof}
    The trivialisation $\varphi_4$ identifies $\Conf_4(\bbR^3)$ with $S^3\times \Conf_3(\bbR^3)$, where $x_4$ corresponds to the point at infinity in $\bbR^3.$ Moving $x_3$ to the origin, we have a further identification:\[\Conf_4(S^3)\cong S^3\times \bbR^3\times \Conf_2(\bbR^3\setminus \{0\})\]
    We now use the action of the positive reals $\bbR^+$ on $\bbR^3\setminus \{0\}$ by scaling to move $x_2$ onto the unit sphere $S^2\subset \bbR^3\setminus \{0\}$. The resulting projection \[\Conf_2(\bbR^3\setminus\{0\})\rightarrow S^2\] is a fibre bundle with fibre $\bbR_+\times (\bbR^3\setminus \{\text{2 pts}\})$.  The associated cohomology spectral sequence degenerates on the $E_2$-page, hence $H^2(\Conf_2(\bbR^3\setminus \{0\}))$ is generated by the pullback of the fundamental class from $S^2$,  and $H^2(\bbR^3\setminus \{\text{2 pts}\})$. The pullback of the class on $S^2$ comes from forgetting $x_1$,  hence is just the class $\omega_{32}^4$. 
    
    The fibre $\bbR^3\setminus \{\text{2 pts}\}$ parametrises the position of $x_1$ in the complement of $x_2$ and $x_3$. Let $S_2$ be a 2-sphere containing $x_2$ in its interior, let $S_3$ be a 2-sphere with $x_3$ in its interior, and let $S_4$ be a 2-sphere with both $x_2$ and $x_3$ in its interior. The three homology classes $[S_2]$, $[S_3]$, and $[S_4]$ generate $H_2(\bbR^3\setminus \{\text{2 pts}\})$ but since $S_2\cup  S_3\cup S_4$ bounds 3-dimensional pair of pants, we have a relation \[[S_4]=[S_3]+[S_2]\]
    Exchanging the roles of $x_1$ and $x_2$, we can also consider spheres $T_2$, $T_3$, and $T_4$, by letting $x_2$ vary in a sphere containing $x_1$, $x_3$,  and both $x_1$ and $x_3$, respectively, in their interior. As above we have \[[T_4]=[T_1]+[T_3]\]
    and moreover $[T_1]=-[S_2]$.  Therefore two generate all of $H_2(\Conf_{2}(\bbR^3\setminus \{0\}))$ we will consider the three classes $[S_2]$, $[S_3]$ and $[T_3]$.   
    
    Recall that $\omega_{ij}^k$ takes the value 1 on a sphere formed by $x_j$ separating $x_i$ from $x_k$ with outward normal oriented toward $x_k$, and that by Equations \ref{eqn:SwitchInfinity} and \ref{eqn:SwitchBasepoint}, any odd permutation of the indices induces a sign change. Thus we have 
    \[\begin{pNiceMatrix}[first-row,first-col]
       &\omega_{21}^4 &\omega_{32}^4& \omega_{31}^4 & \omega_{21}^3  \\
       [S_2]& 1 & 0 & 0 & 1  \\
     [S_3]&0 & 0 & 1 & -1  \\
      [T_3]& 0 & 1 & 0 & 1  \\
\end{pNiceMatrix}\]
from which we conclude \[\omega_{21}^4+\omega_{32}^4=\omega_{31}^4+\omega_{21}^3.\]
Equation \ref{eqn:4PointRelation} now follows from Equation \ref{eqn:SwitchBasepoint}.
\end{proof}

We now describe the rational cohomology ring $H^*(\Conf_d(S^3))$ as a representation of $\Sym_d$.
\begin{prop}\label{prop:H*ConfS^3}
    The rational cohomology of the configuration space of $d$ points in $S^3$ is given by
    \[
        H^*(\Conf_d(S^3)) 
        \cong \Lambda_\bbQ\{\alpha\} \otimes \bbQ[\omega_{ij}^k : 1 \le i,j,k \le d]/I
    \]
    where $\alpha$ is of degree $3$,
    the $\omega_{ij}^k$ are of degree $2$,
    and $I$ is the ideal generated by the relations
    \begin{align*}
        \omega_{ij}^k &= -\omega_{ji}^k = \omega_{jk}^i, &
        (\omega_{ij}^k)^2 &= 0 &
        \omega_{ij}^k &= \omega_{ij}^l + \omega_{jk}^l + \omega_{ki}^l.
    \end{align*}
    The action of $\Sym_d$ is given by permuting the indices.
\end{prop}

\begin{rem}\label{rem:FI}
    In terms of $\FI$-representations, the statement of \cref{prop:H*ConfS^3} may be rephrased as follows.
    The $\FI$-ring $H^*(\Conf_\bullet(S^3))$ is generated by two classes
    \[
        \alpha \in H^3(\Conf_1(S^3)) 
        \qquad \text{and} \qquad 
        \omega_{12}^3 \in H^2(\Conf_3(S^3)) .
    \]
    Moreover, $H^*(\Conf_\bullet(S^3))$ is the free $\FI$-ring on these classes, subject to the following relations:
    \begin{enumerate}
        \item the images of $\alpha$ under the two maps $H^3(\Conf_1(S^3)) \to H^3(\Conf_2(S^3))$ agree,
        \item $\omega_{12}^3 \in H^2(\Conf_3(S^3))$ transforms as the sign representation of $\Sym_3$,
        \item $(\omega_{12}^3)^2 = 0$, and
        \item $\omega_{12}^3 = \omega_{12}^4 + \omega_{23}^4 + \omega_{31}^4$ holds in $H^2(\Conf_4(S^3))$.
    \end{enumerate}
\end{rem}

\begin{proof}[Proof of Proposition \ref{prop:H*ConfS^3}]
    We define a ring homomorphism 
    \[
        \Psi\colon \Lambda_\bbQ\{\alpha\} \otimes \bbQ[\omega_{ij}^k : 1 \le i,j,k \le n]
        \to
        H^*(\Conf_d(S^3)) 
    \]
    by sending $\alpha$ and $\omega_{ij}^k$ to the classes we defined before.
    Next, we need to check that the ideal $I$ lies in the kernel of $\Psi$.
    By \cref{eqn:SwitchBasepoint} and \cref{eqn:SwitchInfinity}, the identities $\omega_{ij}^k = - \omega_{ji}^k = \omega_{jk}^i$ hold in $H^*(\Conf_d(S^3))$.
    By \cref{lem:4PointRelation} we have that 
    \[
        \omega_{12}^4 + \omega_{23}^4 = \omega_{13}^4 + \omega_{12}^3 
        \qquad \text{ and hence } \qquad
        \omega_{12}^3 = \omega_{12}^4 + \omega_{23}^4 + \omega_{31}^4
    \]
    and it also follows that this relation holds if we replace $\{1,2,3,4\}$ by any other set of $4$ distinct indices.
    Finally, the relation $(\omega_{ij}^k)^2=0$ holds since $\omega_{ij}^k$ is pulled back from $\Conf_3(S^3) \simeq S^3 \times S^2$, where $H^4=0$.
    Therefore $\Psi$ descends to a map on the quotient:
    \[
        \overline{\Psi}\colon \Lambda_\bbQ\{\alpha\} \otimes \bbQ[\omega_{ij}^k : 1 \le i,j,k \le d]/I
        \to
        H^*(\Conf_d(S^3)) 
    \]
    
    In order to show that $\overline{\Psi}$ is an isomorphism, we will first have to algebraically rewrite the polynomial algebra $\bbQ[\omega_{ij}^k]/I$.
    Firstly, we can only consider generators $\omega_{ij}^k$ where $1 \leq i < j < k \le d$ by using the permutation relation.
    Next, we can use the relation $\omega_{ij}^k = \omega_{ij}^l + \dots$ 
    to reduce to those generators $\omega_{ij}^k$ with $k=d$.
    This shows
    \[
        \bbQ[\omega_{ij}^k : 1 \le i,j,k \le d]/I
        \cong
        \bbQ[\omega_{ij}^d : 1 \le i<j<d]/J
    \]
    where $J$ is the ideal generated by $(\omega_{ij}^k)^2$ for all $i \le j \le k$.
    In the case that $k \neq d$ this relation amounts to:
    \begin{align*}
        0 &= (\omega_{ij}^k)^2 
        = (\omega_{ij}^n + \omega_{jk}^d + \omega_{ki}^d)^2
        = \omega_{ij}^d\omega_{jk}^d + \omega_{ij}^d\omega_{ki}^d + \omega_{jk}^d\omega_{ki}^d\\
        &= \omega_{ij}^d\omega_{jk}^d - \omega_{ij}^d\omega_{ik}^d - \omega_{jk}^d\omega_{ik}^d
    \end{align*}
    
    By \cite[Proposition 3.1]{Feichtner-Ziegler},
    the rational cohomology algebra of $\Conf_{d-1}(\bbR^3)$ is
    \[
        H^*(\Conf_{d-1}(\bbR^3)) \cong
        \bbQ[\omega_{ij} : 1 \le i<j<d]/( (\omega_{ij})^2, \omega_{ij}\omega_{jk} - \omega_{ij}\omega_{ik} - \omega_{jk}\omega_{ik} )
    \]
    The trivialisation $\Conf_d(S^3) \cong S^3 \times \Conf_{d-1}(\bbR^3)$ is an isomorphism on cohomology and it sends $\omega_{ij}$ to $\omega_{ij}^d$.
    Therefore we can conclude that $\overline{\Psi}$ is an isomorphism.
\end{proof}

\subsection{The cohomology of the moduli space of punctured $3$-spheres}\label{sec:punctured-moduli}
We now compute the rational cohomology of the homotopy quotient of $\Conf_d(S^3)$ by the action of $\SO(4)$. Equivalently, if \smash{$\Diff_{\{1,\dots,d\}}^+(S^3)$} denotes the $\Diff^+(S^3)$-stabiliser of $d$ distinct points $x_1, \dots, x_d \in S^3$, then by Hatcher's theorem \cite{Hatcher}, we have:
\[
    \BDiff^+_{\{1, \dots, d\}}(S^3)
        \simeq \Conf_d(S^3) \hq \SO(4).
\]
This homotopy quotient arose naturally in the base of fibre sequence in \cref{cor:C-fibre-sequence} obtained by forgetting the framing at the foot of each handle attachment in $\rmC^\Gamma_\sigma(P_1,\dots,P_n) $. 

The action of $\SO(4)$ on $\Conf_\bullet(S^3)$ is compatible with the functoriality in $\FI^{\rm op}$, the homotopy quotient $\Conf_\bullet(S^3)\hq \SO(4)$ also defines a functor $\FI^{\rm op} \to \Top$.
This means that the cohomology ring 
\[
    H^*(\Conf_\bullet(S^4)\hq \SO(4)) \colon \FI \to \mathrm{Ring}
\]
is also an $\FI$-ring. Our goal will be to give generators and relations for this ring, just as in \cref{prop:H*ConfS^3}.
Note that the fibre sequence
\[
    \Conf_\bullet(S^3) \to  \Conf_\bullet(S^3) \hq \SO(4) \to  \BSO(4)
\]
induces maps of $\FI$-rings on cohomology, where $H^*(\BSO(4))$ is the constant $\FI$-ring $\bbQ[p_1,e]$.

Before we begin, it is useful to note that
in degree $2$ the restriction map
\[
    H^2(\Conf_\bt(S^3)\hq\SO(4)) \longrightarrow
    H^2(\Conf_\bt(S^3))
\]
is an isomorphism. 
This follows from the Serre spectral sequence since $\BSO(4)$ is simply connected and has no rational cohomology below degree $4$.

\begin{defn}
    We define the classes that will be the generators as follows:
    \begin{enumerate}
        \item The class $c_{12}^3 \in H^2(\Conf_3(S^3)\hq \SO(4))$ is the unique class that restricts to $2\cdot\omega_{12}^3$ under the map
        \[
            \Conf_3(S^3) \hookrightarrow \Conf_3(S^3) \hq \SO(4).
        \]
        (This is well-defined because we observed that this map induces an isomorphism in second rational cohomology.)
        \item The class $\delta \in H^4(\Conf_0(S^3)\hq \SO(4)) \cong H^4(\BSO(4))$ is the first Pontrjagin class.
    \end{enumerate}
\end{defn}

We identify $S^3 = \SU_2$ and $\SO(4) = (\SU_2 \times \SU_2)/(-1,-1)$.
The standard action of $\SO(4)$ on $S^3$ may then be written as the conjugation action:
\[
    (g,h).v := g\cdot v \cdot h^{-1}.
\]
Each factor of $\SU_2$ acts freely, transitively on $S^3$ and the stabiliser of any point is isomorphic to $\SO(3)$. Recall the trivialisation $\varphi_1\colon\Conf_d(S^3)\rightarrow S^3\times \Conf_{d-1}(\bbR^3)$ from Equation \ref{eqn:Conf Isomorphism}. Under this identification, the $\Conf_{d-1}(\bbR^3)$-coordinate corresponds to a collection of distinct vectors in the tangent space of the $S^3$-coordinate.  The action of $\SU_2\times \SU_2$ translates as follows
\[\xymatrix{
(x_1,\ldots,x_d)\ar[rr]^{\varphi_1}\ar[d]_{(g,h)\cdot}&&(x_1,x_1^{-1}x_2,\cdots,x_1^{-1}x_d)\ar[d]_{(g,h)\cdot}\\
(gx_1h^{-1},\ldots, gx_dh^{-1})\ar[rr]^{\varphi_1}&&(gx_1h^{-1}, hx_1^{-1}x_2h^{-1},\ldots, hx_1^{-1}x_dh^{-1})
}\]
In particular, the first factor of $\SU_2$ acts trivially on $\Conf_{d-1}(\bbR^3)$ while the second factor acts by the conjugation action on each tangent space.  
Since the first factor of $\SU_2$ clearly acts freely on $S^3$ and the conjugation action is equivalent to the defining action of $\SO(3)$ on $\bbR^3$, we can  first quotient by the action of the left-hand $\SU_2$ to get
\[
    \Conf_d(S^3)\hq \SO(4) 
    \simeq (\Conf_d(S^3)/\SU_2)\hq \SO(3)
    \simeq \Conf_{d-1}(\bbR^3)\hq \SO(3).
\]

\begin{lem}\label{lem:Conf-for-n=3}
    For $d=3$ let we have a homotopy equivalence $\Conf_3(S^3)\hq \SO(4) \simeq \BSO(2)$ such that $c_{12}^3$ goes to the Euler class in $H^2(\BSO(2))$.
    In particular, $c_{12}^3$ comes from an integral class.
\end{lem}
\begin{proof}
    Let $p = (0, e_1, \infty) \in \Conf_3(S^3)$ be the configuration of three points given by the origin, the first basis vector, and the point at $\infty$ in $S^3 = (\bbR^{3})^+$.
    Let $\SO(2) \subset \SO(4)$ be the subgroup of rotations around the first coordinate axis.
    Because $\SO(2)$ fixes $p$, we get a map 
    \[
        \BSO(2) = \{p\} \hq \SO(2) \longrightarrow \Conf_3(S^3) \hq \SO(4),
    \]
    which we will show is a homotopy equivalence.
    Consider the inclusion $S^2 \subset \Conf_2(\bbR^3) \subset \Conf_3(S^3)$ defined by $v \mapsto (0,v)$ and $(a,b) \mapsto (a,b,\infty)$, respectively.
    Then we have a commutative diagram
    \[\begin{tikzcd}
        & S^2 \rar["\simeq"] \dar["i"] & 
        {\Conf_2(\bbR^3)} \rar["{H^2(-;\bbZ)\text{-iso}}"] \dar & 
        {\Conf_3(S^3)} \dar["{H^2(-;\bbQ)\text{-iso}}"] \\
        {\{p\} \hq \SO(2)} \rar["\simeq"] &
        {S^2 \hq \SO(3)} \rar["\simeq"] &
        {\Conf_2(\bbR^3) \hq \SO(3)} \rar["\simeq"] &
        {\Conf_3(S^3) \hq \SO(4)}
    \end{tikzcd}\]
    where first bottom map is an equivalence by the homotopical orbit-stabiliser lemma (\cite[Lemma 2.10]{BoydBregmanSteinebrunner-finiteness}), the second one is the $\SO(3)$-equivariant equivalence $S^2 \simeq \Conf_2(\bbR^3)$, and the third map is an equivalence by the preceding discussion.
    
    The left-most vertical map $i$ is the fibre inclusion of a fibration $S^2\hq \SO(3)\rightarrow \BSO(3)$. From the long exact sequence on homotopy groups, we see that, since $\pi_2 \BSO(3) = \bbZ/2$, the map $i$ must induce multiplication by $2$ on $\pi_2$ and hence on $H_2$.
    Therefore, the generator of $H^2(S^2\hq \SO(3); \bbZ)$ (which is the Euler class in $\BSO(2)$) pulls back to twice a generator on $S^2$.
    The above diagram hence shows that under the isomorphism $H^2(\BSO(2); \bbZ) \cong H^2(\Conf_3(S^3)\hq \SO(4); \bbZ)$ the Euler class maps to some $e'$ whose pullback to $\Conf_3(S^3)$ is $2 \cdot \omega_{12}^3$.
    By the definition of $c_{ij}^k$ we must have $e'=c_{12}^3$.
\end{proof}

\begin{cor}\label{cor:cohomology presentation}
There is a well-defined map of graded rings
\begin{equation}\label{eqn:Cohomology presentation}
    \Phi\colon\bbQ[\delta, c_{ij}^k : 1 \le i,j,k \le d]/I_\delta
    \longrightarrow
    H^*(\Conf_d(S^3)\hq\SO(4))
 \end{equation}
where $I_\delta$ is the ideal generated by the relations
\begin{align*}
    c_{ij}^k &= -c_{ji}^k = c_{jk}^i, &
    (c_{ij}^k)^2 &= \delta &
    c_{ij}^k &= c_{ij}^l + c_{jk}^l + c_{ki}^l.
\end{align*}
\end{cor}
\begin{proof}
    We need to show that the classes $c_{ij}^k$ and $\delta$ that we define above indeed satisfy these relations.
    Since the restriction map
    \[
        H^*(\Conf_d(S^3)\hq\SO(4)) \longrightarrow H^*(\Conf_d(S^3))
    \]
    is an isomorphism for $*=2$, the linear relations between the $c_{ij}^k$ follow from the relations we know for $\omega_{ij}^k$.
    It remains to check that $(c_{12}^3)^2 = \delta$ holds in $H^*(\Conf_3(S^3) \hq \SO(4))$.
    By \cref{lem:Conf-for-n=3} we can verify this relation in $H^*(\BSO(2))$, which is a polynomial ring $\bbQ[e]$ generated by the Euler class $e$ corresponding to $c_{12}^3$ under the isomorphism.
    The class $\delta$ is the restriction of the first Pontryagin class along the stabilisation map $\BSO(2) \to \BSO(4)$, which gives $p_1 = e^2 = (c_{12}^3)^2$.
    The case for general $i,j,k$ follows by restriction along maps that forget and relabel points.
\end{proof}

Let $R_\delta$ denote the ring $\bbQ[\delta, c_{ij}^k : 1 \le i,j,k \le d]/I_\delta$ appearing in \cref{eqn:Cohomology presentation}. In order to establish that $\Phi$ is an isomorphism, we will need the following property of $R_\delta$:

\begin{lem}\label{lem:delta-cancellable}
    The element $\delta \in R_\delta$ is cancellable, \emph{i.e.}~it is not a zero-divisor.
\end{lem}

\begin{proof}
    Eliminating all linear relations  and the relation $\delta = (c_{d-2,d-1}^d)^2$ in the presentation of $R_\delta$, we observe that $R$ is isomorphic to $\bbQ[c_{ij}^d\colon 1\leq i< j\leq d-1]/J$, where $J$ is the ideal generated by the set $G$ consisting of 
    \[\begin{array}{cc}(c_{ij}^d)^2-(c_{d-2,d-1}^d)^2,&c_{ij}^dc_{jk}^d-c_{ij}^dc_{ik}^d-c_{jk}^dc_{ik}^d-(c_{d-2,d-1}^d)^2
    \end{array}\]
    Order the generators so that $c_{ij}^d>c_{pq}^d$ if and only if $i<p$, or $i=p$ and $j<q$. With respect to the resulting monomial order, $G$ is a Gr\"obner basis. The verification of this fact is routine and left to the reader (see \cite{CLO}). 
    To simplify notation we denote the minimal generator by $z=c_{d-2,d-1}^d$. Note that $z^2$ represents $\delta$.
    
    To prove that $\delta = [z^2]$ is not a zero-divisor in $\bbQ[c_{ij}^d]/J$ we need to argue that $p \cdot z^2 \in J$ implies $p \in J$ for any polynomial $p \in \bbQ[c_{ij}^d]$.
    To show this, note that no multiple of $z$ is the leading term of an element of $G$.
    Therefore the leading term of $pz^2$ is divisible by the leading term of an element of $G$ if and only if the leading term of $p$ is.
    Moreover, when performing a $G$-reduction for such a leading term the result will still be of the form $p'z^2$.
    Hence $pz^2$ can be $G$-reduced to 0 if and only if $p$ can be $G$-reduced to 0. 
    Since $G$ is a Gr\"obner basis, this shows $p\in J \Leftrightarrow pz^2 \in J$, as claimed.
\end{proof}

We now show that $\Phi$ is in fact an isomorphism.
This means that the $\FI$-ring $H^*(\Conf_\bullet(S^3)\hq\SO(4))$ is generated by $c_{12}^3$ and $\delta$ subject to the relations listed above.

\begin{prop}\label{prop:Conf//SO4}
The map $\Phi$ in \cref{eqn:Cohomology presentation} is an isomorphism of graded rings.
Therefore the rational cohomology of $\BDiff^+_{\{1,\dots,d\}}(S^3) \simeq \Conf_d(S^3)\hq\SO(4)$ is 
\[
    H^*(\Conf_d(S^3)\hq\SO(4)) 
    \cong \bbQ[\delta, c_{ij}^k : 1 \le i,j,k \le d]/I_\delta
\]
and the action of $\Sym_d$ is given by permuting the indices.
\end{prop}
\begin{proof}
    We have already constructed the ($\Sym_d$-equivariant) morphism of graded rings from $R_\delta$ to the cohomology of $\Conf_d(S^3)\hq \SO(4)$, so we only need to show that it is an isomorphism of vector spaces.
    Consider the Serre spectral sequence associated to the fibration
    \[
        \Conf_{d-1}(\bbR^3) \to \Conf_{d-1}(\bbR^3)\hq\SO(3) \to \BSO(3).
    \]
    It has signature
    \[
        E_2^{p,q} \cong
        H^q(\Conf_{n-1}(\bbR^3)) \otimes H^p(\BSO(3))
        \Rightarrow 
        H^{p+q}(\Conf_{d-1}(\bbR^3) \hq \SO(3)).
    \]
    The $E_2$-page is even and therefore the spectral sequence collapses: $E_2^{p,q} = E_\infty^{p,q}$.
    The $E_\infty$-page is the associated graded of a decreasing filtration
    \[
        \dots \subseteq F_a^* \subseteq \dots \subseteq F_0^* =
        H^{*}(\Conf_{d-1}(\bbR^3) \hq \SO(3))
    \]
    with $F_a^* = 0$ for $a>*$ and $F_a^a \subseteq H^a(\Conf_{d-1}(\bbR^3)\hq \SO(3))$ being the image of the cohomology of the base $H^a(\BSO(3))$.
    This filtration is compatible with the multiplicative structure in the sense that $F_a^* \cdot F_b^* \subseteq F_{a+b}^{*}$
    and its associated graded is isomorphic, as a bi-graded ring, to:
    \[ 
        \bigoplus_{p} F_p^*/F_{p+1}^* \cong
        \bigoplus_{p} H^*(\Conf_{d-1}(\bbR^3)) \otimes H^p(B\SO(3))
    \]
    The class $\delta \in H^4(\Conf_{d-1}(\bbR^3)\hq\SO(3))$ comes from a cohomology class in the base $H^4(B\SO(3))$ and therefore it is in filtration degree $4$.
    
    We equip the ring $R_\delta$ with the ``$\delta$-adic filtration'' where an element $x \in R_\delta$ is in $F_{4m}^*$ if and only if we can write $x = y \delta^m$.
    Then $\Phi$ is compatible with the filtration and the associated graded of the $\delta$-adic filtration on $R_\delta$ is isomorphic to $(R_\delta/\delta R_\delta) \otimes \bbQ[\delta]$ because $\delta$ is cancellable by \cref{lem:delta-cancellable}.
    The map induced by $\Phi$ on the associated graded is 
    \[
        \mathrm{gr}(\Phi)\colon (R_\delta/\delta R_\delta) \otimes \bbQ[\delta]
        \to H^*(\Conf_{d-1}(\bbR^3)) \otimes H^*(\BSO(3)).
    \]
    This is an isomorphism since $H^*(\BSO(3)) \cong \bbQ[\delta]$ and $R_\delta/\delta R_\delta = \bbQ[c_{ij}^k]/I \cong H^*(\Conf_{d-1}(\bbR^3))$ by \cref{prop:H*ConfS^3}.
    It follows that the original ring homomorphism $\Phi$ was an isomorphism as claimed.
\end{proof}

\begin{cor}\label{cor:free-delta-module}
    The ring $H^*(\Conf_d(S^3)\hq \SO(4))$ is a free graded $\bbQ[\delta]$-module.
\end{cor}
\begin{proof}
    We know that $H^* = H^*(\Conf_d(S^3) \hq \SO(4))$ is a (connective) graded $\bbQ[\delta]$-module on which multiplication by $\delta$ is injective.
    Therefore, we can construct a basis inductively, by picking in each degree $k$ a basis of the complement of $\delta H^{k-4} \subset H^k$.
\end{proof}

\subsection{The cohomology of $\rmC^\Gamma_\sigma$}\label{subsection:cohomology of D}

We now describe the rational cohomology ring of $\rmC^\Gamma_\sigma(P_1,\dots,P_n)$ functorially in automorphisms in $(\Gamma,\sigma) \in \Gr_{g,n}$.
We are mainly interested in the example of $M = U_g := (S^1 \times S^2)^{\sharp g}$, where $n=0$ and $\rmC^\Gamma = \rmD(\Gamma)$.
In this case we get a complete description in terms of generators and relations.
When $n>0$ we still obtain a description in terms of the cohomology of the configuration spaces of the $P_i$.
The functoriality in edge contractions will be discussed in the next subsection (see \cref{lem:cohomology-functoriality}).

In the cohomology of $\rmC^\Gamma_\sigma$ we will have a class $c_{ij}^k$ for every triple of distinct half-edges $(i,j,k) \in H_\Gamma$ that are incident at the same vertex.
It will hence be useful to introduce the notation
\[
    H_\Gamma^\tripod := \{ (i,j,k) \in H_\Gamma^3 \;|\; r(i) = r(j) = r(k) \} 
    = \coprod_{v \in V_\Gamma} (H_v)^3.
\]
We will also write $H_1(\Gamma, \sigma)$ for the first cellular homology of $\Gamma$ relative to the subspace $\{\sigma(1),\dots,\sigma(n)\}$, which is an $\Aut(\Gamma,\sigma)$-representation.

\begin{thm}\label{thm:DGamma Cohomology Ring}
    For $(\Gamma,\sigma) \in \Gr_{g,n}$ there is an $\Aut(\Gamma,\sigma)$-equivariant isomorphism of graded rings
    \begin{align*}
        H^*(\rmC_\sigma^\Gamma(P_1,\dots,P_n)) 
        & \cong 
        \Lambda_{\bbQ} \langle H_1(\Gamma, \sigma)[3] \rangle \otimes 
        \bbQ[\delta, c_{ij}^k : (i,j,k) \in H_\Gamma^{\tripod}]/I^\Gamma
        \otimes \bigotimes_{i=1}^n H^*(\Conf_{H_{\sigma(i)}}(P_i)) 
    \end{align*}
    where $H_1(\Gamma, \sigma)[3]$ denotes the first homology of the graph $\Gamma$ relative to the subset of marked vertices, shifted to have degree 3.
    The ideal $I^\Gamma$ is generated by 
\begin{align*}
    I^\Gamma &= 
    \begin{cases}
    \langle c_{ij}^k = -c_{ji}^k = c_{jk}^i, 
    c_{ij}^k = c_{ij}^l + c_{jk}^l + c_{ki}^l, 
    (c_{ij}^k)^2 = \delta\rangle & \text{ for } n=0\\
    \langle c_{ij}^k = -c_{ji}^k = c_{jk}^i, 
    c_{ij}^k = c_{ij}^l + c_{jk}^l + c_{ki}^l, 
    (c_{ij}^k)^2 = \delta = 0\rangle & \text{ for } n>0
    \end{cases}
\end{align*}
    where the indices $i,j,k$ and $l$ run over distinct half-edges incident at the same vertex.
\end{thm}

\begin{rem}
    Conceptually, the reason that we have $\delta = 0$ for $n>0$ is that $\delta$ represents the $0$th homology group of $\Gamma$ relative to the set of marked points, which is trivial once there is at least one marked point.
\end{rem}

\begin{ex}
    For $n=0$ and $\Gamma \in \Gr_{g,0}$ we have $\Aut(\Gamma)$-equivariant isomorphisms
    \begin{align*}
            H^*(\calS(\Gamma)) 
            \cong H^*(\rmC^\Gamma) 
            & \cong 
            \Lambda_{\bbQ} \langle H_1(\Gamma) [3]\rangle \otimes 
            \frac{\bbQ[\delta, c_{ij}^k : (i,j,k) \in H_\Gamma^{\tripod}]}{
                \langle c_{ij}^k = -c_{ji}^k = c_{jk}^i, 
                c_{ij}^k = c_{ij}^l + c_{jk}^l + c_{ki}^l, 
                (c_{ij}^k)^2 = \delta\rangle
            },
    \end{align*}
    which in low degrees gives
    \begin{align*}
        H^1(\rmC^\Gamma) &\cong 0, &
        H^2(\rmC^\Gamma) &\cong \frac{\bbQ\langle H_\Gamma^\tripod \rangle}{ 
        \langle c_{ij}^k = -c_{ji}^k = c_{jk}^i,  c_{ij}^k = c_{ij}^l + c_{jk}^l + c_{ki}^l \rangle} , & 
        H^3(\rmC^\Gamma) &\cong H_1(\Gamma).
    \end{align*}
\end{ex}

We first establish a general lemma about spectral sequences that we will use in the proof of \cref{thm:DGamma Cohomology Ring}.
\begin{lem}\label{lem:spectral-sequence-lemma}
    Suppose $(E_*,d_*)$ is a multiplicative cohomological spectral sequence with 
    \[
       E_4^{*,*}\cong \Lambda_\bbQ(E_4^{0,3}) \otimes E_4^{*,0} 
    \]
    and assume that $E_4^{*,0}$ is free as a module over $\bbQ[V]$ where $V = \mathrm{Im}(d_4\colon E_4^{0,3} \to E_4^{4,0})$.
    Then the $E_\infty$-page is 
    \[
        E_\infty^{*,*} \cong \Lambda_\bbQ(\ker(d_4\colon E_4^{0,3} \to E_4^{4,0})) \otimes R
    \]
    where $R$ is the quotient of the ring $E_4^{0,*}$ by the ideal generated by $V \subset E_4^{0,4}$.
    Moreover, there are no multiplicative extensions.
\end{lem}
\begin{proof}
    By assumption we can find an isomorphism of graded $\bbQ[V]$-modules $E_4^{*,0} \cong \bbQ[V] \otimes B$ for some graded vector space $B$.
    Let $R$ be the quotient of the ring $E_4^{*,0}$ by the ideal generated by $V \subset E_4^{4,0}$.
    Then the composite $B \hookrightarrow E_4^{*,0} \twoheadrightarrow R$ is an isomorphism, and we can identify $B \cong R$.
    We can non-canonically decompose $E_4^{0,3} \cong K \oplus V$ where $K = \ker(d_4\colon E_4^{0,3} \to E_4^{4,0})$.
    Then the $E_4$ page is 
    \[
        E_4^{*,*} \cong \Lambda_\bbQ(K^{(0,3)}) \otimes \Lambda_\bbQ(V^{(0,3)}) \otimes \bbQ[V^{(4,0)}] \otimes R^{(*,0)}
    \]
    where the superscripts indicate bidegree.
    (Note that this is not an isomorphism of rings, but it respects the module structure over the first three terms.)
    The $d_4$ differential induces the identity isomorphism $d_{4} = \id \colon V^{(0,3)} \to V^{(4,0)}$ and is $0$ on $K^{(0,3)}$ and $R^{(*,0)}$.
    Thus, if we think of $E_4^{*,*}$ as a chain complex (only recording total degree) it decomposes into a tensor product
    \[
        (E_4^{*,*}, d_4) \cong \mathrm{grSym}(V[3] \xrightarrow{d_4 = \id} V[4]) \otimes \left( \Lambda_\bbQ(K[3]) \otimes R, d = 0\right)
    \]
    where $\mathrm{grSym}(-)$ takes the graded symmetric powers.
    Taking homology, and using that over $\bbQ$ the functor $\mathrm{grSym}(-)$ commutes with taking homology, we see that
    \[
        E_5^{*,*} \cong \Lambda_\bbQ(K^{(3,0)}) \otimes R^{(*,0)}.
    \]
    The ring structure on this is the canonical one, coming from thinking of $R^{(*,0)}$ as a quotient of $E_4^{(*,0)}$ -- namely the quotient by the ideal generated by $V$.
    The $d_5$-differential is trivial on $K^{(3,0)}$ and on $R^{(*,0)}$ for degree reasons, so the multiplicative structure implies that $d_5 \equiv 0$ and similarly for all higher differentials. 
    This shows that the $E_\infty$-page has the desired structure.
    Finally, we note that there cannot be any multiplicative extensions because $E_\infty^{0,*}$ is a free graded symmetric algebra.
\end{proof}

\begin{proof}[Proof of \cref{thm:DGamma Cohomology Ring}]
We will use the spectral sequence described in \cref{prop:C-spectral-sequence}.
Using \cref{prop:Conf//SO4} the $E_2$-page of this spectral sequence is 
\[
    \Lambda_\bbQ[\beta_h : h \in H_\Gamma]/\langle \beta_{h^\dagger} = -\beta_h\rangle 
    \otimes \bigotimes_{v\in V_\Gamma^{\rm sph}} \bbQ[\delta_v, c_{ij}^k : (i,j,k) \in H_v]/I_\delta
    \otimes \bigotimes_{i=1}^n H^*(\Conf_{H_{\sigma(i)}}(P_i)).
\]
From \cref{prop:C-spectral-sequence} we also know that the first non-zero differentials are the $d_4$ differentials, and that on the generators $\beta_h$ these are given by
\[
    d_4(\beta_h) = \delta_{r(h)} - \delta_{r(h^\dagger)}.
\]
The notation here indeed matches that of \cref{prop:C-spectral-sequence}: we interpret $\delta_v = 0$ whenever $v$ is a marked vertex.
The first interesting differential is then 
\[
    d_4\colon E_4^{0,3} \cong \bbQ\langle \beta_h : h \in H_\Gamma \rangle / \langle \beta_{h^\dagger} + \beta_h\rangle
    \longrightarrow
    E_4^{4,0} \cong \bbQ\langle \delta_v : v \in V_\Gamma \rangle \oplus (\text{other terms}).
\]
In fact the formula \cref{prop:C-spectral-sequence} gives for $d_4(\beta_h)$ tells us that this map lands in the $\bbQ\langle \delta_v: v \in V_\Gamma\rangle$ summand and gives us a copy of the relative cellular chain complex
\[
    \partial\colon C_1(\Gamma, \sigma) = \frac{\bbQ\langle H_\Gamma \rangle}{[h] = -[h^\dagger]} 
    \longrightarrow C_0(\Gamma, \sigma) = \frac{\bbQ\langle V_\Gamma \rangle}{\bbQ\langle \sigma(1),\dots, \sigma(n)\rangle}
\]
of $\Gamma$ within the spectral sequence.
We thus see the relative homology of $\Gamma$ on the $E_5$-page as
\[
    E_5^{0,3} \cong H_1(\Gamma, \sigma)
    \qquad\text{ and }\qquad
    E_5^{0,3} \cong H_0(\Gamma, \sigma).
\]
To determine the remainder of the spectral sequence, we use \cref{lem:spectral-sequence-lemma}.
We already know that $E_4^{0,*}$ is an exterior algebra on $C_1(\Gamma, \sigma)[3]$.
To describe the row $E_4^{*,0}$, recall from \cref{cor:free-delta-module} that the tensor factor corresponding to a spherical vertex $v \in V_\Gamma^{\rm sph}$ is free over $\bbQ[\delta_v]$, and thus $E_4^{0,*}$ is free as a module over $\bbQ[\delta_v: v \in V_\Gamma^{\rm sph}]$. 
This also implies that $E_4^{0,*}$ is also free as a module over $\mathrm{grSym}_\bbQ( d_4(E_4^{0,3}))$, as $d_4(E_4^{0,3})$ is a subvector space (of codimension $0$ or $1$) of $\bbQ\langle \delta_v : v \in V_\Gamma^{\rm sph}\rangle$.
\cref{lem:spectral-sequence-lemma} now gives us a formula for the $E_5$-page and tells us that the spectral sequence collapses and has no multiplicative extensions, so that there are isomorphisms of rings with $\Aut(\Gamma,\sigma)$-action
\[
    H^*(\rmC^\Gamma_\sigma(P_1,\dots, P_n)) \cong E_\infty^{*,*} \cong E_5^{*,*} \cong \Lambda_\bbQ\langle H_1(\Gamma, \sigma)[3]\rangle \otimes R
\]
where $R$ is the ring obtained as the quotient of $E_4^{*,0}$ by the ideal generated by the image $d_4(E_4^{0,3}) \subset E_4^{4,0}$.
Inserting our previous description of the cohomology of the base we get
\[
    R = 
    \frac{
    \bigotimes_{v\in V_\Gamma^{\rm sph}} \bbQ[\delta_v, c_{ij}^k : (i,j,k) \in H_v]
    }{d_4(E_4^{0,3}) + \sum_{v \in V_\Gamma^{\rm sph}} I_{\delta_v} }
    \otimes \bigotimes_{i=1}^n H^*(\Conf_{H_{\sigma(i)}}(P_i)).
\]
We need to check that the ideal we quotient out in the first term is exactly $I_\Gamma$.
When $n>0$ the cellular boundary map $\partial\colon C_1(\Gamma,\sigma) \to C_0(\Gamma, \sigma)$ is surjective, so that $d_4(E_4^{0,3})$ contains all $\delta_v$.
If instead $n=0$, then the image of $\partial$ has codimension $1$ so that taking the quotient by $d_4(C_1(\Gamma,\sigma))$ exactly identifies all the $\delta_v$.
\end{proof}

\begin{rem}\label{rem:beta-ses-splitting}
    In the proof of \cref{thm:DGamma Cohomology Ring}, the isomorphism $H^*(\rmC^\Gamma_\sigma(P_1,\dots,P_n)) \cong E_\infty^{*,*}$ is not necessarily canonical.
    Nevertheless, it is still compatible with the multiplicative structure (as $E_\infty^{0,*}$ is a free graded symmetric algebra) and the $\Aut(\Gamma,\sigma)$-action (as the category of $\bbQ$-representations over a finite group is semi-simple).
    This will be sufficient for our applications.
    The subtlety lies in splitting the short exact sequence
    \[
        \bigoplus_{i=1}^n H^3(\Conf_{H_{\sigma(i)}}(P_i))
        \longrightarrow
        H^3(\rmC^\Gamma_\sigma(P_1,\dots,P_n)) 
        \longrightarrow
        H_1(\Gamma,\sigma) 
    \]
    that relates $\bigoplus_{a+b=3} E_3^{a,b}$ to $H^3(\rmC^\Gamma_\sigma)$.
    There always is a splitting as $\Aut(\Gamma,\sigma)$-representations, but this is not necessarily canonical.
    (Note that for $n=0$ the first group is trivial, resolving this issue.)
    However, we believe that one could also obtain this splitting canonically, by constructing for each relative cycle in $(\Gamma,\sigma)$ a map $\rmC^\Gamma_\sigma(P_1,\dots,P_n) \to \SO(3)$ that records how the framing changes along the path.
\end{rem}

\subsection{Functoriality of \texorpdfstring{$\beta_h$ and $c_{ij}^k$}{beta and c-ijk}}\label{sec:graph-morphism-functoriality}
In \cref{thm:DGamma Cohomology Ring} we describe the cohomology ring of $\rmC^\Gamma_\sigma(P_1,\dots,P_n)$ functorially in graph \emph{isomorphisms}.
For our computations it will be necessary to also understand what happens to the generating classes under more general graph maps, such as edge-contractions.
For our applications, we will only need the case when $n=0$, but we will also give a partial description for general $n$.

Let $f\colon \Gamma \to \Lambda$ be a morphism in $\Gr_{g,n}$.
Then for a vertex $v \in \Lambda$, the preimage $f^{-1}(v)$ is necessarily a tree. 
Therefore, given three half-edges $i,j,k \in H_v$ incident at $v$ we can find a unique tripod connecting $f^{-1}(i), f^{-1}(j), f^{-1}(k)$ in $f^{-1}(v)$.
This tripod contains a unique triple of half-edges $i', j', k' \in H_\Gamma$ all incident at the same vertex $u \in \Gamma$, which necessarily satisfies $f(u) = v$.
This defines a map
\[
    f^\tripod\colon H_\Lambda^\tripod \to H_\Gamma^\tripod
\]
and in fact makes $H_{(-)}^\tripod\colon \Gr_{g,n}^{\rm op} \to \mathrm{Set}$ a functor.
(For example, in \cref{fig:c_ijk_functorial} the triple $(i,j,k)$ is mapped to $(h,j',k')$.)

\begin{figure}[ht]
\centering
\def\svgwidth{\linewidth}
\import{}{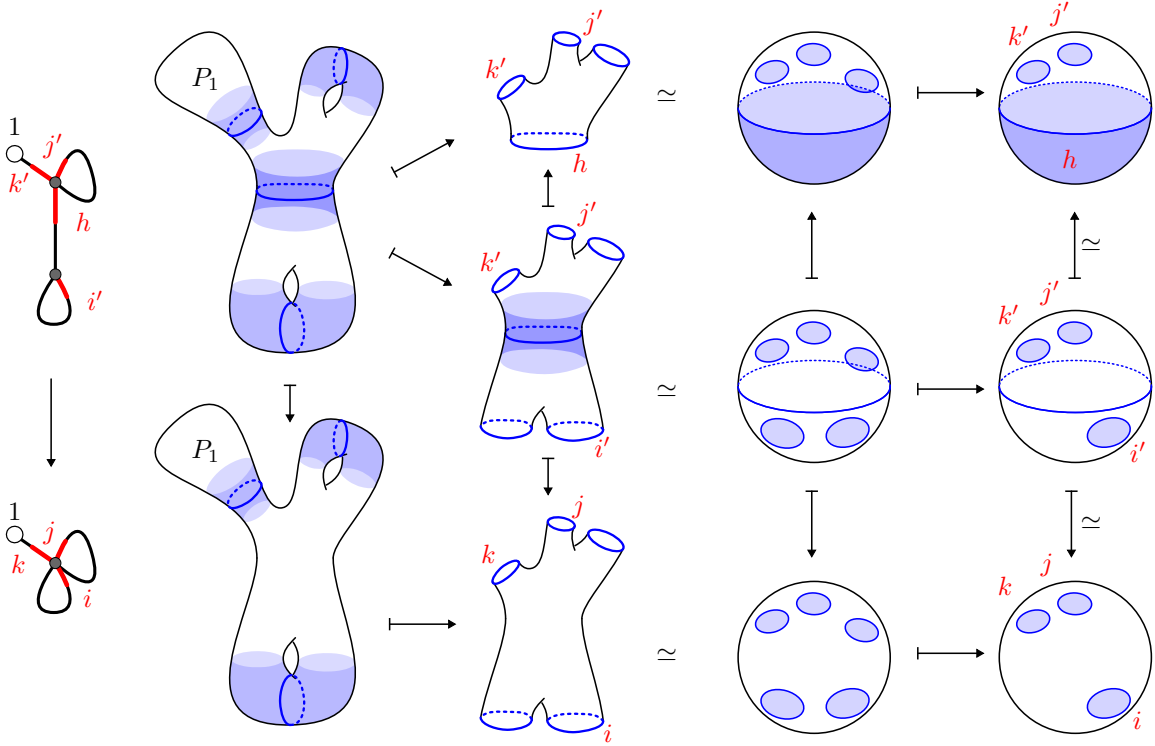}
\caption{An illustration (in $2$D rather than $3$D) of the homotopy commutative diagram in the proof of \cref{lem:cohomology-functoriality}.}
\label{fig:c_ijk_functorial}
\end{figure}

\begin{lem}\label{lem:cohomology-functoriality}
    The map induced by $\rmC^\Gamma_\sigma(P_1,\dots,P_n) \to \rmC^\Lambda_\rho(P_1,\dots, P_n)$ on cohomology maps the $c_{i,j}^k$ classes as
    \[
        c_{(i,j,k)} \mapsto c_{f^\tripod(i,j,k)}
    \]
    for all $(i,j,k) \in H_\Lambda^\tripod$.
\end{lem}
\begin{proof}
    We begin by describing the action on the $c_{ij}^k$.
    Consider a graph morphism $f\colon \Gamma \to \Lambda$, $w \in V_\Lambda^{\rm sph}$ an unmarked vertex, and $i,j,k \in H_w$ adjacent half-edges.
    Without loss of generality we may assume that $f^{-1}(w)$ consists of a single edge $e = \{h, h^\dagger\}$ connecting two unmarked vertices $v_1,v_2 \in V_\Gamma^{\rm sph}$.
    (If $f^{-1}(w)$ is a single vertex, then the claim follows essentially from the definition, and if $f^{-1}(w)$ contains multiple edges, we can write $f$ as a composite of morphisms that each only collapse one edge, reducing to the case at hand.)
    We write $H_{v_1} = A \sqcup \{h\}$ and $H_{v_2} = B \sqcup \{h^\dagger\}$ where $A$ and $B$ are finite sets with each at least two elements -- then $H_w = f(A) \sqcup f(B)$.
    Up to symmetry there are two cases we need consider: either $i \in f(B)$ and $j,k \in f(A)$, or all half edges are on one side. 
    We will consider the former case, and leave the easier case to the reader.

    We would like to show that $\Cthick(f)^*(c_{ij}^k) = c_{hj'}^{k'}$ where $j',k' \in A$ are the (unique) preimages of $j$ and $k$ under $f$.
    To show this, we first introduce some notation.
    We let $S^2 \subset S^3$ be the equator and $D^3_+, D^3_- \subset S^3$ the upper and lower hemisphere.
    We pick, $D_A \subset D^3_+$ and $D_B \subset D^3_-$ collections of discs, indexed by $A$ and $B$, and we write $D_i \subset D_A$ and $D_j, D_k \subset D_B$ for the three discs of particular interest.
    For $N \subset S^3$ a compact submanifold (of not necessarily constant dimension) we let $\Diff_{\pi_0}(S^3, N) \subset \Diff(S^3, N)$ denote the subgroup of the diffeomorphisms $\psi$ with $\psi(N)=N$ that induce the identity permutation on $\pi_0(N)$.
    
    With this notation in hand, there is a homotopy commutative diagram as illustrated in \cref{fig:c_ijk_functorial}:
\[\begin{tikzcd}
	{\Cthick^\Gamma_\sigma} && 
    {\BDiff_{\pi_0}(S^3, D_A \sqcup D^3_-)} && {\BDiff_{\pi_0}(S^3, D^3_- \sqcup D_j \sqcup D_k)} \\
	&& {\BDiff_{\pi_0}(S^3, S^2 \sqcup D_A \sqcup D_B)} && {\BDiff_{\pi_0}(S^3, S^2 \sqcup D_i \sqcup D_j \sqcup D_k)} \\
	{\Cthick^\Lambda_\rho} && {\BDiff_{\pi_0}(S^3,D_A \sqcup D_B)} && {\BDiff_{\pi_0}(S^3, D_i\sqcup D_j \sqcup D_k)}
	\arrow["{\varphi_{v_1}}", from=1-1, to=1-3]
	\arrow["\chi"', bend right = 1pc, from=1-1, to=2-3]
	\arrow["{\Cthick(f)}"', from=1-1, to=3-1]
	\arrow[from=1-3, to=1-5]
	\arrow[from=2-3, to=1-3]
	\arrow[from=2-3, to=2-5]
	\arrow[from=2-3, to=3-3]
	\arrow["\simeq", from=2-5, to=1-5]
	\arrow["\simeq"', from=2-5, to=3-5]
	\arrow["{\varphi_w}", from=3-1, to=3-3]
	\arrow[from=3-3, to=3-5]
\end{tikzcd}\]
    The maps $\varphi_{v_i}$ and $\varphi_w$ were described at the start of \cref{subsec:C-pullback}.
    The version used here implicitly uses the inverse of the homotopy equivalence 
    \[ 
        \BDiff_{\pi_0C}(S^3, C \times D^3) \xrightarrow{\;\simeq\;} \BDiff_{\pi_0C}( S^3 \setminus (C \times \interior{D^3})) \cong \frac{\BDiff_{C}^{\rm sph}(S^3)}{\Map(C, \Diff^+(S^2))}
    \]
    that cuts out the interiors of the discs $C \times D^3$.
    The map $\chi$ in the above diagram is constructed by taking $(\Sigma^\varepsilon \subset M, \dots) \in \Cthick^\Gamma_\sigma$, cutting $M$ along all components of $\Sigma^\varepsilon$ except the one corresponding to $\{h,h^\dagger\}$, and then remembering the component corresponding to $v_1$ and $v_2$ together with the $2$-sphere given by $\Sigma_{\{h,h^\dagger\}}$.
    (Again, this map really lands in $\BDiff_{\pi_0\partial}(S^3 \setminus (\interior{D_A}\sqcup \interior{D_B}), S^2)$ and we use the homotopy inverse to the map that cuts out the interiors.)
    The left half of the diagram commutes up to homotopy by inspecting definitions. 
    (Crucially, the bottom square commutes since $\Cthick(f)$ is exactly given by forgetting the component of $\Sigma^\varepsilon$ that corresponds to the edge $\{h,h^\dagger\}$.)
    The right half of the diagram is constructed by maps that just forget part of the constraints we put on the diffeomorphisms (or project to $*$), \emph{i.e.}~they come from subgroup inclusions (or maps to the trivial group), and thus it automatically commutes.
    The two vertical maps on the right side are homotopy equivalences as a consequence of the Smale conjecture, since exactly one of the discs is on one side of the sphere -- this can be seen in the usual manner by repeatedly applying  fibre sequences.

    Given the homotopy commutative diagram we argue as follows.
    The pullback of the preferred generator of $H^3(\BDiff_{\pi_0}(S^3, D_i \sqcup D_j \sqcup D_k))$ along the bottom two maps is by definition $c_{i,j}^k$.
    On the other hand, the pullback of the preferred generator along the top row is by definition $c_{h,j'}^{k'}$ where $f(j') = j$ and $f(k')=k$.
\end{proof}

We can also describe the functoriality of the classes $\beta_a$ for $a \in H_1(\Gamma, \sigma)$.
However, as pointed out in \cref{rem:beta-ses-splitting} for $n>0$ these classes are only well-defined up to splitting a short exact sequence, and hence we will have the same ambiguity in describing their functoriality.
\begin{lem}\label{lem:cohomology-functoriality-beta}
    For every graph map $f\colon (\Gamma,\sigma) \to (\Lambda,\rho)$ in $\Gr_{g,n}$ we have a commutative square
    \[\begin{tikzcd}
        {H^3( \rmC^\Gamma_\sigma(P_1,\dots,P_n) )} \rar &
        H_1(\Gamma,\sigma) \dar["f_*", "\cong"'] \\
        {H^3( \rmC^\Lambda_\rho(P_1,\dots,P_n) )} \uar["\rmC(f)^*"] \rar &
        H_1(\Lambda,\rho)
    \end{tikzcd}\]
    where the horizontal maps are those discussed in \cref{rem:beta-ses-splitting} and they are isomorphisms for $n=0$.
\end{lem}
\begin{proof}
    Fix some $(\Sigma^\varepsilon \subset M \subset \bbR^\infty, \alpha, \iota) \in \Cthick^\Gamma_\sigma(P_1,\dots,P_n)$.
    We begin by giving a concrete description of the (homotopy) fibre of the top horizontal map in \cref{prop:rmC-pullback}, 
    which deletes $\Sigma^\varepsilon \subset M$ and records the components of the complement.
    Let $X(\Sigma^\varepsilon \subset M)$ be the space of those 
    embeddings $j\colon M \hookrightarrow \bbR^\infty$
    such that $j_{|M \setminus \Sigma^\varepsilon}$ is the inclusion $M\setminus \Sigma^\varepsilon \subset \bbR^\infty$.
    This space is equivalent to $\BDiff_{\partial}(\Sigma^\varepsilon)$ via the map that sends $j$ to $j(\Sigma^\varepsilon)$.
    (Note that the boundary $\partial j(\Sigma^\varepsilon)$ must be equal to $\partial \Sigma^\varepsilon \subset M \subset \bbR^\infty$.)

    For a graph morphism $f\colon \Gamma \to \Lambda$,
    let $(\Upsilon^\varepsilon \subset M)$ denote the thickened sphere system obtained by forgetting the spheres that correspond to edges collapsed by $f$,
    and let $\Xi^\varepsilon = \Sigma^\varepsilon \setminus \Upsilon^\varepsilon$ be the union of the deleted thickened spheres.
    We have a commutative diagram:
    \[\begin{tikzcd}
        {\BDiff_\partial(\Sigma^\varepsilon)} &
        {X(\Sigma^\varepsilon \subset M)} \lar["\simeq"'] \rar &
        {\Cthick^\Gamma_\sigma(P_1,\dots,P_n)} \dar["{\Cthick(f)}"] \\
        {\BDiff_\partial(\Upsilon^\varepsilon)} \uar["{-\cup \id_{\Xi^\varepsilon}}"] &
        {X(\Upsilon^\varepsilon \subset M)} \lar["\simeq"'] \rar \uar[hook] &
        {\Cthick^\Lambda_\rho(P_1,\dots,P_n)}.
    \end{tikzcd}\]    
    Here the right horizontal maps are the inclusions of the fibres of the top map in \cref{prop:rmC-pullback}.
    The middle map takes the union with the missing thickened spheres $\Xi^\varepsilon$.
    (The right square commutes on-the-nose by inspection.)
    Correspondingly, the left vertical map is induced by the group homomorphism $\Diff_\partial(\Upsilon^\varepsilon) \to \Diff_\partial(\Sigma^\varepsilon)$ that extends a diffeomorphism by the identity.
    Recall that we have equivalences $\BDiff_\partial(\Sigma^\varepsilon) \simeq \Map^{C_2}(H_\Gamma, \SO(3))$.
    In particular, we have a canonical identification between $H^3(\BDiff_\partial(\Sigma^\varepsilon))$ and the cellular $1$-chains $C_1(\Gamma)$.
    Because $\sigma$ consists only of vertices $C_1(\Gamma) = C_1(\Gamma, \sigma)$.
    Taking third cohomology groups we thus get a commutative square
    \[\begin{tikzcd}
        {C_1(\Gamma,\sigma)} \rar["{\cong}"] \dar["f_*"'] &
        {H^3(\BDiff_\partial(\Sigma^\varepsilon))} \dar &
        {H^3(\rmC^\Gamma_\sigma(P_1,\dots,P_n))} \ar[l] \\
        {C_1(\Lambda,\rho)} \rar["{\cong}"] & 
        {H^3(\BDiff_\partial(\Upsilon^\varepsilon))} &
        {H^3(\rmC^\Lambda_\rho(P_1,\dots,P_n))}
        \ar[l] \uar["{\rmC(f)^*}"'] 
    \end{tikzcd}\]    
    where the left map is the map $f\colon \Gamma \to \Lambda$ induces on relative cellular $1$-chains. 
    (Namely, it sends all the edges collapsed by $f$, which correspond to spheres in $\Xi$, to $0$, and is an isomorphism otherwise.)
    Finally,
    the image of the right vertical map
    $H^3(\rmC^\Gamma_\sigma(P_1,\dots,P_n)) \to C_1(\Gamma, \sigma)$,
    is exactly the image of the cohomology of the total space in the cohomology of the fibre, for the fibre sequence for the fibre sequence resulting from the top map in \cref{prop:rmC-pullback}.
    In the Serre spectral sequence this image corresponds to $E_\infty^{0,3} \subseteq E_2^{0,3} = C_1(\Gamma, \sigma)$, which by the proof of \cref{thm:DGamma Cohomology Ring}, is exactly $H_1(\Gamma,\sigma) \subseteq C_1(\Gamma, \sigma)$.
    Hence, the square restricts to the desired commutative square.
\end{proof}

\subsection{The spectral sequence}\label{subsec:spectral-sequence}
Applying the Bousfield--Kan spectral sequence for the cohomology of a homotopy colimit \cite[XII.5.8]{BousfieldKan1972} to the space of handle-attachments from \cref{defn:calH_g} yields the following.

\begin{cor}\label{cor:BK-spectral-sequence}
    There is a convergent spectral sequence in cohomological Serre grading with
    \begin{equation}\label{eq:BK-spectral-sequence}
        E_2^{p,q} = \operatornamewithlimits{lim^{\it p}}_{\Gamma \in \Gr_{g,n}^{\rm op}} H^q(\rmC^\Gamma_\sigma(P_1,\dots,P_n); \bbQ) \Rightarrow H^{p+q}( \calH_g(P_1,\dots, P_n) )
    \end{equation}
    and this spectral sequence is $\prod_{i=1}^n \pi_0 \Diff(P_i)$ equivariant.
\end{cor}

By \cref{cor:fibre-sequence-H_g} the space of handle-attachments is the homotopy fibre of the splitting map $\W$, so there is a Serre spectral sequence
    \[
        E_2^{p,q} = H^p(\BDiff(P_M); H^q(\calH_g(P_1,\dots,P_n); \bbQ)) \Rightarrow H^{p+q}(\BDiff(M);\bbQ)
    \]
whose $E_2$ exactly is cohomology with coefficients in the cohomology of $\calH_g(\dots)$, which is described by \cref{cor:BK-spectral-sequence}.
(If we assume for simplicity that none of the two $P_i$ are diffeomorphic, then $\pi_1 \BDiff(P_M)$ is just $\prod_{i=1}^n \pi_0 \Diff(P_i)$, so \cref{cor:BK-spectral-sequence} also describes the action on the coefficients.)

\begin{ex}
    For $n=0$, that is, in the case of $M = U_g$, the target of the splitting map is a point and \cref{eq:BK-spectral-sequence} already computes the cohomology of $\BDiff(U_g)$.
    In this case \cref{lem:cohomology-functoriality} and \cref{lem:cohomology-functoriality-beta} already completely describe the cohomology of $\rmC^\Gamma$ as a functor of $\Gamma \in \Gr_{g,0}$, so the question of determining the $E_2$-page in \cref{eq:BK-spectral-sequence} is a purely algebraic one.
    However, this is still difficult for larger $g$:
    for example, the $0$th row is the derived limit of the constant functor, which gives the cohomology of the category $\Gr_{g,0}$.
    We have
    \[
        E_2^{p,0} = \operatornamewithlimits{lim^{\it p}}_{\Gamma \in \Gr_{g,n}^{\rm op}} H^0(\dots) 
        \cong H^p(B\Gr_{g,0}) \cong H^p(B\Out(F_g))
    \]
    where the last isomorphism follows from the contractibility of outer space \cite{CullerVogtmann1986}.
    These groups become very hard to compute for $g\gg 0$, and hence we should not expect a simple description of the $E_2$ for larger genus.
\end{ex}

\begin{rem}\label{rem:graph-complex}
    In \cite{Giansiracusa11}, Giansiracusa gives a homotopy colimit formula for $\BDiff^+(K_g)$ where $K_g = (S^1 \times D^2)^{\natural g}$ is the $3$-dimensional handlebody of genus $g$.
    In his setting the spectral sequence must collapse on the $E_2$-page \cite[\S7]{Giansiracusa11} because the diagram on $\Gr_{g,0}$ that he uses is built from the cyclic operad of framed $2$-discs, which is known to be formal.
    As a consequence, he can express the rational homology of $\BDiff^+(K_g)$ as the homology of a graph complex.
    We do not expect a similar result to hold in our setting, as the analogous framed $3$-discs operad is not known to be formal as a cyclic $\infty$-operad.
    (However, it was recently shown to be formal as an $\infty$-operad \cite{BCG-equivariant-formality}.)
\end{rem}

\section{Extended Example: \texorpdfstring{$\BDiff^+\!\left((S^1\times S^2)^{\sharp 2}\right)$}{BDiff+((S1 x S2) sharp (S1 x S2))}}\label{sec:extended example}
In this section, we combine the results of the Section \ref{sec: computation} with our main theorem to compute the rational cohomology ring of $\BDiff^+\left(U_2\right)$, where $U_2=(S^1\times S^2)^{\sharp 2}$.
Note that in this case, since $n=0$, the target of the splitting map is contractible and thus main theorem \cref{thm:fibre-sequence-intro} reduces to the homotopy colimit formula \cref{thm:hocolim-intro}.
We further have $\Gr(U_2) = \Gr_{2,0}$ and for any dual graph $\Gamma\in \Gr(U_2)$ the group $\Aut_{\Gr(U_2)}(\Gamma)$ is the usual group of graph symmetries $\Aut(\Gamma)$. 

By \cref{lem:No-Redudant-Htpy-Final}, we can compute the homotopy colimit over the subcategory $\Gr^{\rm nr}_{2,0}$ consisting of graphs without redundant edges. 
For genus 2, there are exactly two homeomorphism types of graph without redundant edges, namely $\Gamma=R_2$, the wedge of two circles (rose with two petals), and $\Gamma=\Theta$, the graph with two vertices connected by three edges. Moreover, there is only one kind of non-invertible morphism up to symmetry, namely the map $\pi\colon \Theta\rightarrow R_2$ that collapses the middle edge of $\Theta$. 
Therefore, the homotopy colimit can be written as a pushout of homotopy orbit constructions, see \cref{ex:hocolim-one-morphism-EI}.

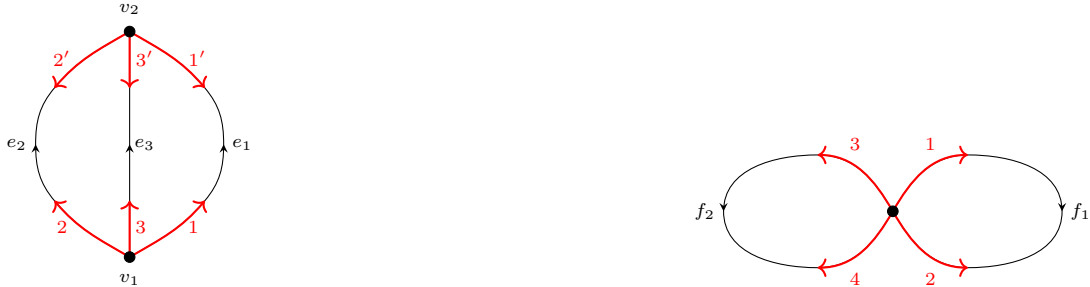
\begin{figure}[h]
    \begin{subfigure}{0.4\textwidth}
     \begin{tikzpicture}
     [decoration={markings, 
    mark= at position 0.5 with {\arrow{stealth}}}] 
     \draw[postaction={decorate}] (0,0) to (0,3);
     \draw[postaction={decorate}] (0,0) to[out=30,in=230] (1,.75) to[out=50,in=270](1.25,1.5) to[out=90,in=-50](1,2.25) to[out=130,in=-30](0,3);
     \draw[postaction={decorate}] (0,0) to[out=150,in=-50] (-1,.75) to[out=130,in=270](-1.25,1.5) to[out=90,in=230](-1,2.25) to[out=50,in=210](0,3);
     
     \draw[thick, red, ->](0,0) to (0,.75);
     \draw[thick, red, ->] (0,0) to[out=30,in=230] (1,.75);
     \draw[thick, red, ->] (0,0) to[out=150,in=-50] (-1,.75);

     \draw[thick, red, ->](0,3) to (0,2.25);
     \draw[thick, red, ->](0,3)[out=-30, in=130] to (1,2.25);
     \draw[thick, red, ->](0,3)[out=210, in=50] to (-1,2.25);
     
     \filldraw (0,0) circle (2pt);
     \filldraw (0,3) circle (2pt);

     \node at (0,-.3) {$\scriptstyle v_1$};
     \node at (0,3.3) {$\scriptstyle v_2$};
     \node at (1.5,1.5){$\scriptstyle e_1$};
     \node at (-1.5,1.5){$\scriptstyle e_2$};
     \node at (0.2,1.5){$\scriptstyle e_3$};

     \node at (.85,.4) {\textcolor{red}{$\scriptstyle 1$}};
     \node at (.15,.4) {\textcolor{red}{$\scriptstyle 3$}};
     \node at (-.9,.4) {\textcolor{red}{$\scriptstyle 2$}};
     \node at (.9,2.65) {\textcolor{red}{$\scriptstyle 1'$}};
     \node at (.2,2.65) {\textcolor{red}{$\scriptstyle 3'$}};
     \node at (-.9,2.65) {\textcolor{red}{$\scriptstyle 2'$}};
    \end{tikzpicture}
    \end{subfigure}
    \hfill
    \begin{subfigure}{0.4\textwidth}
    \begin{tikzpicture}
    [decoration={markings, 
    mark= at position 0.5 with {\arrow{stealth}}}] 
    \draw[postaction={decorate}] (0,0) to[out=60, in=180](1,.75) to[out=0, in=90](2.25,0) to[out=-90,in=0](1,-.75) to[out=-180,in=-60](0,0);
    \draw[postaction={decorate}] (0,0) to[out=120, in=0](-1,.75) to[out=180,in=90](-2.25,0) to[out=-90,in=180](-1,-.75) to[out=0,in=-120](0,0);

    \draw[thick, red, ->] (0,0) to[out=60, in=180](1,.75);
    \draw[thick, red, ->] (0,0) to[out=120, in=0](-1,.75);
    \draw[thick, red, ->] (0,0) to[out=-60,in=-180](1,-.75);
    \draw[thick, red, ->] (0,0) to[out=-120,in=0](-1,-.75);

    \filldraw (0,0) circle (2pt);
    \node at (2.5,0){$\scriptstyle f_1$};
    \node at (-2.5,0){$\scriptstyle f_2$};
    \node at (.5,.9) {\textcolor{red}{$\scriptstyle 1$}};
    \node at (.5,-.9) {\textcolor{red}{$\scriptstyle 2$}};
    \node at (-.5,.9) {\textcolor{red}{$\scriptstyle 3$}};
    \node at (-.5,-.9) {\textcolor{red}{$\scriptstyle 4$}};
    \end{tikzpicture}
    \end{subfigure}
    \caption{Notation for $\Theta$, pictured on the left, and $R_2$, on the right.}
    \label{fig:Graph Notation}
\end{figure}

First we establish notation to refer to the edges and half-edges of these graphs, and compute there automorphism groups. We refer to Figure \ref{fig:Graph Notation} for a summary. When $\Gamma=R_2,$ there are two edges $f_1$ and $f_2$ both of which are loops. Label the half edges $h_1$, $h_2$, $h_3$, $h_4$ so that $h_1=h_2^\dagger$ belong to $f_1$ while $h_3=h_4^\dagger$ belong to $f_2$.  When $\Gamma=\Theta,$ there are two vertices $v_1$, $v_2$ connected by three edges $e_1$, $e_2$ and $e_3$, oriented from $v_1$ to $v_2$.  In this case we label the half edge of $e_i$ at $v_1$ by $h_i$, and the half-edge at $v_2$ by $h_i'=h_i^\dagger$.
Let $\pi\colon \Theta \to R_2$ be the graph morphism that collapses $e_3$.
Concretely, we define $\pi$ on vertices and half-edges by
\begin{align*}
    v_1 &\mapsto v &&& h_1 &\mapsto h_2 &  h_2 &\mapsto h_4 &  h_3 &\mapsto v \\
    v_2 &\mapsto v &&& h_1' &\mapsto h_1 &  h_2' &\mapsto h_3 &  h_3' &\mapsto v.
\end{align*}

We have that $H_1(\Theta)=\langle \beta_1, \beta_2\rangle$
where $\beta_1=[e_3-e_1]$ and $\beta_2=[e_3-e_2]$. 
The map $\pi$ is a homotopy equivalence and identifies $H_1(\Theta)$ with $H_1(R_2)$.  
The orientation of $f_1$ and $f_2$ are chosen so that $\pi^*$ by sends $\beta_1=[e_3-e_1]$ to $[f_1]$ and $\beta_2=[e_3-e_2]$ to $[f_2]$. 
By an abuse of notation we will also write $H_1(R_2)=\langle \beta_1,\beta_2\rangle $ where $\beta_1=[f_1]$ and $\beta_2=[f_2].$

The automorphism group of $R_2$ is isomorphic to $D_8$, the dihedral group of order $8$. We regard this as a subgroup of the symmetric group $\Sym_4$ via its action on the four half-edges $h_1$, $h_2$, $h_3$, and $h_4$.  Since $h_1, h_2$ belong to the edge $f_1$ and $h_3, h_4$ belong to the edge $f_2$, we are identifying $D_8$ with the subgroup of $\Sym_4$ generated by the two elements $\{(12),(13)(24)\}$. 

The automorphism group of $\Theta$ is isomorphic to $\Sym_3\times C_2$. Here the symmetric group $\Sym_3$ (written in cycle notation) acts by permuting the edges, while $C_2\cong \{\pm 1\}$ acts by exchanging the two vertices, and reversing each of the edges. The subgroup of automorphisms that commute with the collapse map $\pi\colon \Theta\rightarrow R_2$ are those which fix $e_3$ setwise.  This subgroup is therefore isomorphic to $C_2\times C_2$, where the first factor is the subgroup $\langle (12)\rangle\leq \Sym_3$ which swaps edges $e_1$ and $e_2$, and the second factor inverts each edge.

Since there is one morphism up to isomorphism in $\Gr_{2,0}^{\rm nr}$, S\l ominska's formula reduces to a homotopy pushout square as in \cref{ex:hocolim-one-morphism-EI}.
Together with the automorphism group calculations above, we obtain:

\begin{prop}\label{prop:U_2-pushout}Let $U_2=(S^1\times S^2)^{\sharp 2}$. Then there exists a homotopy pushout square
    \begin{equation}\label{Eqn:U_2-pushout-square}
     \begin{tikzcd}
     \rmD(\Theta)\hq (C_2\times C_2)&\rmD(\Theta)\hq (\Sym_3\times C_2)\\
     \rmD(R_2)\hq D_8&\BDiff^+(U_2).
     \arrow["\pi",from=1-1,to=2-1]
     \arrow["\iota",from=1-1,to=1-2]
     \arrow[from=1-2,to=2-2]
     \arrow[from=2-1,to=2-2]
     \end{tikzcd}\end{equation}
\end{prop}

Applying rational cohomology, this means $H^*(\BDiff^+(U_2))$ can be computed via a Mayer--Vietoris sequence from the three terms $H^*(\rmD(\Theta))^{\Sym_3\times C_2}$, $H^*(\rmD(\Theta))^{C_2\times C_2}$, and $H^*(\rmD(R_2))^{D_8}$, where the superscripts denote the invariant part of the cohomology. 

As a consequence of Lemma \ref{lem:pi isomorphism} below we will see that the map $\pi$ in the above diagram is an isomorphism on rational cohomology. This implies that the right hand vertical map is also  an isomorphism.
It follows from \cref{eqn:Aut_Gr(M)-sequence} and \cref{cor:rmD-classifying-space} that $\rmD(\Theta)\hq \Aut(\Theta)$ is equivalent to $\BDiff^+(U_2,\Sigma)$ where $\Sigma\subset U_2$ is any separating system whose dual graph is $\Theta$.  We therefore we conclude:
\begin{cor}
    Let $\Sigma\subset U_2$ be any separating system such that $G_\Sigma\cong \Theta$. Then the induced map on classifying spaces
    \[
        \BDiff^+(U_2,\Sigma) 
        \hookrightarrow
        \BDiff^+( U_2 )
    \]
    is a rational homology isomorphism.
\end{cor}

\subsection{Presentations of $H^*(\rmD(\Gamma))$}\label{subsec:presentations}
We use Theorem \ref{thm:DGamma Cohomology Ring} to find presentations for $H^*(\rmD(\Gamma))$ for $\Gamma=\Theta$ and $\Gamma=R_2$. Since $n=0$, we recall that for every triple of distinct, adjacent half-edges $(i,j,k)$ there is a class $c_{ij}^k\in H^2(\rmD(\Gamma))$, as well as a unique class $\delta\in H^4(\rmD(\Gamma))$ such that  $(c_{ij}^k)^2=\delta$ for every $(i,j,k)$. We also have $\Aut(\Gamma)$-equivariant isomorphisms $H^3(\rmD(\Gamma))\cong H^1(\Gamma)$ for both $\Gamma=\Theta$ and $\Gamma=R_2.$
It will be useful to have the following notation.
\begin{defn}
    For two non-negatively graded $\bbQ$-algebras $R,Q$ with $R_0 = \bbQ\langle 1 \rangle = Q$, we let $R \vee Q \subset R \times Q$ denote the subring with 
    \[
        (R \vee Q)_n = \begin{cases}
            \bbQ\langle 1 \rangle & \text{ if } n=0, \\
            R_n \oplus Q_n & \text{ if } n>0.
        \end{cases}
    \]
    This is chosen such that for connected spaces $X, Y$ we have $H^*(X \vee Y) \cong H^*(X) \vee H^*(Y)$.
\end{defn}

    $\mathbf{\Gamma=\Theta:}$ The vertices $v_1,v_2$ contribute degree 2 classes  $c_1=c_{12}^3$ and $c_2=c_{1'2'}^{3'}$, respectively, where we recall that  swapping any two indices changes the sign.  Moreover we have that $c_1^2=c_2^2=\delta$. Thus \begin{align*}
    H^*(\rmD(\Theta))&=\Lambda_\bbQ\langle H_1(\Theta)\rangle\otimes \bbQ[c_1,c_2]/\langle c_1^2=c_2^2\rangle\\
    &=\Lambda_\bbQ\langle \beta_1,\beta_2\rangle\otimes \bbQ[c_1,c_2]/\langle c_1^2=c_2^2\rangle
    \end{align*}
    where the $\beta_i$ have degree 3. In order to decompose the $\Sym_3\times C_2$ action, it will be convenient to introduce the following change of basis.  Let $u_1=(c_1+c_2)/2$ and $u_2=(c_1-c_2)/2$.  Then $u_1u_2=(c_1^2-c_2^2)/4=0$.  
    Hence we can also write
    \[H^*(\rmD(\Theta))=\Lambda_\bbQ\langle \beta_1,\beta_2\rangle\otimes\left(\bbQ[u_1]\vee \bbQ[u_2]\right). \]
    Since $u_1, u_2$ each have degree 2, it follows that in degree $2n$ we have\[\left(\bbQ[u_1]\vee \bbQ[u_2]\right)_{(2n)}=\langle u_1^n,u_2^n\rangle\] and is zero in all odd degrees.

$\mathbf{\Gamma=R_2:}$ In this case there is only one vertex but four half edges so we have classes $c_{ij}^k$, with $1\leq i,j,k\leq 4$ pairwise distinct. We recall that $c_{ij}^k=-c_{ji}^k=c_{jk}^i$, $(c_{ij}^k)^2=\delta$ and $c_{12}^3=c_{12}^4+ c_{23}^4+c_{31}^4$. The ideal generated by these relations is $I_\delta$ and we obtain
    \begin{align*}
    H^*(\rmD(R_2))&=\Lambda_\bbQ\langle H_1(R_2)\rangle\otimes \bbQ[c_{ij}^k, \delta]/I_\delta\\
    &=\Lambda_\bbQ\langle\beta_1,\beta_2\rangle\otimes \bbQ[c_{ij}^k, \delta]/I_\delta.
    \end{align*}
    As above we seek a basis adapted to the $D_8$-action. We rewrite the relation $c_{12}^3=c_{12}^4-c_{13}^4+ c_{23}^4$ as \[c_{12}^3-c_{12}^4= c_{34}^2-c_{34}^1\]
    Let $w=\frac{c_{12}^3-c_{12}^4}{2}= \frac{c_{34}^2-c_{34}^1}{2}$  and define \begin{align*}
        z_1&=\frac{c_{12}^3+c_{12}^4}{2}, ~~
        z_2=\frac{c_{34}^2+c_{34}^1}{2}.
    \end{align*}
    Thus $\{w, z_1,z_2\}$ are linearly independent and generate $\bbQ[c_{ij}^k,\delta]/I_\delta$ as a ring.  Now observe that 
    \begin{align*}
        z_1w&=\frac{(c_{12}^3)^2-(c_{12}^4)^2}{4}=\frac{\delta-\delta}{4}=0\\
        z_2w&=\frac{(c_{34}^2)^2-(c_{34}^1)^2}{4}=\frac{\delta-\delta}{4}=0.
    \end{align*}
    Moreover, we have \[\frac{\delta-c_{12}^3c_{12}^4}{2}= w^2 = \frac{\delta-c_{34}^2c_{34}^1}{2}\] whence $c_{12}^3c_{12}^4=c_{34}^2c_{34}^1$. This implies that \[z_1^2=\frac{\delta+c_{12}^3c_{12}^4}{2}=\frac{\delta+c_{34}^2c_{34}^1}{2}=z_2^2.\]
    Thus, every element of degree $2n$ can be written as a linear combination of $\{w^n,z_1^n,z_1^{n-1}z_2\}$.
    By induction it follows that for all $n\geq 1$ we have:
    \[w^n=\left\{\begin{array}{cl} \delta^{\frac{n-1}{2}}w& n \text{ odd}\\ \delta^{\frac{n-2}{2}}w^2,&n \text{ even}\end{array}\right.\qquad z_1^n=\left\{\begin{array}{cl} \delta^{\frac{n-1}{2}}z_1& n \text{ odd}\\ \delta^{\frac{n-2}{2}}z_1^2,&n \text{ even}\end{array}\right. 
    \qquad z_1^{n-1}z_2=\left\{\begin{array}{cl} \delta^{\frac{n-1}{2}}z_2& n \text{ odd}\\ \delta^{\frac{n-2}{2}}z_1z_2,&n \text{ even}\end{array}\right.\]
   
   It is not difficult to verify that $\{w^2,z_1^2,z_1z_2\}$ is linearly independent when $n=2$, and since $\delta$ is not a zero divisor by \cref{lem:delta-cancellable}, we conclude that $\{w^n, z_1^n,z_1^{n-1}z_2\}$  is linearly independent for every $n$. With respect to the generating set $\{w,z_1,z_2\}$ we thus obtain a presentation
    \[ H^*(\rmD(R_2))=\Lambda_\bbQ\langle\beta_1,\beta_2\rangle\otimes\left(\bbQ[w]\vee \bbQ[z_1,z_2]/\langle z_1^2=z_2^2\rangle\right).\]
Since $w$, $z_1$, and $z_2$ have degree 2, the degree $2n$ part of the righthand factor is \[\left(\bbQ[w]\vee\bbQ[z_1,z_2]/\langle z_1^2=z_2^2\rangle\right)_{(2n)}=\langle w^n,z_1^n, z_1^{n-1}z_2\rangle\]
and is zero in odd degrees.

\subsection{Computing the invariants}
We now decompose $H^*(\rmD(\Theta))$ and $H^*(\rmD(R_2))$ into irreducible representations, with respect to the actions of $\Sym_3\times C_2$ and $C_2\times C_2$ on the former, and $D_8$ on the latter. 
All notation for the irreducible characters of these groups and their representation rings is contained in Figure \ref{fig:Representation Rings}.

\begin{figure}[h]
    \centering
    \begin{subfigure}[b]{\textwidth}
    \captionsetup{singlelinecheck=off}
    \centering
    $\Sym_3\times C_2$
    \vspace{1em}
    
    $\begin{array}{c|cccccc}
   \text{class} & 1 & (12) & (123) & -1 & -(12) & -(123)\\\hline
    \text{size} & 1 & 3 & 2 & 1 & 3 & 2\\\hline
    \mathds{1} & 1 & 1 & 1 & 1 & 1 & 1\\
    \sigma & 1 & -1 & 1 & 1 & -1 & 1\\
    \tau_3 & 2 & 0 & -1 & 2 & 0 & -1\\
    \rho & 1 & 1 & 1 & -1 & -1 & -1\\
    \sigma\rho & 1 & -1 & 1 & -1 & 1 & -1\\
    \tau_3\rho & 2 & 0 & -1 & -2 & 0 & 1\\
    \end{array}$
    \subcaption{The character table for $\Sym_3\times C_2$, with $\Sym_3$ in cycle notation and $C_2=\{\pm 1\}$.  
    \begin{itemize}
        \item Relations: $\sigma^2=\rho^2=\mathds{1}$, $\tau_3\sigma=\tau_3$, $\tau_3^2=\mathds{1}+\sigma+\tau_3.$
        \item Exterior powers: $\Lambda^2(\tau_3)=\sigma$.
    \end{itemize}}
    \end{subfigure}
\vspace{1em}

   \begin{subfigure}[b]{\textwidth}
   \captionsetup{singlelinecheck=off}
   \centering
   $C_2\times C_2$
   \vspace{1em}
   
   $\begin{array}{c|cccc}
  \text{class} & 1 & (12) & -1 & -(12) \\\hline
   \text{size} & 1 & 1 & 1 & 1 \\\hline
   \mathds{1} & 1 & 1 & 1 & 1 \\
   \kappa_1 & 1 & -1 & 1 & -1\\
   \kappa_2 & 1 & 1 & -1 & -1 \\
   \kappa_3 & 1 & -1 & -1 & 1 \\
   \end{array}$
   \subcaption{The character table for $C_2\times C_2$, regarded as $\langle(12)\rangle\times \{\pm 1\}$. 
   \begin{itemize}
       \item Relations: $\kappa_1^2=\kappa_2^2=1$, $\kappa_1\kappa_2=\kappa_3$.
       \item Exterior powers: none.
   \end{itemize}}
   \end{subfigure}
   \vspace{1em} 
   
   \begin{subfigure}[b]{\textwidth}
   \captionsetup{singlelinecheck=off}
   \centering
   $D_8$
   \vspace{1em}

   $\begin{array}{c|ccccc}
  \text{class} & 1 & (12)(34) & (12) & (13)(24) & (1324)\\\hline
   \text{size} & 1 & 1 & 2 & 2 &  2\\\hline
   \mathds{1} & 1 & 1 & 1 & 1 &  1\\
   \lambda_1 & 1 & 1 & -1 & 1 &  -1\\
   \lambda_2 & 1 & 1 & 1 & -1 &  -1\\
   \lambda_3 & 1 & 1 & -1 & -1 & 1\\
   \tau_4 & 2 & -2 & 0 & 0 & 0\\
   \end{array}$
   \subcaption{The character table for $D_8$ as a subgroup of $\Sym_4$, generated by $(12)$ and $(34)$.  
   \begin{itemize}
       \item Relations: $\lambda_1^2=\lambda_2^2=1$, $\lambda_1\lambda_2=\lambda_3$, $\tau_4\lambda_1=\tau_4\lambda_2=\tau_4$, $\tau_4^2=\mathds{1}+\lambda_1+\lambda_2+\lambda_3.$
       \item Exterior powers: $\Lambda^2(\tau_4)=\lambda_3$.
   \end{itemize}}
   \end{subfigure}

    \caption{Character tables and representation rings of $\Sym_3\times C_2$, $C_2\times C_2$, and $D_8$.}
    \label{fig:Representation Rings}
\end{figure}

\begin{lem}\label{lem:Sigma3C2 Invariants}As a graded commutative ring, the invariants for the action $\Sym_3\times C_2$ on $H^*(\rmD(\Theta))$ has presentation \[ H^*(\rmD(\Theta))^{\Sym_3\times C_2}=\bbQ[\gamma_1,\varepsilon_1]/\langle \varepsilon_1^2\rangle\vee \bbQ[\gamma_2]\]
where $\gamma_1=u_1^2$, $\gamma_2=u_2^2$ have degree 4, and $\varepsilon_1=\beta_1\beta_2u_1$ has degree 8. In particular, the Betti numbers are $b_4=2$ and $b_{4k}=3$ for all $k\geq 2.$

\end{lem}
\begin{proof}
    $\Sym_3$ acts by permuting the edges $e_i$, while $C_2$ acts by inverting them.  Therefore the 1-chains in $\Theta$ form a regular representation of $\Sym_3$ which restricts to the 2-dimensional subspace of cycles (the orthogonal complement of $e_1+e_2+e_3$) as the standard representation $\tau_3$.  
    
    The $C_2$-action which inverts each edge acts by $-1$ on $H_1(\Theta) \cong \langle\beta_1,\beta_2\rangle$; 
    the corresponding representation of $\Sym_3\times C_2$ on $\Lambda^1\langle \beta_1,\beta_2\rangle$ is therefore  $\tau_3\rho$, and on $\Lambda^2\langle \beta_1,\beta_2\rangle =\langle \beta_1\beta_2\rangle=\Lambda^2(\tau_3\rho)=\sigma$.
    Each permutation in $\Sym_3$ acts on $c_i$ by the sign representation $\sigma$. 
    On the other hand, the $C_2$ exchanges the two $\Conf_3(S^3)$ factors, mapping $c_1$ to $c_2$ and vice versa. The subspaces spanned by $u_1$ and $u_2$ are preserved so the representation on $\langle c_1,c_2\rangle =\langle u_1,u_2\rangle$ decomposes as two 1-dimensional representations $\langle u_1\rangle \oplus \langle u_2\rangle=\sigma+\sigma\rho$. Each of these 1-dimensional representations squares to the trivial representation $\mathds{1}$. Summarising, we have
\begin{align*}
    \Lambda^1\langle \beta_1,\beta_2\rangle & =\tau_3\rho, &
    \Lambda^2\langle\beta_1,\beta_2\rangle &=\langle \beta_1\beta_2\rangle=\sigma &
    \langle u_1^n,u_2^n\rangle & =\left\{\begin{array}{ll}
     \sigma+\sigma\rho, &n\geq 1 \text{ odd}\\
      2\cdot \mathds{1}, &n\geq 2 \text{ even}
    \end{array}\right.
\end{align*}
From the tensor product description of $H^*(\rmD(\Theta))$ we then have:
\begin{align*}
    \tau_3\rho\otimes(\sigma+\sigma\rho)&=\tau_3\rho+\tau_3 &
    \tau_3\rho\otimes (2\cdot \mathds{1})&=2\cdot \tau_3\rho \\
    \sigma\otimes (\sigma+\sigma \rho)&=\mathds{1}+\rho &
    \sigma\otimes (2\cdot \mathds{1})&=2\cdot \sigma
\end{align*}
The trivial summand in $\sigma \otimes (\sigma + \sigma \rho)$ comes from $ \beta_1\beta_2 \otimes u_1$.  Since $|\beta_i|=3$ and $u_1u_2=0$, setting $\gamma_1=u_1^2$, $\gamma_2=u_2^2,$ and $\varepsilon_1=\beta_1\beta_2u_1$ we have
\[H^*(\rmD(\Theta))^{\Sym_3\times C_2}=\bbQ[\gamma_1,\varepsilon_1]/\langle\varepsilon_1^2\rangle\vee \bbQ[\gamma_2].\]
 Since $|\gamma_1|=|\gamma_2|=4$ and $|\varepsilon_1|=8$, the degree 4 part is 2-dimensional and spanned by $\{\gamma_1,\gamma_2\}$ while for all $k\geq 2$ degree $4k$ part is 3-dimensional and spanned by  spanned by $\{\gamma_1^k, \gamma_1^{k-2}\varepsilon_1,\gamma_2^k\}$ which proves that the Betti numbers are as claimed.
\end{proof}

\begin{lem}\label{lem:C2C2 Invariants}
    As a graded commutative ring, the invariants for the action of $C_2\times C_2$ on $H^*(\rmD(\Theta))$ has presentation
    \[H^*(\rmD(\Theta))^{C_2\times C_2}=\bbQ[\gamma_1,\varepsilon_1]/\langle\varepsilon_1^2\rangle\vee \bbQ[\gamma_2,\mu_2]/\langle\mu_2^2 \rangle\]
    where $\gamma_1=u_1^2$ and $\gamma_2=u_2^2$ have degree 4, $\varepsilon_1=\beta_1\beta_2u_1$ has degree 8, and $\mu_2=(\beta_1-\beta_2)u_2$ has degree 5.
    The Betti numbers are $b_4=2$, $b_{4k+1}=1$ for all $k\geq 1$ and $b_{4k}=3$ for all $k\geq 2.$
\end{lem}
\begin{proof}
    The stabiliser of the edge $e_3$ in $\Theta$ is the subgroup generated by the permutation $(12)$ which exchanges $e_1$ and $e_2$ together with the involution that flips all of the edges.  Then (12) switches $e_1-e_3$ with $e_2-e_3$, while the involution acts by $-1$ on both.  It follows that $\langle \beta_1,\beta_2\rangle$ decomposes as a sum of two 1-dimensional representations $\langle\beta_1+\beta_2\rangle\oplus \langle\beta_1-\beta_2\rangle=\kappa_2+\kappa_3$, and $\langle \beta_1\beta_2\rangle=\kappa_2\cdot\kappa_3=\kappa_1.$ Similarly, $(12)$ acts as $-1$ on both $c_1$ and $c_2$, while the involution that exchanges both sides also exchanges $c_1$ and $c_2$.  It follows that $\langle c_1,c_2\rangle$ decomposes as $\langle u_1\rangle \oplus \langle u_2\rangle=\kappa_1+\kappa_3$. As each of these are 1-dimensional, we obtain
    \[\langle u_1^n,u_2^n\rangle=\left\{\begin{array}{cl}\kappa_1+\kappa_2,&\text{ $n\geq 1$ odd}\\ 2\cdot\mathds{1},&\text{ $n\geq 2$ even.}\end{array}\right.\]

    Considering tensor products we have
    \begin{align*}
        (\kappa_2+\kappa_3)\otimes (\kappa_1+\kappa_3) & =\kappa_3+\kappa_2+\kappa_1+\mathds{1} &
        (\kappa_2+\kappa_3)\otimes (2\cdot\mathds{1}) & =2\cdot\kappa_2+2\cdot \kappa_3 \\
        \kappa_1\otimes (\kappa_1+\kappa_3) & =\mathds{1}+\kappa_2 &
        \kappa_1\otimes (2\cdot \mathds{1}) & =2\cdot \kappa_1 .
    \end{align*}
The trivial summand in $(\kappa_2 + \kappa_3) \otimes (\kappa_1 + \kappa_3)$ is the 5-dimensional class $(\beta_1-\beta_2)\otimes u_2$, while the trivial summand in $\kappa_1 \otimes (\kappa_1 + \kappa_3)$ is the class $\beta_1\beta_2 u_1=\varepsilon_1$ found previously.
The other generators are $\gamma_1=u_1^2$ and $\gamma_2=u_2^2$. 
Defining $\mu_2=(\beta_1-\beta_2)u_2$, since $u_1u_2=0$ and $\mu_2^2 =0$ (because it has odd degree), we obtain the presentation
 \[H^*(\rmD(\Theta))^{C_2\times C_2}=\bbQ[\gamma_1,\varepsilon_1]/\langle\varepsilon_1^2\rangle\vee \bbQ[\gamma_2,\mu_2]/\langle \mu_2^2 \rangle\]
where $|\gamma_1|=|\gamma_2|=4$, $|\varepsilon_1|=8$, and $|\mu_2|=5$. For the Betti numbers, we obtain additional classes of the form $\gamma_2^{k-1}\mu_2$ in degree $(4k+1)$ for all $k\geq 1$. The other Betti numbers agree with the $\Sym_3\times C_2$ invariants.  
\end{proof}

\begin{lem}\label{lem:D8 Invariants} As a graded commutative ring, the invariants for the action of $D_8$ on $H^*(\rmD(R_2))$ has presentation 
    \[H^*(\rmD(R_2))^{D_8}=\bbQ[\alpha_1,\eta_1]/\langle\eta_1^2\rangle \vee \bbQ[\alpha_2,\nu_2]/\langle \nu_2^2\rangle\]
    where $\alpha_1=w^2$ and $\alpha_2=z_1^2=z_2^2$ have degree 4, $\eta_1=\beta_1\beta_2 w$ has degree 8, and $\nu_2=\beta_1z_1+\beta_2z_2$ has degree 5.  The Betti numbers are $b_4=2$, $b_{4k+1}=1$ for all $k\geq 1$ and $b_{4k}=3$ for all $k\geq 2.$
\end{lem}
\begin{proof} With respect to the basis $\{\beta_1,\beta_2\}$, the action of $D_8$ has the following representatives for the four nontrivial conjugacy classes:
\[\begin{array}{rcrc}
(12)\colon&\left(\begin{array}{cc}-1&0\\0&1\end{array}\right)&(13)(24)\colon&\left(\begin{array}{cc}0&1\\1&0\end{array}\right)\\
&&&\\
(12)(34)\colon&\left(\begin{array}{cc}-1&0\\0&-1\end{array}\right)&(1324)\colon&\left(\begin{array}{cc}0&-1\\1&0\end{array}\right)
\end{array}\]
which clearly corresponds to the standard 2-dimensional representation $\tau_4$. It follows that $\Lambda^2\langle \beta_1,\beta_2\rangle=\Lambda^2(\tau_4)=\lambda_3$ as a $D_8$-representation.

On $\langle w,z_1,z_2\rangle$, the $D_8$-action preserves the subspace spanned by $w$ as the 1-dimensional representation $\lambda_3$, while it acts on $\langle z_1,z_2\rangle$ according to the same matrices above ($z_i\leftrightarrow \beta_i$). Thus $\langle w, z_1, z_2\rangle=\lambda_3+\tau_4$.  It follows that $\langle w^2\rangle=\lambda_3^2=\mathds{1}$, while $\langle z_1^2=z_2^2\rangle=\mathds{1}$ and $\langle z_1z_2\rangle=\lambda_1$ (note the difference from $\langle \beta_1\beta_2\rangle$ since $z_1$ and $z_2$ commute). Hence we have 
\[\langle w^n,z_1^n, z_1^{n-1}z_2\rangle=\left\{\begin{array}{cl}\lambda_3+\tau,&\text{ $n\geq 1$ odd}\\ 2\cdot\mathds{1}+\lambda_1,&\text{ $n\geq 2$ even}\end{array}\right.\]
Incorporating tensor products with $\Lambda^1\langle\beta_1,\beta_2\rangle = \tau$ and $\Lambda^2\langle \beta_1,\beta_2 \rangle = \lambda_3$ we have
 \begin{align*}
        \tau\otimes (\lambda_3+\tau)&=\tau+(\mathds{1}+\lambda_1+\lambda_2+\lambda_3) &
        \tau\otimes (2\cdot\mathds{1}+\lambda_1)&=3\cdot \tau \\
        \lambda_3\otimes (\lambda_3+\tau)&=\mathds{1}+\tau &
        \lambda\otimes (2\cdot \mathds{1}+\lambda_1)&=2\cdot \lambda_3+\lambda_2
\end{align*}
The additional trivial summand on the first line comes from $\beta_1z_1+\beta_2z_2$ since $D_8$ acts the same way on $\beta_i$ as $z_i$.  The trivial summand on line three is $\beta_1\beta_2w$. Therefore if we write $\alpha_1=w^2$, $\eta_1=\beta_1\beta_2w$, $\alpha_2=z_1^2=z_2^2$, and $\nu_2=\beta_1z_1+\beta_2z_2$ we obtain the presentation
\[H^*(\rmD(R_2))^{D_8}=\bbQ[\alpha_1,\eta_1]/\langle\eta_1^2\rangle \vee \bbQ[\alpha_2,\nu_2] /\langle \nu_2^2 \rangle\]
since $z_1w=z_2w=0$ and the $\beta_i$ have odd degree. Here $|\alpha_i|=4,$ $|\eta_1|=8$ and $|\nu_2|=5.$ Since this ring is clearly isomorphic to $H^*(\rmD(\Theta))^{C_2\times C_2}$, the Betti numbers are the same in each degree.
\end{proof}

\subsection{Finishing the computation} We can now complete the calculation of $H^*(\BDiff^+(U_2))$, using the pushout in \cref{prop:U_2-pushout}. This is simplified considerably by the following lemma, which describes the effect of $\pi^*$ on cohomology. 
\begin{lem}\label{lem:pi isomorphism}
    The map $\pi\colon \Theta\rightarrow R_2$ induces an isomorphism of rings \[\pi^*\colon H^*(\rmD(R_2))^{D_8}\rightarrow H^*(\rmD(\Theta))^{C_2\times C_2}.\]
\end{lem}
\begin{proof} 
Recall that $\pi\colon \Theta\rightarrow R_2$ is the map which collapses the edge $e_3$.
We have set up our notation so that $\pi^*$ sends $\beta_i\in H^3(\rmD(R_2))$ to $\beta_i\in H^3(\rmD(\Theta))$, as this is exactly the map $\pi^*\colon H^1(R_2)\rightarrow H^1(\Theta)$ according to \cref{lem:cohomology-functoriality-beta}.
It remains to compute the effect of $\pi^*$ on the $c_{ij}^k$. 
For notational clarity we will write $d_1=d_{12}^3$ and $d_2=d_{1'2'}^{3'}$ for the two 2-dimensional classes of $H^*(\rmD(\Theta))$.
Thus $u_1=(d_1+d_2)/2$ and $u_2=(d_1-d_2)/2$. 

The functoriality of the classes $c_{ij}^k$ is described in \cref{lem:cohomology-functoriality} in terms of the map $\pi^\tripod$ on triples of half-edges.
Let $i,j,k \in \{1,2,3,4\}$ be distinct, corresponding to half-edges in $R_2$, each of which has a unique preimage under $\pi$.
(Recall that $\pi$ maps $h_1' \mapsto h_1$, $h_1 \mapsto h_2$, $h_2' \mapsto h_3$, and $h_2 \mapsto h_4$.)
Of the three preimages, two are at the same vertex; for these we take the preimage and for the remaining half-edge we take either $h_3$ or $h_3'$, depending on which vertex the other half-edges are at.
That is, the map $\pi^\tripod$ is
\begin{align*}
(h_1,h_2,h_3)&\mapsto (h_1', h_3', h_2') & (h_1,h_2,h_4) &\mapsto (h_3, h_1, h_2)\\
(h_3,h_4,h_2)&\mapsto (h_3,h_2,h_1)& (h_3,h_4,h_1)&\mapsto (h_2', h_3', h_1'),
\end{align*}
which on cohomology classes gives
\begin{align*}
c_{12}^3&\mapsto d_{1'3'}^{2'}=-d_2,& c_{12}^4&\mapsto d_{31}^{2}=d_1, &
c_{34}^2&\mapsto d_{32}^{1}=-d_1,& c_{34}^1&\mapsto d_{2'3'}^{1'}=d_2.
\end{align*}
Recalling now that $w=(c_{12}^3-c_{12}^4)/2$, $z_1=(c_{12}^3+c_{12}^4)/2$, and $z_2=(c_{34}^2+c_{34}^1)/2$, we see that $\pi^*(w)=(-d_2-d_1)/2=-u_1$,  $\pi^*(z_1)=(-d_2+d_1)/2=u_2$, and $\pi^*(z_2)=(-d_1+d_2)/2=-u_2$.

It follows that $\pi^*\colon H^*(\rmD(R_2))^{D_8}\rightarrow H^*(\rmD(\Theta))^{C_2\times C_2}$ is defined on generators by
\begin{align*}
    \alpha_1&\mapsto\pi^*(w^2)=(-u_1)^2=\gamma_1 &
    \eta_1&\mapsto \pi^*(\beta_1\beta_2w)=\beta_1\beta_2(-u_1)=-\varepsilon_1\\
    \alpha_2&\mapsto \pi^*(z_1^2)=\pi^*(z_2^2)=u_2^2=\gamma_2 &
    \nu_2&\mapsto \pi^*(\beta_1z_1+\beta_2z_2)=\beta_1u_2+\beta_2(-u_2)=(\beta_1-\beta_2)u_2=\mu_2
\end{align*}
which is clearly an isomorphism of rings by Lemmas \ref{lem:C2C2 Invariants} and \ref{lem:D8 Invariants}.
\end{proof}

\begin{thm}\label{thm:H*BDiffU2}
The rational cohomology ring of $\BDiff^+\!\left((S^1\times S^2)^{\sharp2}\right)$ has presentation
\[H^*\left(\BDiff^+\!\left((S^1\times S^2)^{\sharp2}\right)\right)\cong \bbQ[\gamma_1,\varepsilon]/\langle\varepsilon^2\rangle \vee \bbQ[\gamma_2]\]
where $|\gamma_1|=|\gamma_2|=4$ and $|\varepsilon|=8$.
\end{thm}
\begin{proof} Applying cohomology to the homotopy pushout in \cref{Eqn:U_2-pushout-square}, we get a commutative diagram
 \[  \begin{tikzcd}
    H^*(\rmD(\Theta))^{C_2\times C_2} & H^*(\rmD(\Theta))^{\Sym_3\times C_2} \ar[l, "{\iota^*}"']\\
    H^*(\rmD(R_2))^{D_8}\ar[u, "{\pi^*}"] &H^*\left(\BDiff^+(U_2)\right)\ar[u]\ar[l]
\end{tikzcd}\]
The maps in this diagram can be assembled into a Mayer-Vietoris-type long exact sequence.
The left vertical map $\pi^*$ is an isomorphism by Lemma \ref{lem:pi isomorphism}, hence the right  vertical map is an isomorphism as well.  By Lemma \ref{lem:Sigma3C2 Invariants}, we conclude that
\[H^*\left(\BDiff^+(U_2)\right)\cong H^*(\rmD(\Theta))^{\Sym_3\times C_2}= \bbQ[\gamma_1,\varepsilon]/\langle\varepsilon^2\rangle \vee \bbQ[\gamma_2]\]
as desired.
\end{proof}

\subsection{Circle actions}\label{subsec:circle-action}
Cohomology classes on $\BDiff^+(U_2)$ define characteristic classes for oriented smooth $U_2$ bundles.
In this section we evaluate them in two examples that distinguish $\gamma_1$ and $\gamma_2$.
Concretely, we construct two smooth circle actions $\xi_1, \xi_2\colon \SO(2) \to \Diff^+(U_2)$ on $U_2$, which yield $U_2$ bundles $\pi_i\colon E_i \to \BSO(2)$, and we calculate the pullback of $\gamma_1$ and $\gamma_2$ to the cohomology of $\BSO(2)$.

\begin{defn}
    Let $M := D^3 \setminus (D^3 \amalg D^3)$ be the manifold obtained by removing two small disjoint open balls on the first coordinate axis of $D^3 \subset \bbR^3$.
    The group $\SO(2)$ acts on $M$ by rotations around the first coordinate axis.
    We define
    \[
        N_1 := M \cup_{\amalg_3 S^2} M^- 
        \qquad \text{and} \qquad
        N_2 := M \cup_{\amalg_3 S^2}^\varphi M^-
    \]
    where in the first case we double $M$ by gluing along $\id_{\partial M}$, and in the second case we glue along the diffeomorphism $\varphi$ that swaps the inner two spheres.
    (Note that the right copy $M^-$ of $M$ must have the reversed orientation in both cases.)
    Both $N_1$ and $N_2$ are diffeomorphic to $U_2$ and after fixing identifications the smooth $\SO(2)$-actions on $N_i$ induces continuous group homomorphisms
    \[
        \xi_1 \colon \SO(2) \longrightarrow \Diff^+(N_1) \cong \Diff^+(U_2)
        \qquad \text{and} \qquad
        \xi_2 \colon \SO(2) \longrightarrow \Diff^+(N_2) \cong \Diff^+(U_2).
    \]
\end{defn}

\begin{rem}
    Note that the $\SO(2)$-actions on $N_1$ and $N_2$ are indeed qualitatively different: on $N_1$ the fixed submanifold consists of three circles, and on $N_2$ it consists of one circle.
\end{rem}

The maps on classifying spaces $B\xi_i\colon \BSO(2) \to \BDiff^+(U_2)$, classify $U_2$-bundles over $\BSO(2)$, which we can describe concretely as
\[
    E_i := (S^\infty \times N_i)/\SO(2) \longrightarrow S^\infty / \SO(2) = \CP^\infty
\]
where $S^\infty$ is the unit sphere in $\bbC^\infty$, on which $\SO(2)$ acts by complex multiplication.
Note that the restriction 
\[
    (E_i)_{|\CP^n} := (S^{2n+1} \times N_i)/\SO(2) \longrightarrow S^{2n+1}/\SO(2) = \CP^n
\]
is a smooth submersion from a closed $(2n+3)$-manifold to a $2n$-manifold with fibre $U_2$.
The proposition below evaluates the characteristic classes for $U_2$-bundles, $\gamma_1, \gamma_2, \varepsilon$ on these two bundles.

\begin{prop}
    Under the maps $B\xi_i\colon \BSO(2) \to \BDiff^+(U_2)$ the cohomology classes from \cref{thm:H*BDiffU2} restrict as 
    \begin{align*}
        B\xi_1^*(\gamma_j) &= \begin{cases}
            0 & j=1, \\
            e^2 & j=2,
        \end{cases}&
        B\xi_2^*(\gamma_j) &= \begin{cases}
            e^2  & j=1, \\
            0& j=2,
        \end{cases}
    \end{align*}
    where $e$ is the Euler class in $H^*(\BSO(2)) = \bbQ[e]$.
    The class $\varepsilon$ restricts to $0$ in both cases.
\end{prop}
\begin{proof}
    By construction $N_i$ has a separating system $\Sigma_i \subset N_i$ consisting of the three boundary spheres of $M$.
    We identify the dual graph in each case with $\Theta$ in such a way that in both cases the spheres on the the left copy of $M \sqcup M^- \cong (N_i \ca \Sigma)$ are are identified with the half-edges $\{1,2,3\}$ in the same way.
    (This means that the identification of the spheres on the right copy with $\{1',2',3'\}$ is the identity for $N_1$ and a transposition for $N_2$.)
    We let $r\colon M \cong M^-$ denote the reflection of the third coordinate axis, this is an orientation-preserving diffeomorphism as we have the opposite orientation on $M^-$, but it reverses the rotation.
    
    The rotations preserve these sphere systems and their dual graphs, so they define maps $\BSO(2) \to \BDiff_{\pi_0(2\Sigma)}^+(N_i,\Sigma_i) \simeq \rmD(\Theta)$.
    Using the notation from \cref{subsec:presentations}, we  need to determine the ring homomorphism
    \[
        H^*(\rmD(\Theta)) \cong \Lambda_\bbQ\langle \beta_1,\beta_2\rangle\otimes \bbQ[c_1,c_2]/\langle c_1^2=c_2^2\rangle
        \longrightarrow \bbQ[e] = H^*(\BSO(2)).
    \]
    For degree reasons, the $\beta_i$ go to $0$.
    It will suffice to determine the image of $c_1 = c_{12}^3$ and $c_2 = c_{1'2'}^{3'}$, from which we can read off the value of $\gamma_1 = (\tfrac{c_1 + c_2}{2})^2$ and $\gamma_2 = (\tfrac{c_1 - c_2}{2})^2$.
    These classes are defined by pulling back, along the restriction to either $M$ or $M^-$, a chosen (integral) generator of $H^2$ of the space
    \[
        \BDiff_{\pi_0(\partial M)}^+(M) \simeq \BDiff_{\pi_0(\amalg_3 S^2)}^+(S^3 \setminus \amalg_3 \interior{D^3}) \simeq \Conf_3(S^3)\sslash\SO(4) .
    \]
    Consider the four maps
    \begin{align*}
        f_i^{\rm L} &\colon \BSO(2) \xrightarrow{B\xi_i} \BDiff_{\pi_0(2\Sigma)}^+(N_i,\Sigma_i) 
        \xrightarrow{\text{left}} \BDiff^+_{\pi_0(\partial M)}(M) \\
        f_i^{\rm R} &\colon \BSO(2) \xrightarrow{B\xi_i} \BDiff_{\pi_0(2\Sigma)}^+(N_i,\Sigma_i) 
        \xrightarrow{\text{right}} \BDiff^+_{\pi_0(\partial M^-)}(M^-) \simeq \BDiff^+_{\pi_0(\partial M)}(M)
    \end{align*}
    where the second map restricts to the left or right factor.
    (In the case of $f_i^{\rm R}$ the last equivalence uses the reflection $r\colon M \cong M^-$.)
    Each of these composites is a homotopy equivalence by \cref{lem:Conf-for-n=3}, and hence they pull back $c_{1,2}^3$ (or $c_{1',2'}^{3'}$) to $\pm e$.
    We will need to work out how these signs differ.
    The maps $f_1^{\rm L}$ and $f_2^{\rm L}$ are identical, because we identified the left copy of $M\sqcup M^- = (N_i \ca \Sigma_i)$ in the same way,
    but the maps $f_1^{\rm R}$ and $f_2^{\rm R}$ differ by a the involution on $\Conf_3(S^3)\sslash \SO(4)$ that swaps two of the points.
    (This introduces a sign as $c_{2'1'}^{3'} = -c_{1'2'}^{3'}$.)
    Moreover, the maps $f_1^{\rm L}$ and $f_1^{\rm R}$ differ the reflection $r$, which reverses the $\SO(2)$-action and hence acts by $-1$ on $H^2$.
    In summary, we have
    \begin{align*}
        (f_2^{\rm L})^* c_{1'2'}^{3'} = (f_1^{\rm L})^* c_{12}^3 &= - (f_1^{\rm R})^* c_{12}^3 = (f_2^{\rm R})^* c_{1'2'}^{3'}
    \end{align*}
    and all of these classes are $\pm e$.
    It follows that 
    \begin{align*}
        B\xi_1^*(\gamma_2) &= B\xi_1^*(\tfrac{c_1 - c_2}{2})^2 = (\pm e)^2 = e^2 &
        B\xi_2^*(\gamma_1) &= B\xi_2^*(\tfrac{c_1 + c_2}{2})^2 = (\pm e)^2 = e^2
    \end{align*}
    and $B\xi_1^*(\gamma_1) = 0 = B\xi_2^*(\gamma_2)$.
    Finally, $B\xi_i^*(\varepsilon)$ must be $0$ as its square is $0$.
\end{proof}

\bibliography{mybib}{}
\bibliographystyle{alpha}

\end{document}